\documentclass{amsart}
\usepackage{amsmath}
\pdfoutput=1
\input xy

\xyoption{all}
\CompileMatrices
\usepackage{xcolor,sseq,amssymb}
\usepackage{wasysym}
\usepackage{enumerate}
\usepackage{pdflscape}
\usepackage{todonotes}
\usepackage{bbm}
\usepackage{xypic}
\usepackage{datetime}
\usepackage{accents}
\usepackage{amsthm}
\usepackage{mathrsfs}
\usepackage{tikz, tikz-cd}
\usepackage{mathtools}
\usepackage{hyperref}
\usepackage{graphicx}
\usepackage{pdfpages}
\usepackage{float}
\usepackage{caption}
\usepackage{cite}
\usepackage{tcolorbox}
\usepackage[capitalize, nameinlink]{cleveref}
\usepackage[normalem]{ulem}
\usepackage[left=3.4cm,right=3.4cm]{geometry}
\usepackage{soul}
\usepackage{afterpage}
\usepackage{lscape}
\usepackage{fancyhdr}

\newtheorem{thm}{Theorem}[section]
\newtheorem{lem}[thm]{Lemma}
\newtheorem{defi}[thm]{Definition}
\newtheorem{prop}[thm]{Proposition}
\newtheorem{exam}[thm]{Example}
\newtheorem{rmk}[thm]{Remark}
\newtheorem{workflow}[thm]{Workflow}
\newtheorem{cor}[thm]{Corollary}

\newtheorem{nota}[thm]{Notation}

\newcommand{\UZ}{\underline{\mathbb{Z}}}
\newcommand{\U}[1]{\underline{#1}}
\newcommand{\ZZ}{\mathbb{Z}}
\newcommand{\FF}{\mathbb{F}}

\newcommand{\Ov}[1]{\overline{#1}}
\newcommand{\ov}{\Ov{v}}
\newcommand{\ovr}{\Ov{r}}
\newcommand{\ot}{\Ov{t}}

\newcommand{\RR}{\mathbb{R}}

\newcommand{\JJJ}{\blacktriangle}

\renewcommand{\star}{\bigstar}
\newcommand{\D}{\mathscr{D}\!}
\newcommand{\DP}{\mathscr{D}P}

	\newcommand{\SfinG}{\mathcal{S}^{\mathrm{fin},G}_*}
\newcommand{\SG}{\mathcal{S}^{G}_*}
\DeclareMathOperator{\Ind}{Ind}
\DeclareMathOperator{\Sub}{Sub}
\DeclareMathOperator{\Cat}{Cat}
\DeclareMathOperator{\id}{id}

\newcommand{\ssdiff}[1]{\ssarrow{-1}{#1}\ssmove{1}{-#1}}
\newcommand{\ssreddiff}[1]{\ssarrow[color=red]{-1}{#1}\ssmove{1}{-#1}}

\newcommand{\ssblueline}[1]{\ssline[color=blue]{0}{#1}}
\newcommand{\ssgreenline}[1]{\ssline[color=green]{0}{#1}}

\newcommand{\sscurvedblueline}[1]{\ssline[color=blue, curve=-.1]{0}{#1}}

\def\JJ{\blacktriangledown}

\newcount\xval
\newcount\yval
\newcount\difflim

\newcount\ssdoublexval
\newcount\ssdoubleyval

\def\ssray#1#2#3#4#5#6#7{
\ifnum#3>#1
\else
\ifnum#4>#2
\else
   \ssmoveto{#3}{#4}\ssdrop{#5}
   \xval=#3
   \yval=#4
   \advance \xval by #6
   \advance \yval by #7
   \ssray{#1}{#2}{\xval}{\yval}{#5}{#6}{#7}
\fi
\fi
}

\def\sslineray#1#2#3#4#5#6#7#8{
\ifnum#3>#1
\else
\ifnum#4>#2
\else
   \ssmoveto{#3}{#4}\ssline{#5}{#6}\ssmove{-#5}{-#6}
   \xval=#3
   \yval=#4
   \advance \xval by #7
   \advance \yval by #8
   \sslineray{#1}{#2}{\xval}{\yval}{#5}{#6}{#7}{#8}
\fi
\fi
}

\def\sseightfanray#1#2#3#4#5#6{
\ifnum#3>#1\else \ifnum#4>#2\else
   \xval=#3
   \yval=#4
   \ssmoveto{#3}{#4}
   \advance \xval by 1
   \advance \yval by 7
   \ifnum\xval>#1\else \ifnum\yval>#2\else
         \ssline{1}{7}\ssmove{-1}{-7}
   \fi\fi
   \advance \xval by 2
   \advance \yval by -2
   \ifnum\xval>#1\else \ifnum\yval>#2\else
         \ssline{3}{5}\ssmove{-3}{-5}
   \fi\fi
   \advance \xval by 2
   \advance \yval by -2
   \ifnum\xval>#1\else \ifnum\yval>#2\else
         \ssline{5}{3}\ssmove{-5}{-3}
   \fi\fi
   \advance \xval by 2
   \advance \yval by -2
   \ifnum\xval>#1\else \ifnum\yval>#2\else
         \ssline{7}{1}\ssmove{-7}{-1}
   \fi\fi
   \advance \xval by -7
   \advance \yval by -1
   \advance \xval by #5
   \advance \yval by #6
   \sseightfanray{#1}{#2}{\xval}{\yval}{#5}{#6}
\fi
\fi
}

\def\sscoloreightfanray#1#2#3#4#5#6{
\ifnum#3>#1\else \ifnum#4>#2\else
   \xval=#3
   \yval=#4
   \ssmoveto{#3}{#4}
   \advance \xval by 1
   \advance \yval by 7
   \ifnum\xval>#1\else \ifnum\yval>#2\else
         \ssline{1}{7}\ssmove{-1}{-7}
   \fi\fi
   \advance \xval by 2
   \advance \yval by -2
   \ifnum\xval>#1\else \ifnum\yval>#2\else
         \ssline[color=blue]{3}{5}\ssmove{-3}{-5}
   \fi\fi
   \advance \xval by 2
   \advance \yval by -2
   \ifnum\xval>#1\else \ifnum\yval>#2\else
         \ssline[color=green]{5}{3}\ssmove{-5}{-3}
   \fi\fi
   \advance \xval by 2
   \advance \yval by -2
   \ifnum\xval>#1\else \ifnum\yval>#2\else
         \ssline[color=cyan]{7}{1}\ssmove{-7}{-1}
   \fi\fi
   \advance \xval by -7
   \advance \yval by -1
   \advance \xval by #5
   \advance \yval by #6
   \sscoloreightfanray{#1}{#2}{\xval}{\yval}{#5}{#6}
\fi
\fi
}

\def\sseightfandoubleray#1#2#3#4#5#6#7#8{
\ifnum#3>#1\else \ifnum#4>#2\else
   \sseightfanray{#1}{#2}{#3}{#4}{#5}{#6}
   \ssdoublexval=#3
   \ssdoubleyval=#4
   \advance \ssdoublexval by #7
   \advance \ssdoubleyval by #8
   \sseightfandoubleray{#1}{#2}{\ssdoublexval}{\ssdoubleyval}{#5}{#6}{#7}{#8}
\fi\fi
}

\def\sscoloreightfandoubleray#1#2#3#4#5#6#7#8{
\ifnum#3>#1\else \ifnum#4>#2\else
   \sscoloreightfanray{#1}{#2}{#3}{#4}{#5}{#6}
   \ssdoublexval=#3
   \ssdoubleyval=#4
   \advance \ssdoublexval by #7
   \advance \ssdoubleyval by #8
   \sscoloreightfandoubleray{#1}{#2}{\ssdoublexval}{\ssdoubleyval}
                            {#5}{#6}{#7}{#8}
\fi\fi
}

\def\ssantiray#1#2#3#4#5#6#7{
\ifnum#3<#1
\else
\ifnum#4<#2
\else
   \ssmoveto{#3}{#4}\ssdrop{#5}
   \xval=#3
   \yval=#4
   \advance \xval by #6
   \advance \yval by #7
   \ssantiray{#1}{#2}{\xval}{\yval}{#5}{#6}{#7}
\fi
\fi
}

\def\ssraydiff#1#2#3#4#5#6#7#8{
\difflim=#2
\advance \difflim by -#8
\ifnum#3>#1
\else
\ifnum#4>#2
\else
   \ssmoveto{#3}{#4}\ssdrop{#5}
   \ifnum#4>\difflim
   \else
       \ssdiff{#8}
   \fi
   \xval=#3
   \yval=#4
   \advance \xval by #6
   \advance \yval by #7
   \ssraydiff{#1}{#2}{\xval}{\yval}{#5}{#6}{#7}{#8}
\fi
\fi
}

\def\ssrayreddiff#1#2#3#4#5#6#7#8{
\difflim=#2
\advance \difflim by -#8
\ifnum#3>#1
\else
\ifnum#4>#2
\else
   \ssmoveto{#3}{#4}\ssdrop{#5}
   \ifnum#4>\difflim
   \else
       \ssreddiff{#8}
   \fi
   \xval=#3
   \yval=#4
   \advance \xval by #6
   \advance \yval by #7
   \ssrayreddiff{#1}{#2}{\xval}{\yval}{#5}{#6}{#7}{#8}
\fi
\fi
}

\def\ssdoubleray#1#2#3#4#5#6#7#8#9{
\ifnum#3>#1
\else
\ifnum#4>#2
\else
   \ssray{#1}{#2}{#3}{#4}{#5}{#6}{#7}
   \ssdoublexval=#3
   \ssdoubleyval=#4
   \advance \ssdoublexval by #8
   \advance \ssdoubleyval by #9
   \ssdoubleray{#1}{#2}{\ssdoublexval}{\ssdoubleyval}
               {#5}{#6}{#7}{#8}{#9}
\fi
\fi
}

\def\ssantidoubleray#1#2#3#4#5#6#7#8#9{
\ifnum#3<#1
\else
\ifnum#4<#2
\else
   \ssantiray{#1}{#2}{#3}{#4}{#5}{#6}{#7}
   \ssdoublexval=#3
   \ssdoubleyval=#4
   \advance \ssdoublexval by #8
   \advance \ssdoubleyval by #9
   \ssantidoubleray{#1}{#2}{\ssdoublexval}{\ssdoubleyval}{#5}{#6}{#7}{#8}{#9}
\fi
\fi
}

\def\ssrayblueline#1#2#3#4#5#6#7{
\ifnum#3>#1
\else
\ifnum#4>#2
\else
   \ssmoveto{#3}{#4}\ssdrop{#5}\ssblueline{-2}
   \xval=#3
   \yval=#4
   \advance \xval by #6
   \advance \yval by #7
   \ssrayblueline{#1}{#2}{\xval}{\yval}{#5}{#6}{#7}
\fi
\fi
}

\def\ssraygreenline#1#2#3#4#5#6#7{
\ifnum#3>#1
\else
\ifnum#4>#2
\else
   \ssmoveto{#3}{#4}\ssdrop{#5}\ssgreenline{-2}
   \xval=#3
   \yval=#4
   \advance \xval by #6
   \advance \yval by #7
   \ssraygreenline{#1}{#2}{\xval}{\yval}{#5}{#6}{#7}
\fi
\fi
}

\def\ssdoublerayblueline#1#2#3#4#5#6#7#8#9{
\ifnum#3>#1
\else
\ifnum#4>#2
\else
   \ssrayblueline{#1}{#2}{#3}{#4}{#5}{#6}{#7}
   \ssdoublexval=#3
   \ssdoubleyval=#4
   \advance \ssdoublexval by #8
   \advance \ssdoubleyval by #9
   \ssdoublerayblueline{#1}{#2}{\ssdoublexval}{\ssdoubleyval}
                       {#5}{#6}{#7}{#8}{#9}
\fi
\fi
}

\def\ssdoubleraygreenline#1#2#3#4#5#6#7#8#9{
\ifnum#3>#1
\else
\ifnum#4>#2
\else
   \ssraygreenline{#1}{#2}{#3}{#4}{#5}{#6}{#7}
   \ssdoublexval=#3
   \ssdoubleyval=#4
   \advance \ssdoublexval by #8
   \advance \ssdoubleyval by #9
   \ssdoubleraygreenline{#1}{#2}{\ssdoublexval}{\ssdoubleyval}
                       {#5}{#6}{#7}{#8}{#9}
\fi
\fi
}

\def\ssrayblueblueline#1#2#3#4#5#6#7{
\ifnum#3>#1
\else
\ifnum#4>#2
\else
   \ssmoveto{#3}{#4}\ssdrop{#5}\ssblueline{-4}
   \xval=#3
   \yval=#4
   \advance \xval by #6
   \advance \yval by #7
   \ssrayblueblueline{#1}{#2}{\xval}{\yval}{#5}{#6}{#7}
\fi
\fi
}

\def\ssraygreengreenline#1#2#3#4#5#6#7{
\ifnum#3>#1
\else
\ifnum#4>#2
\else
   \ssmoveto{#3}{#4}\ssdrop{#5}\ssgreenline{-4}
   \xval=#3
   \yval=#4
   \advance \xval by #6
   \advance \yval by #7
   \ssraygreengreenline{#1}{#2}{\xval}{\yval}{#5}{#6}{#7}
\fi
\fi
}

\def\ssdoublerayblueblueline#1#2#3#4#5#6#7#8#9{
\ifnum#3>#1
\else
\ifnum#4>#2
\else
   \ssrayblueblueline{#1}{#2}{#3}{#4}{#5}{#6}{#7}
   \ssdoublexval=#3
   \ssdoubleyval=#4
   \advance \ssdoublexval by #8
   \advance \ssdoubleyval by #9
   \ssdoublerayblueblueline{#1}{#2}{\ssdoublexval}{\ssdoubleyval}
                           {#5}{#6}{#7}{#8}{#9}
\fi
\fi
}

\def\ssdoubleraygreengreenline#1#2#3#4#5#6#7#8#9{
\ifnum#3>#1
\else
\ifnum#4>#2
\else
   \ssraygreengreenline{#1}{#2}{#3}{#4}{#5}{#6}{#7}
   \ssdoublexval=#3
   \ssdoubleyval=#4
   \advance \ssdoublexval by #8
   \advance \ssdoubleyval by #9
   \ssdoubleraygreengreenline{#1}{#2}{\ssdoublexval}{\ssdoubleyval}
                           {#5}{#6}{#7}{#8}{#9}
\fi
\fi
}

\def\ssreddiffray#1#2#3#4#5#6#7{
\difflim=#2
\advance \difflim by -#7
\ifnum#3>#1
\else
\ifnum#4>\difflim
\else
       \ssmoveto{#3}{#4}\ssreddiff{#7}
       \xval=#3
       \yval=#4
       \advance \xval by #5
       \advance \yval by #6
       \ssreddiffray{#1}{#2}{\xval}{\yval}{#5}{#6}{#7}
\fi
\fi
}

\def\reddiffdoubleray#1#2#3#4#5#6#7#8#9{
\difflim=#2
\advance \difflim by -#7
\ifnum#3 > #1
\else
\ifnum#4 > \difflim
\else
       \ssreddiffray{#1}{#2}{#3}{#4}{#5}{#6}{#7}
       \ssdoublexval=#3
       \ssdoubleyval=#4
       \advance \ssdoublexval by #8
       \advance \ssdoubleyval by #9
       \reddiffdoubleray{#1}{#2}{\ssdoublexval}{\ssdoubleyval}{#5}{#6}{#7}{#8}{#9}
\fi
\fi
}

\def\ssreddiffantiray#1#2#3#4#5#6#7{
\difflim=#1
\advance \difflim by 1
\ifnum#3< \difflim
\else
\ifnum#4<#2
\else
       \ssmoveto{#3}{#4}\ssreddiff{#7}
       \xval=#3
       \yval=#4
       \advance \xval by #5
       \advance \yval by #6
       \ssreddiffantiray{#1}{#2}{\xval}{\yval}{#5}{#6}{#7}
\fi
\fi
}

\def\reddiffantidoubleray#1#2#3#4#5#6#7#8#9{
\difflim=#2
\ifnum#3<#1
\else
\ifnum#4<\difflim
\else
       \ssreddiffantiray{#1}{#2}{#3}{#4}{#5}{#6}{#7}
       \ssdoublexval=#3
       \ssdoubleyval=#4
       \advance \ssdoublexval by #8
       \advance \ssdoubleyval by #9
       \reddiffantidoubleray{#1}{#2}{\ssdoublexval}{\ssdoubleyval}{#5}{#6}{#7}{#8}{#9}
\fi
\fi
}

\def\ssbluelineray#1#2#3#4#5#6#7{
\difflim=#2
\advance \difflim by -#7
\ifnum#3>#1
\else
\ifnum#4>\difflim
\else
       \ssmoveto{#3}{#4}\ssblueline{#7}
       \xval=#3
       \yval=#4
       \advance \xval by #5
       \advance \yval by #6
       \ssbluelineray{#1}{#2}{\xval}{\yval}{#5}{#6}{#7}
\fi
\fi
}

\def\sscurvedbluelineray#1#2#3#4#5#6#7{
\difflim=#2
\advance \difflim by -#7
\ifnum#3>#1
\else
\ifnum#4>\difflim
\else
       \ssmoveto{#3}{#4}\sscurvedblueline{#7}
       \xval=#3
       \yval=#4
       \advance \xval by #5
       \advance \yval by #6
       \sscurvedbluelineray{#1}{#2}{\xval}{\yval}{#5}{#6}{#7}
\fi
\fi
}

\def\ssbluelineantiray#1#2#3#4#5#6#7{
\ifnum#3<#1
\else
\ifnum#4<#2
\else
       \ssmoveto{#3}{#4}\ssblueline{#7}
       \xval=#3
       \yval=#4
       \advance \xval by #5
       \advance \yval by #6
       \ssbluelineantiray{#1}{#2}{\xval}{\yval}{#5}{#6}{#7}
\fi
\fi
}

\def\ssgreenlineantiray#1#2#3#4#5#6#7{
\ifnum#3<#1
\else
\ifnum#4<#2
\else
       \ssmoveto{#3}{#4}\ssgreenline{#7}
       \xval=#3
       \yval=#4
       \advance \xval by #5
       \advance \yval by #6
       \ssgreenlineantiray{#1}{#2}{\xval}{\yval}{#5}{#6}{#7}
\fi
\fi
}

\def\bluelinedoubleray#1#2#3#4#5#6#7#8#9{
\difflim=#2
\advance \difflim by -#7
\ifnum#3>#1
\else
\ifnum#4>\difflim
\else
       \ssbluelineray{#1}{#2}{#3}{#4}{#5}{#6}{#7}
       \ssdoublexval=#3
       \ssdoubleyval=#4
       \advance \ssdoublexval by #8
       \advance \ssdoubleyval by #9
       \bluelinedoubleray{#1}{#2}{\ssdoublexval}{\ssdoubleyval}{#5}{#6}{#7}{#8}{#9}
\fi
\fi
}

\def\bluelineantidoubleray#1#2#3#4#5#6#7#8#9{
\difflim=#2
\advance \difflim by -#7
\ifnum#3<#1
\else
\ifnum#4<\difflim
\else
       \ssbluelineantiray{#1}{#2}{#3}{#4}{#5}{#6}{#7}
       \ssdoublexval=#3
       \ssdoubleyval=#4
       \advance \ssdoublexval by #8
       \advance \ssdoubleyval by #9
       \bluelineantidoubleray{#1}{#2}{\ssdoublexval}{\ssdoubleyval}{#5}{#6}{#7}{#8}{#9}
\fi
\fi
}

\def\greenlineantidoubleray#1#2#3#4#5#6#7#8#9{
\difflim=#2
\advance \difflim by -#7
\ifnum#3<#1
\else
\ifnum#4<\difflim
\else
       \ssgreenlineantiray{#1}{#2}{#3}{#4}{#5}{#6}{#7}
       \ssdoublexval=#3
       \ssdoubleyval=#4
       \advance \ssdoublexval by #8
       \advance \ssdoubleyval by #9
       \greenlineantidoubleray{#1}{#2}{\ssdoublexval}{\ssdoubleyval}{#5}{#6}{#7}{#8}{#9}
\fi
\fi
}

\def\bluessray#1#2#3#4#5#6#7{
\ifnum#3>#1
\else
\ifnum#4>#2
\else
   \ssmoveto{#3}{#4}\ssdrop[color=blue]{#5}
   \xval=#3
   \yval=#4
   \advance \xval by #6
   \advance \yval by #7
   \bluessray{#1}{#2}{\xval}{\yval}{#5}{#6}{#7}
\fi
\fi
}

\def\violetssray#1#2#3#4#5#6#7{
\ifnum#3>#1
\else
\ifnum#4>#2
\else
   \ssmoveto{#3}{#4}\ssdrop[color=violet]{#5}
   \xval=#3
   \yval=#4
   \advance \xval by #6
   \advance \yval by #7
   \violetssray{#1}{#2}{\xval}{\yval}{#5}{#6}{#7}
\fi
\fi
}

\def\violetssantiray#1#2#3#4#5#6#7{
\ifnum#3<#1
\else
\ifnum#4<#2
\else
   \ssmoveto{#3}{#4}\ssdrop[color=violet]{#5}
   \xval=#3
   \yval=#4
   \advance \xval by #6
   \advance \yval by #7
   \violetssantiray{#1}{#2}{\xval}{\yval}{#5}{#6}{#7}
\fi
\fi
}

\def\bluessantiray#1#2#3#4#5#6#7{
\ifnum#3<#1
\else
\ifnum#4<#2
\else
   \ssmoveto{#3}{#4}\ssdrop[color=blue]{#5}
   \xval=#3
   \yval=#4
   \advance \xval by #6
   \advance \yval by #7
   \bluessantiray{#1}{#2}{\xval}{\yval}{#5}{#6}{#7}
\fi
\fi
}

\def\ssbluelinedoubleray#1#2#3#4#5#6#7#8#9{
\ifnum#3>#1
\else
\ifnum#4>#2
\else
   \ssbluelineray{#1}{#2}{#3}{#4}{#5}{#6}{#7}
   \ssdoublexval=#3
   \ssdoubleyval=#4
   \advance \ssdoublexval by #8
   \advance \ssdoubleyval by #9
   \ssbluelinedoubleray{#1}{#2}{\ssdoublexval}{\ssdoubleyval}
                       {#5}{#6}{#7}{#8}{#9}
\fi
\fi
}

\def\ssgreenlineray#1#2#3#4#5#6#7{
\difflim=#2
\advance \difflim by -#7
\ifnum#3>#1
\else
\ifnum#4>\difflim
\else
       \ssmoveto{#3}{#4}\ssgreenline{#7}
       \xval=#3
       \yval=#4
       \advance \xval by #5
       \advance \yval by #6
       \ssgreenlineray{#1}{#2}{\xval}{\yval}{#5}{#6}{#7}
\fi
\fi
}

\def\ssgreenlinedoubleray#1#2#3#4#5#6#7#8#9{
\ifnum#3>#1
\else
\ifnum#4>#2
\else
   \ssgreenlineray{#1}{#2}{#3}{#4}{#5}{#6}{#7}
   \ssdoublexval=#3
   \ssdoubleyval=#4
   \advance \ssdoublexval by #8
   \advance \ssdoubleyval by #9
   \ssgreenlinedoubleray{#1}{#2}{\ssdoublexval}{\ssdoubleyval}
                       {#5}{#6}{#7}{#8}{#9}
\fi
\fi
}

\newcommand{\cC}{\mathcal{C}}
\renewcommand{\cD}{\mathcal{D}}
\newcommand{\cF}{\mathcal{F}}
\newcommand{\cO}{\mathcal{O}}

\newcommand{\cP}{\mathcal{P}}
\newcommand{\cU}{\mathcal{U}}
\newcommand{\cT}{\mathcal{T}}
\newcommand{\tensor}{\otimes}
\newcommand{\sm}{\wedge}

\newcommand{\MUR}{MU_{\mathbb{R}}}
\newcommand{\BPR}{BP_{\mathbb{R}}}
\newcommand{\BPn}{{BP}^{(\!(C_{2^{n}})\!)}}
\newcommand{\BPG}{{BP}^{(\!(G)\!)}}
\newcommand{\MUn}{{MU}^{(\!(C_{2^{n}})\!)}}
\newcommand{\MUG}{{MU}^{(\!(G)\!)}}
\newcommand{\BPfour}{{BP}^{(\!(C_4)\!)}}

\newcommand{\EFi}{\tilde{E}\mathcal{F}[C_{2^i}]}
\newcommand{\EF}{\tilde{E}\mathcal{F}}
\newcommand{\HZ}{H\underline{\mathbb{Z}}}

\DeclareMathOperator{\SliceSS}{SliceSS}
\DeclareMathOperator{\HFPSS}{HFPSS}
\DeclareMathOperator{\TateSS}{TateSS}
\DeclareMathOperator{\Map}{Map}
\DeclareMathOperator{\Ho}{Ho}
\DeclareMathOperator{\Sp}{Sp}
\DeclareMathOperator{\colim}{colim}
\DeclareMathOperator{\hocolim}{hocolim}
\DeclareMathOperator{\ind}{Ind}
\DeclareMathOperator{\res}{Res}
\DeclareMathOperator{\Res}{Res}
\DeclareMathOperator{\Tr}{Tr}
\DeclareMathOperator{\im}{im}

\newcommand{\qedxymatrix}[1]{\begin{gathered}[b]
	\xymatrix{#1}\\[-\dp\strutbox]
	\end{gathered} \qedhere}

\crefname{figure}{Figure}{Figures}

\title[Norms of Real Bordism and the Segal conjecture]{The localized slice spectral sequence, norms of Real bordism, and the Segal conjecture}

\author[Meier]{Lennart Meier}
\address{Mathematical Institute, Utrecht University, Utrecht, 3584 CD, the Netherlands}
\email{f.l.m.meier@uu.nl}

\author[Shi]{XiaoLin Danny Shi}
\address{Department of Mathematics, 
University of Washington, 4110 E Stevens Way NE, Seattle, WA 98195}
\email{dannyshixl@gmail.com}

\author[Zeng]{Mingcong Zeng}
\address{Max Planck Institute for Mathematics, Vivatsgasse 7, 53111 Bonn, Germany}
\email{mingcongzeng@gmail.com}

\begin{document}
	\maketitle
	
	\begin{abstract}
	In this paper, we introduce the localized slice spectral sequence, a variant of the equivariant slice spectral sequence that computes geometric fixed points equipped with residue group actions. We prove convergence and recovery theorems for the localized slice spectral sequence and use it to analyze the norms of the Real bordism spectrum. As a consequence, we relate the Real bordism spectrum and its norms to a form of the $C_2$-Segal conjecture.  We compute the localized slice spectral sequence of the $C_4$-norm of $BP_\mathbb{R}$ in a range and show that the Hill--Hopkins--Ravenel slice differentials is in one-to-one correspondence with a family of Tate differentials for $N_1^2 H{\mathbb{F}}_2$.
	\end{abstract}

	\setcounter{tocdepth}{1}
	
	\tableofcontents

\section{Introduction}

The complex conjugation action on the complex bordism spectrum $MU$ defines a $C_2$-spectrum $\MUR$, the Real bordism spectrum of Landweber, Fujii, and Araki \cite{Landweber:MU, Fujii, Araki:BPR}.  Its norms 
\[\MUn:= N_{C_2}^{C_{2^n}}\MUR = N_{2}^{2^n}\MUR \]
have played a central role in the solution of the Kervaire invariant one problem \cite{HHR}. After localizing at $2$, the norm $\MUn$ splits as a wedge of suspensions of $\BPn:=N_{2}^{2^n}\BPR$, where $\BPR$ is the Real Brown--Peterson spectrum. 

The spectra $\MUn$ and $\BPn$ connect many fundamental objects and computations in non-equivariant stable homotopy theory to equivariant stable homotopy theory.  The fixed points of these norms are ring spectra, and their Hurewicz images detect families of elements in the stable homotopy groups of spheres \cite{HHR, Hill:eta, LSWX}.  The Lubin--Tate spectra at prime 2 with finite group actions can also be built from these norms and their quotients \cite{HahnShi, BHSZ}.  They produce higher height analogues of topological $K$-theory and play a fundamental role in chromatic homotopy theory.

To compute the equivariant homotopy groups of $\MUn$ and $\BPn$, Hill, Hopkins, and Ravenel introduced the equivariant slice spectral sequence \cite{HHR}.  However, due to the complexity of the equivariant computations, besides $\MUR$ and $\BPR$, we still know relatively little about the behavior of their norms.  For example, we are still far from a complete understanding of the equivariant homotopy groups of $\BPfour$. 

Our project arose from the desire to systematically understand the equivariant homotopy groups of $\MUn$ and $\BPn$.  The goal of this paper is two-fold: first, we establish our main computational tool, the \emph{localized slice spectral sequence}.  This is a variant of the slice spectral sequence that is easier for computations while at the same time recovers the original slice differentials.  Second, as an application of the localized slice spectral sequence, we focus on the $C_4$-norm $\BPfour$.  We compute its localized slice spectral sequence in a range and build a new connection to the Segal conjecture at $C_2$.  As a consequence, we establish correspondences between families of slice differentials for $\BPfour$ and families of differentials in the Tate spectral sequence for $N_1^2 H\FF_2$.

\subsection{Fixed points and geometric fixed points}\hfill\\

It is well-known in equivariant stable homotopy theory that a map between $G$-spectra is a weak equivalence if and only if for all subgroups $H \subset G$, it induces (non-equivariant) weak equivalences on all $H$-fixed points or $H$-geometric fixed points.  Despite this fact, fixed points and geometric fixed points behave very differently. 

The fixed points of a $G$-spectrum $X$ can be difficult to understand.  For a suspension spectrum, its fixed points can be described by using the tom Dieck splitting \cite[Section~V.11]{LMayS}, but such a splitting does not exist in general.  Nevertheless, by the Wirthm\"uller isomorphism, there are natural maps between fixed points of different subgroups of $G$.  The induced maps on their homotopy groups can be assembled into an algebraic object $\underline{\pi}_* X$, called a Mackey functor.  Organizing information in terms of Mackey functors is one of the most powerful ideas in equivariant stable homotopy theory, and this has produced new insights in both theory and computation (e.g. \cite{Guillou-May,HHR}). 

As an important example, the $C_2$-fixed points of the Real bordism spectrum $\MUR$ is computable but complicated \cite{HuKriz1, GreenleesMeier}.  For groups beyond $C_2$, we still don't know very much about the fixed points of the norms $\MUn$ aside from the computations in \cite{HHR, HHR:C4, Hill:eta, HSWX}.  Nevertheless, these fixed points contain very rich information about the stable homotopy groups of spheres (such as the Kervaire invariant elements) and chromatic homotopy theory \cite{HHR, LSWX, HahnShi, BHSZ}. 

On the other hand, the geometric fixed points are easier to understand.  The geometric fixed points functor $\Phi^H\colon \Sp_G \to \Sp$ is compatible with the suspension spectrum functor, commutes with all homotopy colimits, and is symmetric monoidal.  

For the Real bordism spectrum $\MUR$, a straightforward geometric argument, based on the fact that the fixed points of the $C_2$-Galois action on $\mathbb{C}$ is $\RR$, shows that the $C_2$-geometric fixed points of $MU_{\RR}$ and $\BPR$ are $MO$ (the unoriented bordism spectrum) and $H\FF_2$, respectively.  The geometric fixed points functor also behaves well with respect to the norm functor \cite[Proposition~2.57]{HHR}.  This renders the geometric fixed points of the norms $\MUn$ easy to understand. 

Although the homotopy groups of the geometric fixed points for various subgroups do not form a Mackey functor, there are reconstruction theorems which recovers a $G$-spectrum from structures on its geometric fixed points \cite{AbramKriz, Glasman, AyalaMazel-GeeRozenblyum}. 

At this point, it is natural to ask the following questions: 
\begin{enumerate}
\item How do the fixed points and the geometric fixed points of an equivariant spectrum interact with each other?
\item Computationally, how to recover the fixed points of equivariant spectra, such as norms of $\MUR$, through their geometric fixed points, which are significantly easier to compute? 
\end{enumerate}

In order to attack these questions, the first observation is that it is necessary to consider the $H$-geometric fixed points not only as a non-equivariant spectrum, but as a $W_G(H)$-equivariant spectrum, where $W_G(H)$ is the Weyl group.  In our examples of interest, $H$ will be a normal subgroup of $G$, so that $W_G(H) \cong G/H$.  When the $G$-spectrum is of the form $N_H^G X$, we prove the following theorem. 

\begin{thm}\label{thm:introthm1}
    Let $H \subset G$ be a normal subgroup and $X$ be an $H$-spectrum. Then we have an equivalence of $G/H$-spectra
 	\[
 	    \Phi^H N_H^G X \simeq N_e^{G/H} \Phi^H(X).
 	\]
 	If $X$ is an $H$-commutative ring spectrum, then this equivalence is an equivalence of $G/H$-commutative ring spectra.
\end{thm}

This theorem is by no means difficult to prove, and in fact it only marks the starting point of our analysis. To understand how the $H$-fixed points and $H$-geometric fixed points interact with each other, we introduce our main computational tool: the localized slice spectral sequence.

\subsection{The localized slice spectral sequence}\hfill\\

Let $X$ be a $G$-spectrum and $H \subset G$ a normal subgroup.  As a $G/H$-spectrum, $\Phi^{H} X$ can be constructed as $(\tilde{E}\mathcal{F}[H] \wedge X)^H$, where $\tilde{E}\mathcal{F}[H]$ is the universal space of the family $\mathcal{F}[H]$ consisting of all subgroups that do not contain $H$.  In many cases, including $G$ cyclic, smashing with $\tilde{E}\mathcal{F}[H]$ is equivalent to inverting an Euler class $a_V \in \pi^{G}_{-V}S^0$ for $V$ a certain $G$-representation. In particular, the residue fixed points $(\Phi^H X)^{G/H}$ are equivalent to the fixed points $(a_{V}^{-1}X)^G$.

To define the localized slice spectral sequence, let $P^{\bullet}X$ be the regular slice tower of $X$ \cite{HHR}\cite{Ullman:Thesis}.  The \textit{$a_{V}$-localized slice spectral sequence} of $X$ is, by definition, the spectral sequence corresponding to the localized tower $\{a_{V}^{-1}P^{\bullet}X\}$.  It has $E_2$-page 
\[E_2^{s,t} = \underline{\pi}_{t-s} a_V^{-1} P^t_t X.\]
\begin{thm}\label{thm:introThm2}
    Let $X$ be a $C_{2^n}$-spectrum and $V$ be an actual $C_{2^n}$-representation.  Then the $a_V$-localized slice spectral sequence converges strongly to the homotopy groups $\underline{\pi}_{t-s}a_{V}^{-1}X$.
\end{thm}

The localized slice spectral sequence serves as a bridge between the fixed points $X^G$ and the residue fixed points $(\Phi^H X)^{G/H}$.  More precisely, even though the localized slice spectral sequence only computes the geometric fixed points, its $E_2$-page is closely related to the original slice spectral sequence, which computes the fixed points.  From now on, we will denote the regular slice spectral sequence and the $a_V$-localized slice spectral sequence of $X$ by $\SliceSS(X)$ and $a_{V}^{-1}\SliceSS(X)$, respectively.
The following theorem directly follows from computations of the homotopy groups of $H\UZ$ \cite[Section~3]{HHR:C4}.

\begin{thm}\label{intro:iso}
    Let $X$ be a $(-1)$-connected $C_{2^n}$-spectrum whose slices are wedges of the form $C_{2^n+} \wedge_{C_{2^k}} \Sigma^{i\rho_k} H\UZ$, and $\lambda$ be the 2-dimensional real $C_{2^n}$-representation that is rotation by $\frac{\pi}{2^{n-1}}$. Then the localizing map
    \[
        \SliceSS(X) \longrightarrow a_{\lambda}^{-1}\SliceSS(X)
    \]
    induces an isomorphism on the $E_2$-page for classes whose filtration is greater than $0$.  On the $0$-line, this map is surjective, with kernel consisting of elements in the image of the transfer $Tr_{e}^{C_{2^n}}$.
\end{thm}

By the slice theorem \cite[Theorem~6.1]{HHR}, the $C_{2^n}$-norms of $\MUR$ and $\BPR$ both satisfy the conditions of Theorem~\ref{intro:iso}.

An upshot of Theorem~\ref{intro:iso} is that despite the fact that the fixed points are harder to compute than the geometric fixed points, if we already know the differentials in the localized slice spectral sequence, then we can use the isomorphism on the $E_2$-page given by Theorem~\ref{intro:iso} to recover differentials in the original slice spectral sequence.  This allows us to approach the computation of the fixed points $X^G$ from the residue fixed points $(\Phi^{H}X)^{G/H}$. 

A subtlety that arises from the localized slice spectral sequence is its compatibility with multiplicative structures.  More precisely, let $R$ be a connective $G$-commutative ring spectrum.  Ullman \cite{Ullman:Thesis} has shown that the slice tower of $R$ is multiplicative.  Therefore, the corresponding slice spectral sequence has all the desired multiplicative properties such as the Leibniz rule, the Frobenius relation \cite[Definition~2.3]{HHR:C4}, and most importantly, the norm \cite[Corollary~4.8]{HHR:C4}.  On the other hand, the localization $a_{V}^{-1}R$ can never be a $G$-commutative ring spectrum because its underlying spectrum is contractible. 

To establish multiplicative properties for the localizations, we apply the theory of $N_{\infty}$-operads from \cite{BlumbergHill}.  More precisely, in Section \ref{sec:MultLoc}, we establish a criterion generalizing the results of \cite{HillHopkins} and \cite{Boehme}.  As a consequence, we obtain the following theorem, which shows that $a_V$-localization preserves algebra structures over a certain $N_{\infty}$-operad $\mathcal{O}$ that depends on the class $a_{V}$.

\begin{thm}\label{thm:introMult}
Let $V$ be a $G$-representation. Assume that $\ind_K^H\Res_K^GV$ is a summand of a multiple of $\Res_H^GV$ for every $K\subset H \subset G$ such that $H/K$ is an admissible $H$-set. Then localization at $a_V$ preserves $\cO$-algebras.     
\end{thm}
Therefore, the homotopy of the $a_{V}$-localization of an equivariant commutative ring spectrum such as $\MUn$ forms an incomplete Tambara functor \cite{Blumberg-Hill:incomplete}, and the norm maps essential to our computation are still available.  In Section \ref{sec-norm}, we draw consequences of the behavior of norms in the localized slice spectral sequence.

Aside from the localized slice spectral sequence $a_\lambda^{-1}\SliceSS(X)$, the $G/H$-slice spectral sequence of $\Phi^H X$ also computes the residue fixed points $(\Phi^{H}X)^{G/H}$.  Even though both spectral sequences compute the same homotopy groups, their behaviors can be very different.  Surprisingly, we have the following theorem, which shows that after a modification of filtrations, there is map between the two spectral sequences. 

\begin{thm}\label{thm:introSSmap}
    Let $X$ be a $C_{2^n}$-spectrum, then there is a canonical map of spectral sequences
    \[
    a_{\lambda}^{-1}\SliceSS^{C_{2^n}} (X) \rightarrow P_{C_{2^n}/C_2}^*\D \SliceSS^{C_{2^n}/C_2}(\Phi^{C_2}X)
    \]
    that converges to an isomorphism in homotopy groups. Here $\D$ is the doubling operation defined in \cref{sec:compspec}, which slows down a tower by a factor of $2$, and $P_{C_{2^n}/C_2}^*$ is the pullback functor from \cite{Hill:Slice}, which is recalled in \cref{sec:GeometricFixedPoints}.
\end{thm}

In the second half of the paper, as an application of all the tools that we have developed, we will use the localized slice spectral sequence to analyze the norms of $MU_{\RR}$.

\subsection{Norms of Real bordism and the Segal conjecture}\hfill\\

The Segal conjecture is a deep result in equivariant homotopy theory. In its original formulation, it was proven by Lin \cite{Lin} for the group $C_2$ and by Carlsson \cite{Carlsson:Segal} for all finite groups, building on the works of May--McClure \cite{MayMcClure} and Adams--Gunawardena--Miller \cite{AdamsGunawardenaMiller}.  When the group is $C_2$, the most general formulation can be found in Lun{\o}e-Nielsen--Rognes \cite{LNR} and Nikolaus--Scholze \cite{NikolausScholze}: for every bounded below spectrum $X$, the Tate diagonal map $X \to (N_1^2X)^{tC_2}$ is a $2$-adic equivalence.

We are interested in the case when $X = H\FF_2$, the mod $2$ Eilenberg--Mac Lane spectrum. This case is intriguing for at least two reasons: first, Nikolaus--Scholze \cite{NikolausScholze} show that the general formulation follows formally from this case.  Second, even though the Segal conjecture implies the equivalence $H\FF_2 \simeq (N_1^2 H\FF_2)^{tC_2}$, this is still a mystery from a computational perspective. 

More precisely, the Tate spectral sequence computing $(N_1^2 H\FF_2)^{tC_2}$ has $E_2$-page $\hat{H}^*(C_2;\mathcal{A}_*)$, the Tate cohomology of the dual Steenrod algebra $\mathcal{A}_*$ with the conjugate $C_2$-action.  This cohomology is highly nontrivial and we currently don't even have a closed formula \cite{Bruner:Tate}.  However, because of the equivalence $H\FF_2 \simeq (N_1^2 H\FF_2)^{tC_2}$ given by the Segal conjecture, every element besides $1 \in \FF_2 \cong \hat{H}^0(C_2;(\mathcal{A}_*)_0)$ must either support or receive a differential. 

Understanding equivariant equivalences from a computational perspective can be extremely useful.  For example, in the case of $\BPR$ and its norms, it is relatively straightforward to establish the equivalence $\Phi^{C_{2^n}}\BPn \simeq\Phi^{C_2}\BPR \simeq H\FF_2$.  By working backwards, Hill--Hopkins--Ravenel used this equivalence to prove a family of differentials in the slice spectral sequence of $\BPn$ \cite[Theorem~9.9]{HHR}, from which their Periodicity Theorem and eventually the nonexistence of the Kervaire invariant elements followed. 

By \cref{thm:introthm1}, we have a $C_2$-equivalence 
\[
\Phi^{C_2}\BPfour \simeq N_1^2 H\FF_2.
\]
For the left hand side, we can use the localized slice spectral sequence to compute the $C_2$-fixed points of $\Phi^{C_2}\BPfour$.  We demonstrate this computation in a range (\cref{thm:htpy}).  Note that we can actually compute much further than the range we have shown, but the point is to give the readers a taste of the computations involved and to draw comparisons to the slice spectral sequence computations in \cite{HHR, HHR:C4, HSWX}. 

After demonstrating these computations, we use \cref{thm:introSSmap} to establish a map between the slice spectral sequence of $\BPfour$ and the Tate spectral sequence of $N_1^2H\FF_2$.  We prove that this map establishes a correspondence between families of differentials in the two spectral sequences.  

\begin{thm}\label{thm:introTateSS}
    After the $E_2$-page, the Hill--Hopkins--Ravenel slice differentials \cite[Theorem~9.9]{HHR} are in one-to-one correspondence to a family of differentials on the first diagonal of slope $(-1)$ in the Tate spectral sequence of $N_1^2 H\FF_2$.  This completely determines all the differentials in the Tate spectral sequence that originate from the first diagonal of slope $(-1)$. 
\end{thm}

In the future, we wish to reverse the flow of information: to prove slice differentials from spectral sequences associated to $N_1^2H\FF_2$.  Computations along this line appear in \cite{BHLSZ} and will be refined in an upcoming article by the same authors.  There are various methods to study the norms of $H\FF_2$ and their modules, such as the modified Adams spectral sequence \cite{Rav:Cyclic, BBLNR} and the descent spectral sequence \cite{HahnWilsonSegal}.  These methods allow one to understand modules over norms of $H\FF_2$ and $\BPR$ from different perspectives.  

It is worth noting that in another direction, one can prove the $C_2$-Segal conjecture by showing that $N_2^4 MU_{\RR}$ is cofree and using \cref{thm:introthm1}. This approach is taken by Carrick in \cite{Carrick}.

\cref{thm:introTateSS} has an unexpected consequence. Let $R$ be an arbitrary non-equivariant $(-1)$-connected homotopy ring spectrum with $\pi_0 R\cong \ZZ$ (or a localization thereof not containing $\frac12$). We can use the (stable) EHP spectral sequence and the Tate spectral sequence of $N_1^2 H\FF_2$ to bound the length of differentials on powers of the Tate generator in the Tate spectral sequence of $N_1^2 R$. 

\begin{thm}
    Let $v \in \hat{H}^{2}(C_2;\pi_0N_1^2 R)$ be the generator of the Tate cohomology, and $l_k$ be the length of differential that $v^{2^k}$ supports in the Tate spectral sequence of $N_1^2 R$. Then 
        \[
        \rho(2^{k+1}) \leq l_k \leq 2^{k+2}-1,
        \]
    where $\rho(n)$ is the Radon-Hurwitz number.
\end{thm}

\subsection{Outline of paper}
In \cref{sec:2}, we recall a few basics of equivariant homotopy theory. In particular, we discuss the interplay between the norm functor, the geometric fixed points functor, and the pull back functor. We prove \cref{thm:introthm1}. We also investigate the multiplicative structure of localizations and give a criterion for a localization at an element to preserve multiplicative structures, thus proving \cref{thm:introMult}. 

In \cref{sec:sliceSSBackground}, we recall the spectra $\MUG$ and $\BPG$ and their slice spectral sequences. We then introduce the main computational tool for this paper, the localized slice spectral sequence. We prove \cref{thm:introThm2}, the strong convergence of the localized slice spectral sequence (\cref{thm:invertingAclassE2page}). We also discuss exotic extensions and norms.

\cref{sec:4,section:sliceSScomputation} are dedicated to the computation of the localized slice spectral sequence of $a_{\lambda}^{-1}\BPfour$. In \cref{sec:4}, we give an outline of the computation and list our main results (\cref{thm-main} and \cref{thm:htpy}). The detailed computations are in \cref{section:sliceSScomputation}. While computing differentials, we make full use of the Mackey functor structure of the spectral sequence. Certain differentials are proven using exotic extensions and norms by methods established in \cref{sec:sliceSSBackground}. 

In \cref{sec-tate}, we turn our attention to the Tate spectral sequence of $N_1^2 H\FF_2$. We use the computation of the localized slice spectral sequence of $\BPfour$ to prove families of differentials and compute the Tate spectral sequence in a range. In particular, \cref{thm:introTateSS} is proven as \cref{thm:TateDiff}, which describes the first infinite family of differentials in the Tate spectral sequence. 

\subsection*{Acknowledgments}
The authors would like to thank Bob Bruner for sharing his computation on the Tate generators of the dual Steenrod algebra, and J.D. Quigley for sharing his computation of the Adams spectral sequence of $N_1^2H\FF_2$. The authors would furthermore like to thank Agn\`es Beaudry, Christian Carrick, Gijs Heuts, Mike Hill, Tyler Lawson, Guchuan Li, Viet-Cuong Pham, Doug Ravenel, John Rognes and Jonathan Rubin for helpful conversations.  Finally, we would like to thank the anonymous referee for the many detailed suggestions.  The second author was supported by National Science Foundation grant DMS-2104844.

	\subsection*{Conventions}
\begin{enumerate}
	\item Given a finite group $G$, all representations will be finite-dimensional and orthogonal. Per default actions will be from the left.
	\item We denote by $\rho_G$ the real regular representation of a finite group $G$ and we abbreviate $\rho_{C_2}$ to $\rho_2$.
	\item We will often use the abbrevation $\BPfour$ for $N_2^4\BPR$ and more generally $\BPG$ for $N_{C_2}^G\BPR$.
    \item All spectral sequences use the Adams grading.
    \item We use the regular slice filtration and its corresponding tower and spectral sequence defined in \cite{Ullman:Thesis} throughout the paper, often omitting ``regular".
\end{enumerate} 

	\section{Equivariant stable homotopy theory} \label{sec:2}
	\subsection{A few basics}
	We work in the category of genuine $G$-spectra for a finite group $G$, and our particular model will be the category of orthogonal $G$-spectra $\Sp_G$. For us these will be simply $G$-objects in orthogonal spectra as in \cite{Schwede:ESHT}, which will often be just called \emph{$G$-spectra}. This category is equivalent to the categories of orthogonal $G$-spectra considered in \cite{MandellMay} and \cite{HHR}. In particular, we are able to evaluate a $G$-spectrum at an arbitrary $G$-representation to obtain a $G$-space. We refer to the three cited sources for general background on $G$-equivariant stable homotopy theory, of which we will recall some for the convenience of the reader.
	
	For each $G$-representation $V$, we denote by $S^V$ its one-point compactification. Denoting further by $\rho_G$ the regular representation, we obtain for each subgroup $H\subset G$ and each $G$-spectrum its homotopy groups
	\[\pi^H_n(X) = \colim_{k} [S^{k\rho_G+n},X(k\rho_G)]^H.\]
	These assemble into a Mackey functor $\U{\pi}_n(X)$. A map of $G$-spectra is an \emph{equivalence} if it induces an isomorphism on all $\U{\pi}_n$. Inverting the equivalences of $G$-spectra in the $1$-categorical sense yields the genuine equivariant stable homotopy category $\Ho(\Sp_G)$ and inverting them in the $\infty$-categorical sense the $\infty$-category of $G$-spectra $\Sp_G^{\infty}$. These constructions are well-behaved as there is a stable model structure on $\Sp_G$ with the weak equivalences we just described \cite[Theorem III.4.2]{MandellMay}. The fibrant objects are precisely the $\Omega$-$G$-spectra. In the main body of the paper we will implicitly work in $\Ho(\Sp_G)$ or $\Sp_G^{\infty}$; in particular, commutative squares are meant to be only commutative up to (possibly specified) homotopy. 
	
	By \cite[Proposition V.3.4]{MandellMay}, the categorical fixed point construction $\Sp_G \to \Sp$ is a right Quillen functor. We call the right derived functor $(-)^G\colon \Sp_G^{\infty} \to \Sp^{\infty}$ the \emph{(genuine) fixed points}. We can define fixed point functors for subgroups $H\subset G$ by applying first the restriction functor $\Sp_G \to \Sp_H$ and then the $H$-fixed point functor. One easily shows that $\pi_nX^H \cong \pi_n^HX$. Thus, a map is an equivalence if it is an equivalence on all fixed points.
	
	Note that if $H\subset G$ is normal, the categorical fixed points carry a residual $G/H$-action. The resulting functor $\Sp_G \to \Sp_{G/H}$ is a right Quillen functor as well \cite[p.\ 81]{MandellMay} and thus $H$-fixed points actually define a functor $\Sp_G^{\infty} \to \Sp_{G/H}^{\infty}$. 
	The left adjoint of this is the inflation functor $p^*$ associated to the projection $p\colon G \to G/H$.
	
	As $\pi_n^H$ translates filtered homotopy colimits into colimits, we see that fixed points $\Sp_G^{\infty} \to \Sp^{\infty}$ preserve filtered homotopy colimits. As they preserve homotopy limits as well (as they are induced by a Quillen right adjoint) and are a functor between stable $\infty$-categories, they preserve all finite homotopy colimits \cite[Proposition 1.1.4.1]{HigherAlgebra} and hence all homotopy colimits \cite[Proposition 4.4.2.7]{HTT}. By the associativity of fixed points, the same is true for $(-)^H\colon \Sp^{\infty}_G \to \Sp^{\infty}_{G/H}$ for a normal subgroup $H\subset G$.
	
	\subsection{Geometric fixed points and pullbacks}\label{sec:GeometricFixedPoints}
	To define other versions of fixed points, we need the notion of a universal space for a given family $\cF$ of subgroups of $G$, i.e.\ a collection of subgroups closed under taking subgroups and conjugation. For every such family there exists a \emph{universal space}, i.e.\ a $G$-CW complex $E\cF$, which is up to $G$-homotopy equivalence characterized by its fixed points:
	\[(E\cF)^H \simeq \begin{cases} \ast & \text{ if } H\in \cF\\
	\varnothing & \text{ if } H\notin \cF
	\end{cases} \]
	Examples of such families include the case $\cF = \{e\}$ of just the trivial group, where we denote $E\cF$ by $EG$, and the case $\cF = \cP$ of all proper subgroups. To each family, we can associate furthermore the cofiber $\EF$ of $E\cF_+ \to S^0$, which is again characterized by its fixed points
	\[(\EF)^H \simeq \begin{cases} \ast & \text{ if } H\in \cF\\
	S^0 & \text{ if } H\notin \cF
	\end{cases} \]
	For each family $\cF$ and every $G$-spectrum we have an associated isotropy separation diagram, whose rows are parts of cofiber sequences:
	\[
	\xymatrix{
		X\sm E\cF_+ \ar[r]\ar[d]^{\simeq} & X\ar[d] \ar[r] & X \sm \EF \ar[d] \\
		X \sm E\cF_+ \ar[r] & X^{E\cF_+} \ar[r] & X^{E\cF_+} \sm \EF
	}
	\]
	Upon taking fixed points, we can identify some of the entries with well-known constructions. If $E\cF = EG$, then $(X^{E\cF_+})^G$ is the spectrum of \emph{homotopy fixed points} $X^{hG}$ and $(X \sm E\cF_+)^G$ is (by the Adams isomorphism) the spectrum of \emph{homotopy orbits} $X_{hG}$. Moreover, one calls in this case $(X^{E\cF_+} \sm \EF)^G$ the \emph{Tate construction} and denotes it by $X^{tG}$. If $\cF = \cP$, then $(X\sm \widetilde{E}\cP)^G$ is called the \emph{geometric fixed points} and denoted by $\Phi^GX$. One can show that $\Phi^G(\Sigma^{\infty}X) \simeq \Sigma^{\infty}X^G$ for pointed $G$-spaces $X$.
	
	Let $H\subset G$ be normal. As mentioned above, $H$-fixed points define a functor $\Sp^{\infty}_G \to \Sp^{\infty}_{G/H}$. We want to define a similar version for geometric fixed points. Let $\cF[H]$ be the family of all subgroups of $G$ not containing $H$. We consider the functor
	\[\Phi^H\colon \Sp^{\infty}_G \to \Sp^{\infty}_{G/H}, \qquad X \mapsto (\EF[H] \sm X)^H.\]
	This agrees with our previous definition when $H= G$ since $\cF[G] = \cP$. Another important special case is $G=C_{2^n}$ and $H = C_2$; then $\EF[H] = \widetilde{E}G$.
	
	As the geometric fixed points functor $\Phi^H\colon \Sp^{\infty}_G \to \Sp^{\infty}_{G/H}$ is the composition of smashing with a space and taking fixed points, it preserves all homotopy colimits as well.
	
	This property implies that $\Phi^H$ must possess a right adjoint, which was constructed in \cite[Definition 4.1]{Hill:Slice} as the \emph{pullback functor}
	\[P^*_{G/H}\colon \Sp^{\infty}_{G/H} \to  \Sp^{\infty}_G, \qquad X \mapsto \EF[H] \sm p^*X,\]
	where $p^*$ is the functor induced by the projection $p\colon G \to G/H$, as defined in the previous subsection. For the adjointness see \cite[Proposition 4.4]{Hill:Slice}, at least on the level of homotopy categories. Several pleasant properties of $P^*_{G/H}$ are shown in \cite[Section 4.1]{Hill:Slice}, in particular that $P^*_{G/H}$ defines a fully-faithful embedding of $G/H$-spectra into $G$-spectra (with image agreeing with that of $- \sm \EF[H])$. Equivalently, the unit map $\mathrm{id} \to \Phi^HP^*_{G/H}$ is an equivalence. This also implies that $P^*_{G/H}$ is (strong) symmetric monoidal (since the image of $P^*_{G/H}$ is closed under $\sm$). Moreover, it follows that $P^*_{G/H}\Phi^H$ is equivalent to $- \sm \EF[H]$.
	
	We furthermore note:
	\begin{lem}\label{lem:homotopypullback}
		For every $G/H$-spectrum $X$, every $H\subset K \subset G$ and every $V \in RO(G/H)$, there is a canonical isomorphism
		\[\pi_V^K(P^*_{G/H}X) \cong \pi_V^{K/H}(X).\]
		Here we view $V$ also as an element of $RO(G)$ by pullback along $G\to G/H$.
	\end{lem}
\begin{proof}
	Essentially by definition, $\pi_V^K(P^*_{G/H}X) \cong [\Sigma^VG/K_+, P^*_{G/H}X]^G$. By the containment $H\subset K$, all points in $G/K$ are $H$-fixed and moreover $V^H =V$. Hence we get $\Phi^H \Sigma^VG/K_+ \simeq \Sigma^V(G/H)/(K/H)_+$. By the adjointness of $\Phi^H$ and $P^*_{G/H}$ we thus obtain the result.
\end{proof}

\subsection{Universal properties of $G$-spectra} 

In \cite[Corollary C.7]{GepnerMeier}, Gepner and the first-named author established a universal property for symmetric monoidal colimit-preserving functors out of $\Sp_G^{\infty}$. We will need a variant of this for functors just preserving \emph{filtered} colimits.

Localizing the $1$-category of pointed finite $G$-CW-complexes at $G$-homotopy equivalences yields an $\infty$-category $\SfinG$. This $\infty$-category is essentially small. For every essentially small $\infty$-category $\cC$, we can freely adjoin filtered colimits to obtain an $\infty$-category $\Ind(\cC)$ \cite[Section 5.3]{HTT}. The inclusion $\SfinG \to \SG$ into the $\infty$-category of pointed $G$-spaces induces a functor $\Ind(\SfinG) \to \SG$. Since $\SfinG$ consists of compact objects inside $\SG$ and generates $\SG$ under filtered colimits, the functor is an equivalence. 

Let us explain to obtain $G$-spectra and finite $G$-spectra as stabilization of $\SG$ and $\SfinG$ respectively. Let $\cU$ be a complete $G$-universe and denote by $\Sub_{\cU}$ the poset of finite-dimensional sub-representations. Following \cite[Appendix C]{GepnerMeier}, we can consider functors $\cT$ and $\cT^{\mathrm{fin}}$ from $\Sub_{\cU}$ to $\Cat_\infty^{\omega}$ (resp.\ $\Cat_{\infty}$), sending each $V\in \Sub_{\cU}$ to $\SG$ (resp.\ $\SfinG$) and each inclusion $V\subset W$ to smashing with $S^{W-V}$. Here, $\Cat_\infty^{\omega}$ is the $\infty$-category of compactly generated $\infty$-categories with compact object preserving left adjoints as morphisms, and $W-V$ is the orthogonal complement of $V$ in $W$. As explained in \cite[Appendix C]{GepnerMeier}, $\colim_{\Sub_{\cU}} \cT$ carries a canonical symmetric monoidal structure, which is as a symmetric monoidal $\infty$-category canonically equivalent to $\Sp^G_{\infty}$. Denote $\colim_{\Sub_{\cU}}\cT^{\mathrm{fin}}$ by $\Sp^{G, \mathrm{fin}}_{\infty}$. General properties of colimits in $\Cat_{\infty}^{\omega}$ (\cite[Proposition 5.5.7.10]{HTT}) imply that the functor $\Sp^{G, \mathrm{fin}}_{\infty} \to \Sp^G_{\infty}$ extends to an equivalence $\Ind(\Sp^{G, \mathrm{fin}}_{\infty}) \simeq \Sp^G_{\infty}$. This yields directly: 

\begin{lem}Let $\cD$ be an $\infty$-category with filtered colimits. The space of functors $\Sp_{\infty}^G \to \cD$ preserving filtered colimits is equivalent to that of functors $\Sp^{\mathrm{fin}, G}_{\infty} \to \cD$.
\end{lem}

\begin{rmk}
With our convention that $G$ is always finite, we could simplify the colimit $\colim_{\Sub_{\cU}} \cT$ to the colimit of the directed system 
\[\SG \xrightarrow{S^{\rho_G}} \SG \xrightarrow{S^{\rho_G}} \SG \xrightarrow{S^{\rho_G}} \cdots \] 
and similarly for $\SfinG$. For possible future applications, we chose however to present the proofs in this section in a way that applies to all compact Lie groups. 
\end{rmk}

We want to discuss a universal property of $\Sp^{\mathrm{fin}, G}_{\infty}$ using symmetric monoidal structures. For this, we need the following result of Robalo. Recall here that an object $X$ in a symmetric monoidal $\infty$-category is \emph{symmetric} if the cyclic permutation of $X \tensor X \tensor X$ is homotopic to the identity

\begin{prop}
Let $\cC$ be a small symmetric monoidal $\infty$-category and $X \in \cC$ symmetric. Then $\cC[X^{-1}] := \colim \cC \xrightarrow{X \tensor} \cC \xrightarrow{X \tensor}\cdots$ has a symmetric monoidal structure such that $\cC \to \cC[X^{-1}]$ refines to a symmetric monoidal functor, which is initial among all those that send $X$ to an invertible object. 
\end{prop}
\begin{proof}
The proof is the same as that of \cite[Corollary 2.22]{RobaloBridge}; all necessary previous results are actually proven for small $\infty$-categories and not just for presentable ones. 
\end{proof}

\begin{cor}
Let $\cD$ be a symmetric monoidal $\infty$-category. Then taking the suspension spectrum defines an equivalence between the space of symmetric monoidal functors $\Sp^{\mathrm{fin}, G}_{\infty} \to \cD$ and  the space of symmetric monoidal functors $\SfinG \to \cD$ sending $S^V$ for any $G$-representation to an invertible object.
\end{cor}
\begin{proof}This can be deduced from the previous proposition as in \cite[Corollary C.7]{GepnerMeier}
\end{proof}

\begin{cor}\label{cor:symmonfiltered}
Let $\cD$ be a symmetric monoidal $\infty$-category with filtered colimits. Then any symmetric monoidal functor $F\colon \Sp^G_{\infty} \to \cD$ which preserves filtered colimits is uniquely (up to equivalence) determined by its restriction $F\Sigma^{\infty}\colon \SfinG \to \cD$ (as a symmetric monoidal functor). 
\end{cor}

\begin{rmk}
The idea behind the preceding corollary is that we can write every $G$-spectrum canonically as a filtered homotopy colimit of $S^{-V} \sm \Sigma^{\infty} X$. We chose the above treatment to give a precise meaning to how canonical this colimit actually is. 
\end{rmk}
	
		\subsection{Norms and pullbacks}
	In this section, we will identify certain localizations of norm functors with pullbacks of norms from quotient groups. In the case of $\BPG$ this is a central ingredient of this paper.
	
	First, we will recall the norm construction. For a group $G$, let $\mathcal{B}G$ denote the category with one object and having $G$ as morphisms. Given an arbitrary symmetric monoidal category $(\cC, \otimes, \mathbbm{1})$, there is for a subgroup $H\subset G$ a norm functor
	\[\cC^{\mathcal{B}H} \to \cC^{\mathcal{B}G}, \qquad X \mapsto X^{\tensor_H G}\]
	from $H$-objects to $G$-objects, where the $G$-action is induced by the right $G$-action on $G$. In the case of spaces or sets, one can identify $X^{\times_H G}$ with $\Map_H(G, X)$ and for based spaces or sets, one can likewise identify $X^{\sm_H G}$ with $\Map^{\ast}_H(G, X)$. In the case of orthogonal spectra, one can by \cite[Proposition B.105]{HHR} left derive the functor $(-)^{\sm_H G}$ to obtain a functor $N_H^G$. (Often, $N_H^G$ is also used for the corresponding underived functor, but the derived functor will be more important for us.) The functor $N_H^G$ commutes with filtered (homotopy) colimits by \cite[Propositions A.53, B.89]{HHR}. Note moreover that $N_H^G \Sigma^{\infty}X \simeq \Sigma^{\infty}\Map_H^*(G,X)$ (if $X$ is cofibrant or at least well-pointed) as $\Sigma^{\infty}$ is symmetric monoidal.
	
	\begin{lem}\label{lem:normfixedpoints}
		Let $G$ be a finite group, $K,H\subset G$ be two subgroups and $X$ be a (based) topological $H$-space. Let $ H\backslash G/K = \{Hg_1K, \dots, Hg_lK\}$. Then there are natural (based) homeomorphisms
		\[\Map_H(G, X)^K \cong X^{g_1Kg_1^{-1}\cap H} \times \cdots \times X^{g_lKg_l^{-1}\cap H} \]
		and
		\begin{equation}\label{eq:MapHgK}\Map_H^{\ast}(G, X)^K \cong X^{g_1Kg_1^{-1}\cap H} \sm \cdots \sm X^{g_lKg_l^{-1}\cap H},
		\end{equation}
		where the $K$-action on the mapping spaces is induced by the right $K$-action on $G$. In particular, if $H=K$ is normal, we obtain a natural $G/H$-equivariant homeomorphism
		\[\Map^*_H(G,X)^H \cong \Map^*(G/H,X^H).\]
	\end{lem}
	\begin{proof}
		The first two statements follow from the $H$-$K$-equivariant decomposition of $G$ into $\coprod_{i=1}^{l} Hg_iK$. For the last one observe that if $H =K$ is normal, $H\backslash G/K = G/H$ and $G/H$ permutes the factors of the decomposition in \eqref{eq:MapHgK}.
	\end{proof}
	
	To put the following theorem and its corollary into context, recall from \cite[Proposition B.213]{HHR} that $\Phi^GN_K^GX \simeq \Phi^KX$. We show more generally that $\Phi^HN_K^GX \simeq N_e^{G/H}\Phi^KX$ if $K\subset H\subset G$ and $H$ is normal. Before we do so in \cref{cor:geomfixed}, we provide a version that gives an equivalence on the level of $G$-spectra, i.e.\ before taking fixed points. 
 	\begin{thm}\label{prop-main}
 		Let $H \subset G$ be a normal subgroup and $X$ be an $H$-spectrum. Then we have an equivalence of $G$-spectra
 		\[
 		\EF[H] \wedge N_H^G X \simeq P^*_{G/H}(N_e^{G/H} \Phi^H(X)).
 		\]
 		\end{thm}
	\begin{proof}
	We have 
	\begin{equation}\label{eq:PhiN}\Phi^HN_H^G X \simeq N_e^{G/H}\Phi^HX\end{equation}
	for all $H$-spectra $X$. Indeed: If $X$ is a suspension spectrum, this reduces to the space-level statement $\Map_H^{\ast}(G, X)^H \simeq \Map^{\ast}(G/H, X^H)$, which is part of \cref{lem:normfixedpoints}. Both sides of \eqref{eq:PhiN} are symmetric monoidal and commute with filtered homotopy colimits. Thus \cref{cor:symmonfiltered} implies the claim. 
	
	Applying $P^*_{G/H}$ to \eqref{eq:PhiN}, it suffices to check that $P^*_{G/H}\Phi^HN_H^GX$ is equivalent to $\EF[H] \wedge N_H^GX$. But the equivalence of $P^*_{G/H}\Phi^H$ with $\EF[H]\sm -$ was already noted above. \end{proof}	

	\begin{cor}\label{cor-pullback}
		Let $K \subset H \subset G$ be subgroups and assume that $H \subset G$ is normal. Let moreover $X$ be a $K$-spectrum. Then there is an equivalence of $G$-spectra
		\[
		\EF[H] \wedge N_K^G X \simeq P^{*}_{G/H}(N_e^{G/H} \Phi^K(X)).
		\]
	\end{cor}
	\begin{proof}
		This follows from \cref{prop-main} by applying it to $N_K^HX$. Here, we use $N_K^G X \simeq N_H^G N_K^H X$ and  $\Phi^{H}N_K^H X \simeq \Phi^{K}X$.
	\end{proof}
	
	Taking $H$-fixed points we obtain a strengthened form of \cref{thm:introthm1}:
	
	\begin{cor}\label{cor:geomfixed}
	    	Let $K \subset H \subset G$ be subgroups and assume that $H \subset G$ is normal. Let moreover $X$ be a $K$-spectrum. Then there is an equivalence of $G/H$-spectra
	    	\[\Phi^HN_K^GX \simeq N_e^{G/H} \Phi^K(X). \]
	\end{cor}
	
 	\begin{rmk}
 	An alternative proof of this result is possible using \cite[Theorem 2.7]{Yuan:Frobenius}.
 	\end{rmk}

As we will recall below, there is a $C_2$-spectrum $\BPR$ with geometric fixed points $H\FF_2$. For $G = C_4$ and $H=C_2$, we can express $\EF[H]$ as $S^{\infty \lambda}$, where $\lambda$ is the $2$-dimension representation of $C_4$ corresponding to rotation by an angle of $\frac{\pi}2$. Denoting the norm $N_{C_2}^{C_4}\BPR$ by $\BPfour$, we obtain our main example for \cref{prop-main}.
\begin{cor}\label{cor:main}
	There is an equivalence
	\[\BPfour \sm S^{\infty \lambda} \simeq P^*_{C_4/C_2} N_1^2(H\FF_2).\]
\end{cor}
	
	We end this section with a different kind of compatibility of norms and pullbacks. 
	
	\begin{prop}\label{prop:normpullback}
	    Let $K\subset H \subset G$ be subgroups such that $K$ is normal in $G$. Then there is a natural equivalence $N_H^GP_{H/K}^* \simeq P^*_{G/K}N_{H/K}^{G/K}$ of functors $\Sp^{\infty}_{H/K} \to \Sp^{\infty}_G$.
	\end{prop}
	\begin{proof}
	    Since both $N_H^GP_{H/K}^*$ and $P^*_{G/K}N_{H/K}^{G/K}$ commute with filtered colimits and are symmetric monoidal, it suffices (as in the proof of \cref{prop-main}) to provide a natural equivalence of their restriction to suspension spectra. We compute
	    \[N_H^GP_{H/K}^*\Sigma^{\infty}X \simeq  N_H^G \Sigma^{\infty}\EF[H] \sm X \simeq \Sigma^{\infty} \Map_H^*(G,\EF_H[K]) \sm \Map_H^*(G, X)\]
	    and 
	    \[P^*_{G/K}N_{H/K}^{G/K}\Sigma^{\infty} X \simeq P^*_{G/K} \Sigma^{\infty}\Map_{H/K}^*(G/K, X) \simeq \Sigma^{\infty} \EF_G[K] \sm \Map_H^*(G, X),\]
	    where we used a subscript at $\EF$ to indicate whether it is a $G$-space or an $H$-space. Using \cref{lem:normfixedpoints}, one can check that $\Map_H^*(G,\EF_H[K])^L \simeq S^0$ if $L\subset K$ and is contractible otherwise; thus indeed $\Map_H^*(G,\EF_H[K]) \simeq \EF_G[K]$. 
	\end{proof}

	\subsection{Multiplicative structures of localizations}\label{sec:MultLoc}
	
	In many cases, smashing with $\EF[H]$ is equivalent to localizing at a certain element in $\pi_{\bigstar}^G\mathbb{S}$ (for example if $G$ is cyclic). The goal of this section is to investigate which kind of multiplicative structure localization at such an element preserves. More specifically let us fix an $N_{\infty}$-operad $\cO$, i.e.\ an operad $\cO$ in (unbased) $G$-spaces such that each $\cO(n)$ is a universal space for a family $\cF_n$ of graph subgroups of $G\times \Sigma_n$, containing all $H\times \{e\}$ for subgroups $H\subset G$. This notion was introduced in \cite{BlumbergHill}. In the maximal case, we speak of a \emph{$G$-$E_{\infty}$-operad} and by \cite[Theorem A.6]{BlumbergHill} every algebra over such an operad can be strictified to a commutative $G$-spectrum. In the minimal case, we speak of a \emph{(naive) $E_{\infty}$-operad}.
	
	Essentially, the different versions of $N_{\infty}$-operads encode which norms we see in the homotopy groups of an $\cO$-algebra. To be more precise, call an $H$-set $T$ \emph{admissible} if the graph of the $H$-action on $T$ lies in $\cF_{|T|}$. By \cite[Remark 5.15]{GradedTambara} an $\cO$-algebra $R$ admits norms $N_K^H\colon \pi_V^{K}R \to \pi_{\ind_K^HV}^HR$ if $H/K$ is admissible, and the groups $\pi_{\bigstar}^HR$ assemble into an $RO(G)$-graded incomplete Tambara functor.
	
	As already observed in \cite{McClure}, localizations only need to preserve naive $E_{\infty}$-structures, but not $G$-$E_{\infty}$-structures. Later, \cite{HillHopkins} gave a criterion when localizations indeed preserve $G$-$E_{\infty}$-structures and this was extended in \cite{Boehme} to $N_{\infty}$-algebras, albeit only for localizations of elements in degree $0$. In this section, we will extend this work to elements in non-trivial degree and follow the proof strategy of \cite[Proposition 2.30]{Boehme}.
	
	Let us first recall what localizing at some $x\in \pi_{V}^G\mathbb{S}$ means. We say that a $G$-spectrum $E$ is \emph{$x$-local} if $x$ acts invertibly on $E$ or, equivalently, on $\pi_{\bigstar}^GE$. Given a $G$-spectrum $E$, we construct its \emph{$x$-localization} as
	\[x^{-1}E = \hocolim \left(E \xrightarrow{x} \Sigma^{-V}E \xrightarrow{x} \Sigma^{-2V} E \xrightarrow{x} \cdots\right).\]
	Note that $x^{-1}E \simeq E \wedge x^{-1}\mathbb{S}$.
	\begin{exam}\label{exam:main}\rm
		Given a $G$-representation $V$, let $a_V\colon S^0 \to S^V$ be the Euler class. Then $a_V^{-1}\mathbb{S} \simeq S^{\infty V}$ and hence in general $a_V^{-1}E \simeq S^{\infty V} \sm E$. In particular, we can reformulate \cref{cor:main} as
		\[a_{\lambda}^{-1}\BPfour \simeq P^*_{C_4/C_2} N_1^2(H\FF_2).\]
	\end{exam}
	
	A map $f\colon E \to F$ is an \emph{$x$-local equivalence} if $f\wedge x^{-1}\mathbb{S}$ is an equivalence; by abuse of notation, we call for $H\subset G$ a map of $H$-spectra an $x$-equivalence if it is a $\Res_H^G(x)$-equivalence.
	
	\begin{defi}\label{def:localizationO}
		Localization at $x$ \emph{preserves $\cO$-algebras} if for every $\cO$-algebra $R$, we can lift the morphism $R \to x^{-1}R$ in $\Ho(\Sp^G)$ (up to isomorphism) to a morphism in $\Ho(\cO-\mathrm{Alg})$.
	\end{defi}
	
	We will use the following specialization of a criterion of \cite[Corollary 7.10]{GW18}:
	\begin{prop}\label{prop:GW}
		Localization at $x$ preserves $\cO$-algebras if and only if \[N_K^H\Res_K^G\colon \Sp^G_{\infty} \to \Sp^H_{\infty}\]
		preserves $x$-equivalences for every $K \subset H \subset G$ such that $H/K$ is admissible as an $H$-set.
	\end{prop}
	
	To reformulate this criterion, we need the following lemma.
	\begin{lem}\label{lem:NormResSphere}
		There is an equivalence $N_K^H\Res_K^G(x^{-1}\mathbb{S}) \simeq (N_K^H\Res_K^G(x))^{-1}(\mathbb{S}_H)$ for $\mathbb{S}_H$ the $H$-equivariant sphere spectrum.
	\end{lem}
	
	\begin{proof} Applying $N_K^H\Res_K^G$ to
		\[\mathbb{S} \xrightarrow{x} \Sigma^{-V}\mathbb{S} \xrightarrow{x} \Sigma^{-2V} \mathbb{S} \xrightarrow{x} \cdots,\]
		 we obtain precisely
		\[\mathbb{S}_H \xrightarrow{N_K^H\Res_K^G(x)} \Sigma^{-\ind_K^H\Res_K^GV}\mathbb{S}_H \xrightarrow{N_K^H\Res_K^G(x)} \Sigma^{-2\ind_K^H\Res_K^GV} \mathbb{S}_H \xrightarrow{N_K^H\Res_K^G(x)} \cdots\]
		Here we have used that the norm of a representation sphere is computed by induction. As both $N_K^H$ and $\Res_K^G$ preserve filtered homotopy colimits, the result follows. 
	\end{proof}
	
	\begin{prop}\label{prop-Ninf}
		Localization at $x$ preserves $\cO$-algebras if and only if $N_K^H\Res_K^G(x)$ divides a power of $\Res_H^G(x)$ for every $K \subset H \subset G$ such that $H/K$ is admissible as an $H$-set.
	\end{prop}
	\begin{proof}
		Let $K \subset H \subset G$ be subgroups such that $H/K$ is admissible as an $H$-set. By \cref{prop:GW}, we have to show that $N_K^H\Res_K^G(x)$ divides a power of $\Res_H^G(x)$ if and only if
		\[N_K^H\Res_K^G\colon \Sp^G_{\infty}\to \Sp^H_{\infty}\]
		preserves $x$-equivalences.
		
		Assume first that $N_K^H\Res_K^G$ preserves $x$-equivalences. By the preceding lemma, we see in particular that $\mathbb{S}_H \to  (N_K^H\Res_K^G(x))^{-1}\mathbb{S}_H$ is an $x$-equivalence, i.e.\ $N_K^H\Res_K^G(x)$ becomes a unit after inverting $\Res_H^G(x)$ and just must divide a power of it.
		
		Assume now that $N_K^H\Res_K^G(x)$ divides a power of $\Res_H^G(x)$. Then the map $\mathbb{S}_H \to \Res_H^G(x)^{-1}\mathbb{S}_H$ factors over the standard map $\mathbb{S}_H \to (N_K^H\Res_K^G(x))^{-1}\mathbb{S}_H$.
		
		Let now $f\colon E \to F$ be an $x$-equivalence of $G$-spectra, i.e.\ we assume that $f \wedge x^{-1}\mathbb{S}$ is an equivalence. As $N_K^H$ and $\Res_H^G$ are symmetric monoidal, we see that $N_K^H\Res_H^G(f \wedge x^{-1}\mathbb{S})$ is equivalent to $N_K^H\Res_H^G(f) \wedge (N_K^H\Res_K^G(x))^{-1}\mathbb{S}_H$, which is thus an equivalence. Tensoring with $\Res_H^G(x)^{-1}\mathbb{S}_H$ over $(N_K^H\Res_K^G(x))^{-1}\mathbb{S}_H$ yields the result.
	\end{proof}
	
	We specialize now to the case that $x$ is the Euler class $a_V\colon S^0 \to S^V$. In this case we have $N_K^G\Res_H^Ga_V = a_{\ind_K^G\Res_H^G V}$. Thus to see which multiplicative structure localization at $a_V$ preserves, we only have to understand divisibility relations between Euler classes. In particular, we obtain the following corollary:
	\begin{cor}
		Let $V$ be a $G$-representation. Assume that $\ind_K^H\Res_K^GV$ is a summand of a multiple of $\Res_H^GV$ for every $K\subset H \subset G$ such that $H/K$ is an admissible $H$-set. Then localizing at $a_V$ preserves $\cO$-algebras.
	\end{cor}
	
	\begin{rmk}\label{rmk-DividingEulerClasses}\rm
	While this corollary is everything we need, one can be more precise. For a $H$-representation $V$, let $\cF_V^{\mathrm{fix}}$ be the family of subgroups $K\subset H$ such that $V^K \neq 0$. Thus, $a_V^{-1}\mathbb{S}_H \simeq S^{\infty V} \simeq \widetilde{E\cF_V^{\mathrm{fix}}}$. In general, $a_W$ divides a power of $a_V$ if and only if $a_W^{-1}a_V^{-1}\mathbb{S}_H \simeq a_V^{-1}\mathbb{S}_H$, i.e.\ if $\cF_W^{\mathrm{fix}} \subset \cF_V^{\mathrm{fix}}$. (This is a weaker condition than $W$ being contained in a multiple of $V$: for example, take $G = C_8$ and $W$ and $V$ be the two-dimensional real representation corresponding to rotation by $\frac18\cdot 2\pi$ and $\frac38\cdot 2\pi$, respectively, which both have trivial fixed point family.)
	
	Specializing \cref{prop-Ninf} thus yields: For a $G$-representation $V$, localization at $a_V$ preserves $\cO$-algebras if and only if $\cF_{\ind_K^H\Res_K^GV}^{\mathrm{fix}} \subset \cF^{\mathrm{fix}}_{\Res_H^GV}$ for all $K\subset H \subset G$ such that $H/K$ is an admissible $H$-set. 
	\end{rmk}
	
	\begin{exam}\label{exam-MU}\rm
		Let $G = C_{2^n}$ and $\lambda = \lambda^n$ be the two-dimensional representation of $C_{2^n}$ given by rotation by an angle of $\frac{2\pi}{2^n}$. We observe that $\Res_{C_{2^k}}^{C_2^{n}}\lambda^n = \lambda^k$ and $\ind_{C_{2^k}}^{C_2^{m}}\lambda^k = 2^{m-k}\lambda^m$ unless $k=1$. Thus localizing at $a_{\lambda}$ preserves $\cO$-algebras if the following holds: $H/K$ is $H$-admissible if and only if $K\neq e$. In particular, we see that for any commutative $C_{2^n}$-spectrum $R$, the localization $a_{\lambda}^{-1}R$ admits norms from $\pi_{\ast}^{C_{2^k}}$ to $\pi_{\ast}^{C_{2^n}}$ for $0 < k < n$, but will not admit norms from $\pi_{\ast}^{e}$ unless the target is zero. The example we care most about is $a_{\lambda}^{-1} MU^{(\!(C_{2^n})\!)}$.
		
		These considerations have consequences for the multiplicative behaviour of the pullback functor $P^*_{C_{2^n}/C_2}$. Let $R$ be an algebra over a $C_{2^n}/C_2$-$E_{\infty}$-operad $\cO$ in $C_{2^n}/C_2$-spectra. Denoting the projection $C_{2^n} \to C_{2^n}/C_2$ by $p$, we see that $p^*\cO$ is an $N_{\infty}$-operad for which $\Gamma \subset C_{2^n} \times \Sigma_n$ is in $\cF_n$ if and only if $\Gamma/(C_2 \times e)$ is a graph subgroup of $(C_{2^n}/C_2) \times \Sigma_n$. This means that $H/K$ is $H$-admissible if and only if $K \neq e$. Note further that 
		\[P_{C_{2^n}/C_2}^*R = p^*R\sm \widetilde{E}\cF[C_2] \simeq p^*R[a_{\lambda}^{-1}]\] 
		since $\lambda^K = 0$ unless $K$ is trivial. Using the paragraph above we see that $P_{C_{2^n}/C_2}^*R$ retains the structure of a $p^*\cO$-algebra.
		
		Likewise we can apply our considerations to the geometric fixed point functor. With $p^*\cO$ as above, we see that for a $G$-commutative ring spectrum $R$, the localization $a_{\lambda}^{-1}R$ retains an action of $p^*\cO$ and thus $\Phi^{C_2}R \simeq (a_{\lambda}^{-1}R)^{C_2}$ has the structure of a $\cO$-algebra. Thus $\Phi^{C_2}R$ is equivalent to a $G/C_2$-commutative ring spectrum. 
	\end{exam}
	
\section{The slice spectral sequence and the localized slice spectral sequence}\label{sec:sliceSSBackground}

\subsection{The slice spectral sequence of $\MUn$ and $\BPn$} \label{subsec:sliceSS}
Our main computational tool in this paper is a modification of the equivariant slice spectral sequence of Hill--Hopkins--Ravenel.  In this subsection, we list some important facts about the slice filtration for norms of $\MUR$ and $\BPR$, which we will need for the rest of the paper.  For a detailed construction of the slice spectral sequence and its properties, see \cite[Section 4]{HHR} and \cite{HHR:C4}. 

Let $G = C_{2^n}$ be the cyclic group of order $2^n$, with generator $\gamma$.  The spectrum $\MUG$ is defined as
$$\MUG := N_{C_2}^{G} \MUR.$$
The underlying spectrum of $\MUG$ is the smash product of $2^{n-1}$-copies of $MU$.

Hill, Hopkins, and Ravenel \cite[Section 5]{HHR} constructed elements
$$\ovr_i \in \pi_{i\rho_2}^{C_2} \MUG$$
such that
$$\pi_{*\rho_2}^{C_2} \MUG \cong \mathbb{Z}[G \cdot \ovr_1, G \cdot \ovr_2, \ldots],$$
Here, $G \cdot x$ denotes the set $\{x, \gamma x, \gamma^2 x, \ldots, \gamma^{2^{n-1}}x\}$, and the Weyl action is given by
$$\gamma \cdot \gamma^j \ovr_i = \left\{\begin{array}{ll}\gamma^{j+1} \ovr_i & 0 \leq j \leq 2^{n-1}-2 \\ (-1)^i \ovr_i & j = 2^{n-1}-1. \end{array} \right.  $$

Adjoint to each map
$$\ovr_i: S^{i\rho_2} \longrightarrow i_{C_2}^*\MUG$$
is an associative algebra map from the free associative algebra
$$S^0[\ovr_i] = \bigvee_{j\geq 0} (S^{i\rho_2})^{\sm j} \longrightarrow i_{C_2}^*\MUG.$$
Applying the norm and using the norm-restriction adjunction, this gives a $G$-equivariant associative algebra map
$$S^0[G \cdot \ovr_i] = N_{C_2}^G S^0[\ovr_i] \longrightarrow \MUG.$$
Smashing these maps together produces an associative algebra map
$$A:= S^0[G \cdot \ovr_1, G \cdot \ovr_2, \ldots] = \bigwedge_{i =1}^\infty S^0[G \cdot \ovr_i] \longrightarrow \MUG.$$
Note that by construction, $A$ is a wedge of representation spheres, indexed by monomials in the $\ovr_i$s.  By the Slice Theorem \cite[Theorem~6.1]{HHR}, the slice filtration of $\MUG$ is the filtration associated with the powers of the augmentation ideal of $A$.  The slice associated graded for $\MUG$ is the graded spectrum
$$S^0[G \cdot \ovr_1, G \cdot \ovr_2, \ldots]\wedge \HZ,$$
where the degree of a summand corresponding to a monomial in the $\ovr_i$ generators and their conjugates is the underlying degree.

As a consequence of the slice theorem, the slice spectral sequence for the $RO(G)$-graded homotopy groups of $\MUG$ has $E_2$-term the $RO(G)$-graded homology of $S^0[G \cdot \ovr_1, G \cdot \ovr_2, \ldots]$ with coefficients in the constant Mackey functor $\underline{\mathbb{Z}}$.  To compute this, note that $S^0[G \cdot \ovr_1, G \cdot \ovr_2, \ldots]$ can be decomposed into a wedge sum of slice cells of the form
$$G_+ \wedge_{H_p} S^{\frac{|p|}{|H_p|} \rho_{H_p}},$$
where $p$ ranges over a set of representatives for the orbits of monomials in the $\gamma^j \ovr_i$ generators, and $H_p \subset G$ is the stabilizer of $p$ (mod 2).  Therefore, the $E_2$-page of the integer graded slice spectral sequence can be computed completely by writing down explicit equivariant chain complexes for the representation spheres $S^{\frac{|p|}{|H_p|} \rho_{H_p}}$.

The exact same story holds for norms of $\BPR$ as well. By \cite[Theorems 2.25, 2.33]{HuKriz1}, the classical Quillen idempotent $MU \longrightarrow MU$ lifts to a multiplicative idempotent $\MUR \to \MUR$ with image $\BPR$, resulting in particular in a multiplicative $C_2$-equivariant map
$$\MUR \longrightarrow \BPR.$$
Taking the norm $N_{C_2}^G(-)$ of this map produces a multiplicative $G$-equivariant map
$$\MUG \longrightarrow \BPG =: N_{C_2}^G \BPR.$$
The exact same technique as the one used in \cite[Section 5]{HHR} shows that there are generators
$$\ot_i \in \pi_{(2^i-1)\rho_2}^{C_2} \BPG$$
such that
$$\pi_{*\rho_2}^{C_2} \BPG \cong \mathbb{Z}_{(2)}[G \cdot \ot_1, G \cdot \ot_2, \ldots]. $$

Throughout the paper, the generators $\ot_i$ are chosen to be the coefficients of the canonical isomorphism from the formal group law of the first $\BPR$ component to the formal group law of the second $\BPR$-component.  In the case when $G = C_4$, it is the canonical isomorphism from the formal group law $F_L$ to $F_R$, where $F_L$ is induced by the map
\[\BPR \simeq \BPR \wedge S^0 \longrightarrow \BPR \wedge \BPR,\]
and $F_R$ is induced by the map 
\[\BPR \simeq S^0 \wedge \BPR \longrightarrow \BPR \wedge \BPR.\]

\begin{rmk}\rm
Our specific choice of the formal group law and the generators $\ot_i$ is because we would like to control their geometric fixed points (See \cref{prop:ti}).  Nevertheless, we would like to remark that the proofs and formulas in both \cite{HHR} and \cite{BHSZ} work for any choice of formal group law and the corresponding $\bar{t}_i$ generators we get for $\pi_{*\rho_2}^{C_2}\BPG$, as long as the conditions in \cite[Proposition~5.45]{HHR} are satisfied. 
\end{rmk}

Just like $\MUG$, we can build an equivariant refinement
$$S^0[G \cdot \ot_1, G \cdot \ot_2, \ldots] \longrightarrow \BPG$$
from which the Slice Theorem implies that the slice associated graded for $\BPG$ is the graded spectrum
$S^0[G \cdot \ot_1, G \cdot \ot_2, \ldots] \wedge \HZ_{(2)}.$

\begin{rmk}\rm
The regular slice towers of $\MUG$ and $\BPG$ are isomorphic to their slice towers in \cite{HHR}, since all the HHR-slices of them are regular slices. For a proof of the slice theorem in terms of the regular slice filtration, see \cite[Chapter~12.4]{HHR:ESHT}
\end{rmk}

Since the slice filtration is an equivariant filtration, the slice spectral sequence is a spectral sequence of $RO(G)$-graded Mackey functors.  Moreover, the slice spectral sequences for $\MUG$ and $\BPG$ are multiplicative spectral sequences and the natural maps between them are multiplicative as well (see \cite[Section 4.7]{HHR}), and the slice spectral sequence for $\BPG$ is a spectral sequence of modules over the spectral sequence of $\MUG$ in Mackey functors.

\subsection{The localized spectral sequence}
In this subsection, we introduce a variant of the slice spectral sequence which we call the localized slice spectral sequence.  This will be our main computational tool to compute $a_\lambda^{-1} \BPfour$ in the later sections.

Let $\lambda_{2^{n-i}}$ denote the 2-dimensional real $C_{2^n}$-representation corresponding to rotation by $\left(\frac{\pi}{2^{n-i}}\right)$ and $\sigma$ denote the real sign representation of $C_{2^n}$.  Given a $C_{2^n}$-spectrum $X$, we have an equivalence
$$\EFi \wedge X \simeq S^{\infty \lambda_{2^{n-i}}} \wedge X \simeq a_{\lambda_{2^{n-i}}}^{-1} X$$
for all $1 \leq i \leq n$.  For example, there are equivalences
\begin{eqnarray*}
\widetilde{E}\mathcal{F}[C_{2^n}] \wedge X &\simeq& a_{\lambda_1}^{-1}X = a_{2\sigma}^{-1} X = a_\sigma^{-1} X, \\
\widetilde{E}\mathcal{F}[C_{2^{n-1}}] \wedge X &\simeq& a_{\lambda_2}^{-1}X,\\
\widetilde{E}\mathcal{F}[C_{2^{n-2}}] \wedge X &\simeq& a_{\lambda_4}^{-1}X.
\end{eqnarray*}

The following theorem shows that one can compute the homotopy groups of $\EFi \wedge X = a_{\lambda_{2^{n-i}}}^{-1}X$ by smashing the slice tower of $X$ with $\EFi$.  The resulting localized slice spectral sequence will converge to the homotopy groups of $a_{\lambda_{2^{n-i}}}^{-1}X$.

\begin{thm}\label{thm:invertingAclassE2page}
Let $X$ be a $C_{2^n}$-spectrum, and let $\{P^\bullet\}$ denote the (regular) slice tower for $X$.  Consider the tower
$$\{Q^\bullet\} :=\{\EFi \wedge P^\bullet \}$$
obtained by smashing $\{P^\bullet\}$ with $\EFi$.  The spectral sequence associated to $\{Q^\bullet\}$ converges strongly to the homotopy groups of $\EFi \wedge X$.
\end{thm}

\begin{proof}
Let $\lambda:= \lambda_{2^{n-i}}$.  Consider the tower
$$\begin{tikzcd}
S^{\infty \lambda} \wedge X \ar[r] &\varprojlim (S^{\infty \lambda} \wedge P^\bullet X) \ar[d] \\
&\vdots \ar[d] \\
&S^{\infty \lambda} \wedge P^n X \ar[d] & \ar[l] S^{\infty \lambda} \wedge P^n_n X\\
&S^{\infty \lambda} \wedge P^{n-1} X \ar[d] & \ar[l] S^{\infty \lambda} \wedge P^{n-1}_{n-1} X\\
&\vdots
\end{tikzcd}$$
We will first show that the spectral sequence converges to the limit, $\varprojlim(S^{\infty \lambda}\wedge P^\bullet X)$.  Since smash products commute with colimits, we have the equivalence
$$\varinjlim(S^{\infty \lambda}\wedge P^\bullet X) \simeq *$$
and so the colimit of the tower is contractible.  The slices $P_n^n X$ satisfy $P_n^n X \geq n$ for all $n$.  Furthermore, since $S^{\infty \lambda} \geq 0$, we also have
$$S^{\infty \lambda} \wedge P_n^n X \geq n$$
by \cite[Proposition~4.26]{HHR}.\footnote{The proof of this result and of the part of \cite[Proposition~4.40]{HHR} we need are still valid for the regular slice filtration instead of the slice filtration as used in \cite{HHR}.}  Applying Proposition~4.40 in \cite{HHR} to $S^{\infty \lambda} \wedge P_n^n X$ shows that the homotopy groups
$$\underline{\pi}_k (S^{\infty \lambda} \wedge P_{n}^n X) = 0 \,\,\, \text{if } \left\{\begin{array}{l}n \geq 0 \text{ and } k < \lfloor\frac{n}{|G|} \rfloor, \\
n < 0 \text{ and } k < n. \end{array} \right.$$
This gives a vanishing line on the $E_2$-page of the spectral sequence.  It follows that the spectral sequence converges strongly to the homotopy groups of the limit, $\underline{\pi}_k \varprojlim (S^{\infty \lambda} \wedge P^\bullet X)$ \cite[Section~5-6]{Board:SS}.

To finish our proof, it suffices to show that the map
$$S^{\infty \lambda} \wedge X \longrightarrow \varprojlim (S^{\infty \lambda} \wedge P^\bullet X)$$
is an equivalence.

Consider the cofiber sequence
$$P_{n+1}X \longrightarrow X \longrightarrow P^n X$$
used in the definition of the slice tower.  In the cofiber sequence, $P_{n+1} X \geq n+1$ and $P^n X \leq n$.  Smashing this cofiber sequence with $S^{\infty \lambda}$ produces a new cofiber sequence
$$S^{\infty \lambda} \wedge P_{n+1}X \longrightarrow S^{\infty \lambda} \wedge X \longrightarrow S^{\infty \lambda} \wedge  P^n X.$$
Since $S^{\infty \lambda} \geq 0$, \cite[Proposition~4.26]{HHR} implies that
$$S^{\infty \lambda} \wedge P_{n+1} X \geq  n+1. $$
Applying \cite[Proposition~4.40]{HHR} to $S^{\infty \lambda} \wedge P_{n+1} X$ shows that
$$\underline{\pi}_k (S^{\infty \lambda} \wedge P_{n+1} X) = 0 \,\,\, \text{if } \left\{\begin{array}{l}n+1 \geq 0 \text{ and } k < \lfloor\frac{n+1}{|G|} \rfloor, \\
n+1 < 0 \text{ and } k < n+1. \end{array} \right.$$

The cofiber sequence above induces the following long exact sequence in homotopy groups:
$$\underline{\pi}_k(S^{\infty \lambda} \wedge P_{n+1}X) \longrightarrow \underline{\pi}_k(S^{\infty \lambda} \wedge X) \longrightarrow \underline{\pi}_k(S^{\infty \lambda} \wedge  P^n X) \longrightarrow \underline{\pi}_{k-1}(S^{\infty \lambda} \wedge P_{n+1}X) \longrightarrow \cdots $$
It follows from this long exact sequence and the discussion above that
$$\underline{\pi}_k(S^{\infty \lambda} \wedge X) \cong \underline{\pi}_k(S^{\infty \lambda} \wedge  P^n X)\,\,\, \text{if } \left\{\begin{array}{l}n+1 \geq 0 \text{ and } k < \lfloor\frac{n+1}{|G|} \rfloor, \\
n+1 < 0 \text{ and } k < n+1. \end{array} \right. $$
This means that for any $k$, the $k$th homotopy groups of $S^{\infty \lambda} \wedge X$ and $S^{\infty \lambda} \wedge  P^n X$ will be isomorphic when $n$ is large enough.  In particular, the map $S^{\infty \lambda} \wedge P^{n+1} X \to S^{\infty \lambda} \wedge P^{n}X$ will induce an isomorphism on $\underline{\pi}_k$.  It is then immediate that the system $\underline{\pi}_k(S^{\infty \lambda} \wedge P^\bullet X)$ satisfies the Mittag--Leffler condition and therefore 
\[\underline{\pi}_k \varprojlim (S^{\infty \lambda} \wedge  P^{\bullet} X) \cong \varprojlim \underline{\pi}_k (S^{\infty \lambda} \wedge  P^{\bullet} X) \cong \underline{\pi}_k (S^{\infty \lambda} \wedge  P^n X)\]
for $n$ large. 

Another way to observe this is by using the localized slice spectral sequence.  As we have shown, the spectral sequence associated to the tower $\{Q^\bullet\} :=\{S^{\infty \lambda} \wedge P^\bullet \}$ converges to the homotopy groups of $\varprojlim (S^{\infty \lambda} \wedge P^\bullet X)$.  It takes the form
$$E_2^{s,n} = \underline{\pi}_{n-s} (S^{\infty \lambda} \wedge P_n^n X) \Longrightarrow \underline{\pi}_{n-s} \varprojlim (S^{\infty \lambda} \wedge P^\bullet X).$$
By \cite[Proposition~4.40]{HHR}, the homotopy groups
$$\underline{\pi}_{n-s} (S^{\infty \lambda} \wedge P_n^n X)$$
do not contribute to $\underline{\pi}_k \varprojlim (S^{\infty \lambda} \wedge P^\bullet X)$ when $n \geq 0$ and $k < \lfloor\frac{n}{|G|} \rfloor$, or when $n < 0$ and $k < n$ (see Figure~\ref{fig:SliceSSVanishing}).  Therefore,
$$\underline{\pi}_k \varprojlim (S^{\infty \lambda} \wedge P^\bullet X) \cong \underline{\pi}_k(S^{\infty \lambda} \wedge  P^n X)\,\,\, \text{if } \left\{\begin{array}{l}n \geq 0 \text{ and } k < \lfloor\frac{n}{|G|} \rfloor, \\
n < 0 \text{ and } k < n. \end{array} \right.$$

\begin{figure}
\begin{center}
\makebox[\textwidth]{\includegraphics[trim={0cm 0cm 0cm 2cm},clip,scale = 0.6, ]{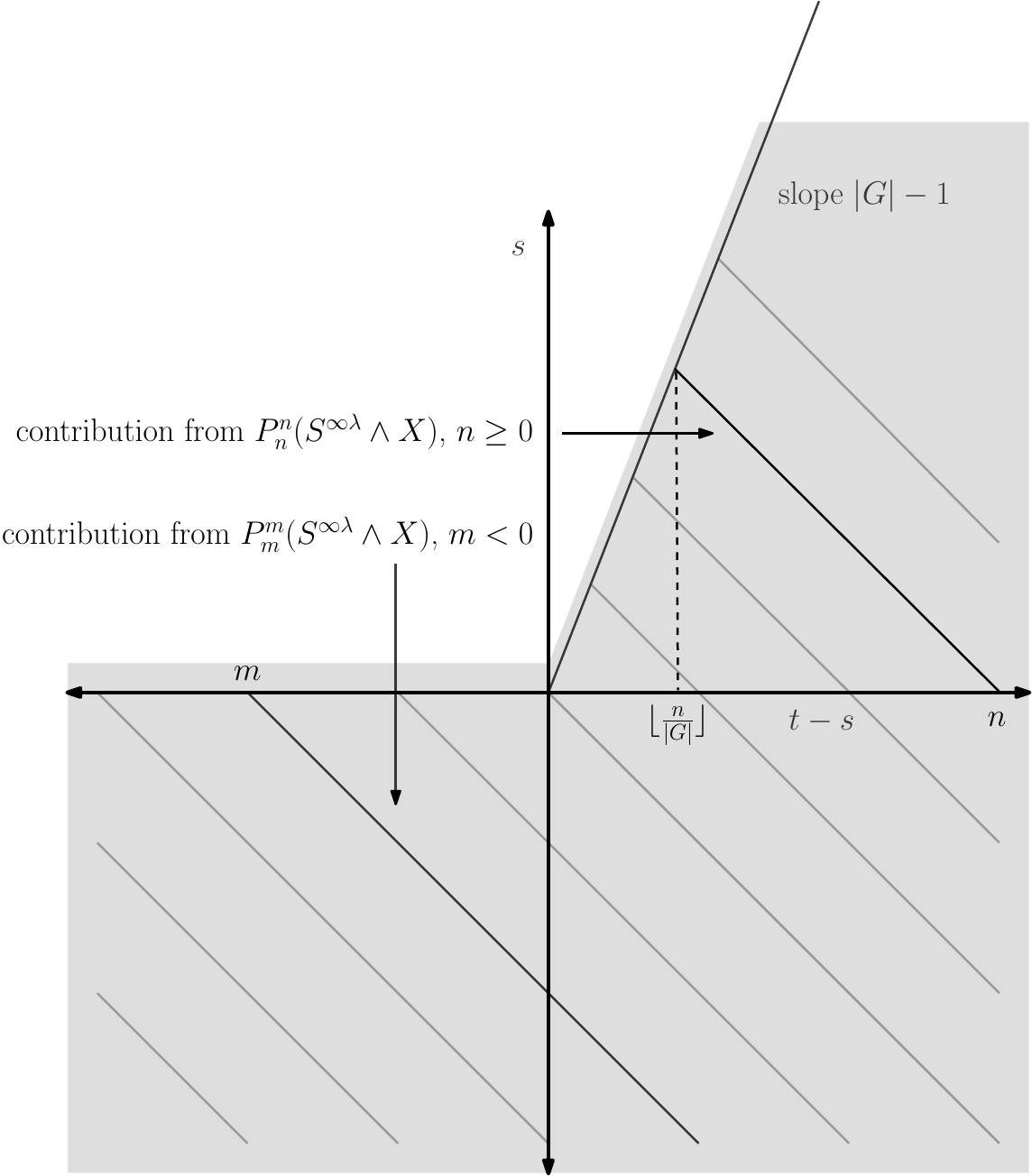}}
\caption{Spectral sequence associated to the tower $\{\EFi \wedge P^\bullet \}$.}
\label{fig:SliceSSVanishing}
\end{center}
\end{figure}

For any $k$, consider the diagram
$$\begin{tikzcd}
\underline{\pi}_k (S^{\infty \lambda} \wedge X) \ar[r] \ar[rd, "\cong", swap]& \underline{\pi}_k\varprojlim (S^{\infty \lambda} \wedge P^\bullet X) \ar[d, "\cong"] \\
& \underline{\pi}_k (S^{\infty \lambda} \wedge P^n X)
\end{tikzcd}$$
We have proven that when $n$ is large enough ($n > k$), the vertical arrow and the diagonal arrow are isomorphisms.  Therefore, the horizontal arrow induces an isomorphism
$$\underline{\pi}_k (S^{\infty \lambda} \wedge X) \cong \underline{\pi}_k \varprojlim (S^{\infty \lambda} \wedge P^\bullet X)$$
for all $k$.  It follows that $S^{\infty \lambda} \wedge X \simeq \varprojlim (S^{\infty \lambda} \wedge P^\bullet X)$, as desired.
\end{proof}

From the discussion in \cite[Section I.4]{Ullman:Thesis} it follows that the localized slice spectral sequences of $\MUG$ and $\BPG$ (and more generally of $G$-ring spectra) are multiplicative spectral sequences.

\subsection{Exotic transfers} \label{sec-transfer}
If the transfer of a given class in the slice spectral sequence is zero, it might still support a non-trivial \emph{exotic transfer} in a higher filtration. Understanding these is both crucial for understanding the Mackey functor structure of the spectral sequence and helpful to deduce differentials and extensions inside the spectral sequence. While the concept of exotic transfers is pretty transparent for permanent cycles, it is slightly more subtle for exotic transfers just happening on finite pages. Following the lead of \cite{BBHS} (in the case of the Picard spectral sequence), we will give a precise definition of this phenomenon and show how it behaves with respect to differentials. It turns out that it is no more difficult to treat a more general setting, which specializes to several different known spectral sequences and allows also for more general operations than just transfers. 
	
In this subsection, we will first state a general definition of exotic $w$-operations and prove some general results.  Then, we will specialize to the case of cyclic 2-groups and prove a variant of \cite[Theorem 4.4]{HHR:C4} that also work for exotic transfers and restrictions on finite pages. 
	
We consider a tower
\[\cdots \to X^{i+1} \to X^i \to X^{i-1} \to \cdots \]
of $G$-spectra. Recall that to this we can associate a spectral sequence as follows: Let $X^m_n = \mathrm{fib}(X^m \to X^{n-1})$. For $V$ a virtual $G$-representation of dimension $t$, we set $E_2^{s,V} = \U{\pi}_{V-s}(X^t_t)$ and more generally 
\[E_r^{s,V} = \im(\U{\pi}_{V-s}X^{t+(r-2)}_t \to \U{\pi}_{V-s} X^t_{t-(r-2)}).\] 
The differentials $d_r\colon E_r^{s,V} \to E_r^{s+r, V+r-1}$ are defined as the restrictions of the boundary maps $\delta\colon \U{\pi}_{V-s} X^t_{t-(r-2)} \to \U{\pi}_{V-s-1} X^{t+r-1}_{t+1}$ (coming from the cofiber sequence $X_{t+1}^{t+(r-1)} \to X_{t-(r-2)}^{t+(r-1)} \to X_{t-(r-2)}^{t}$). See e.g.\ \cite[Section 1.2.2]{HigherAlgebra} for some details in the setting of an ascending filtration. Our setting specializes in particular to the following spectral sequences:
\begin{enumerate}
	\item Given a spectrum $Z$ with a $G$-action, set $X^i = (\tau_{\leq i}Z)^{EG_+}$. We recover the homotopy fixed point spectral sequence. 
	\item Given a spectrum $Z$ with a $G$-action, set $X^i = (\tau_{\leq i}Z \sm \tilde{E}G)^{EG_+}$. We recover the Tate spectral sequence. 
	\item Given a $G$-spectrum $Z$, set $X^i = P^iZ$, the slice tower. We obtain the slice spectral sequence. 
	\item Given a $C_{2^n}$-spectrum $Z$ and $1\leq j \leq n$, set $X^i = \EF[C_{2^j}] \wedge P^iZ$. We obtain the localized slice spectral sequence. This will be the main example of relevance for us. 
\end{enumerate}
We fix an arbitrary map $\Sigma^{\infty} G/K \to \Sigma^{\infty} G/H$ and denote the resulting operation $\pi^H_n \to \pi^K_{n}$ by $w$. The most important case for us will be $H\subset K$ and $w = \Tr_H^K$. But equally well $w$ could be a restriction map, multiplication by a fixed element such as $2$, or any combination of these. 

For notational simplicity, we will restrict for our treatment of exotic $w$-operations to integer degrees. By suspending by a representation sphere, one can easily translate our definitions and results to the $RO(G)$-grading.

\begin{defi}\rm
	Let $x\in E_r^{s, t}(G/H)$, and let $0 \leq p \leq r-2$ and $0\leq q\leq p$. We may lift the corresponding element in $\pi^H_{t-s}X^t_{t-r+2}$ to an element $\widetilde{x}\in \pi^H_{t-s}X^{t+p}_{t+p-q-(r-2)}$ since by definition, we can actually lift it to an element in $\pi^H_{t-s}X^{t+(r-2)}_t$. If $w(\widetilde{x})\in\pi^K_{t-s}X^{t+p}_{t+p-q-(r-2)}$ lies in $E_{r+q}^{s+p, t+p}(G/K)$, we call it a \emph{$w$-operation of $x$ of filtration jump $p$ and page jump $q$}. 
	
	If $p>0$, we speak of an \emph{exotic $w$-operation}, which, depending on $w$, might be an \emph{exotic transfer}, \emph{exotic restriction} etc.\footnote{If $w$ is multiplication by an integer, then the existence of exotic $w$-operations corresponds essentially to hidden extensions. The basic issue of dependence on choices is already present in this more classical case.}  If the page jump is zero, we omit the mention of it.
	

\end{defi}
This definition can be illustrated with the following diagram: 
\[
\xymatrix{
{\widetilde{x}} \ar@{|->}[ddd] \ar@{|->}[rr] && x  \\
&\pi_{t-s}^H(X^{t+p}_{t+p-q-(r-2)}) \ar[r] \ar[d]^w& \pi_{t-s}^H(X^t_{t-(r-2)}) \ar[d]^w
\\
& \pi_{t-s}^K(X^{t+p}_{t+p -q- (r-2)}) \ar[r] \ar[r] &\pi_{t-s}^K(X^t_{t-(r-2)}) \\
 w(\widetilde{x})
}
\]


\begin{rmk}\label{rem:bigremark}\rm 
\begin{enumerate}
\item The $w$-operations of filtration jump $0$ are just the algebraic $w$-operations on the $E_r$-page as inherited from the $E_2$-page. This is why we call the $w$-operations of higher filtration jump \emph{exotic}.
\item The most classical case of exotic $w$-operations is the limiting case when $r= \infty$. If $x\in E_{\infty}^{s,t}(G/H)$ and $X$ denotes $\lim_t X^t$, we can actually lift $x$ to $\widetilde{x} \in \pi_{t-s}^HX$ (which is further than to $\pi_{t-s}^HX^{t+p}$ as required by the previous definition). If $w(\widetilde{x}) \neq 0$, it must be detected in some $E_{\infty}^{s+p, t+p}(G/K)$ and the resulting element is an example of a $w$-operation of filtration jump $p$ on $x$. In the case when $x$ is just on a finite page, we can suitably truncate the original spectral sequence to force $x$ to be a permanent cycle that survives to the $E_\infty$-page. We will do this in the proof of \cref{lem:UniquenessOfExoticW}.
\item \label{item:Non-uniquness} Even in the classical situation of the last item, exotic $w$-operations are in general not unique; in other words, $w(\widetilde{x})$ will depend on the choice of lift $\widetilde{x}$. With notation as in the last bullet point, suppose for example  that there exists $z\in E_{\infty}^{s+i, t+i}(G/H)$ for $0<i<p$ such that $z$ supports a non-exotic $w$-operation. If we lift $z$ to $\widetilde{z}\in \pi_{t-s}^HX$, then $w(\widetilde{x}+\widetilde{z})$ will be detected by $w(z)\in E_{\infty}^{s+i, t+i}(G/K)$, while $\widetilde{x}+\widetilde{z}$ lifts $x$. In the extreme case, $x$ might even be zero. In \cref{lem:UniquenessOfExoticW}, we will prove a criterion that ensures the uniqueness of exotic $w$-operations.  This criterion is often fulfilled in practice.  
    
    \item A $w$-operation $z=w(\widetilde{x})$ of filtration jump $p$ and page jump $0$ defines a $w$-operation of filtration jump $p$ and page jump $q$ if $d_r(z) = \cdots = d_{r+q-1}(z) = 0$ by just mapping $z \in \pi_{t-s}^KX^{t+p}_{t+p-(r-2)}$ down to $\pi_{t-s}^KX^{t+p}_{t+p-q-(r-2)}$. All $w$-operations of page jump $q$ are of this form.
    \item With $x$ and $\widetilde{x}$ fixed, a $w$-operation of filtration jump $p$ can only exist if all $w$-operations of lower filtration jump vanish. Indeed, if the image of $w(\widetilde{x})$ in $\pi^K_{t-s}X^{t+p}_{t+p-q-(r-2)}$ lies in $E_{r+q}^{s+p, t+p}(G/K)$, it is in the image of $\pi^K_{t-s}X^{t+p+q+(r-2)}_{t+p}$. The map from this group to $\pi^K_{t-s}X^{t+p-1}_{t+p-q-(r-2)-1}$ factors through $\pi^K_{t-s}X^{t+p-1}_{t+p} =0$.
\end{enumerate}
\end{rmk}

The following lemma holds by definition. 
\begin{lem}
    Let $x\in E_r^{s,t}(G/H)$ be a $d_r$-cycle and denote by $\overline{x}$ its image in $E_{r+1}^{s,t}(G/H)$. Let $z\in E_{r+q}^{s+p,t+p}$ be a $w$-operation on $x$ of filtration jump $p$ and page jump $q\geq 1$. Then $z$ is a $w$-operation on $\overline{x}$ of filtration jump $p$ and page jump $q-1$. 
\end{lem}

The following is the uniqueness result for exotic $w$-operations that we will use. 

\begin{lem}\label{lem:UniquenessOfExoticW}
    Let $x\in E_r^{s, t}(G/H)$ and $0 < p \leq r-2$.  Suppose every class in $E_2^{s+k,t+k}(G/H)$ for $0<k<p$ is either hit by a differential of length at most $r+k-1$ or supports a differential of length at most $p-k+1$. Denoting by $I$ the image of all $(r+p)$-cycles in $w(E_2^{s+p,t+p}(G/H))$ in $E_{r+p}^{s+p,t+p}(G/K)$, then there is at most one class in $E_{r+p}^{s+p,t+p}(G/K)/I$ that is a $w$-operation of $x$ of filtration jump $p$ and page jump $p$. 
\end{lem}
\begin{proof}
   Consider the towers $\overline{X}^{\bullet}$ and  $\widetilde{X}^{\bullet}$ with 
   \[\overline{X}^i = \begin{cases} X^{t+p} & \text{ if } i\geq t+p \\
   X^i & \text{ else }\end{cases}\]
   and
   \[\widetilde{X}^i =\begin{cases} X_{t-(r-2)}^{t+p} & \text{ if }i\geq t+p \\
   X_{t-(r-2)}^i &\text{ if }t-(r-2)\leq i \leq t+p \\
   0 & \text{ else }\end{cases}. \]
      The maps $X^{\bullet} \to \overline{X}^{\bullet} \leftarrow \widetilde{X}^{\bullet}$ induce maps of spectral sequences $E_{\bullet}^{\bullet, \bullet} \to \overline{E}_{\bullet}^{\bullet, \bullet} \leftarrow \widetilde{E}_{\bullet}^{\bullet,\bullet}$. The first induces isomorphisms $E_2^{s',t'} \to \overline{E}_2^{s',t'}$ for $t' \leq t+p$ and the second isomorphisms  $\widetilde{E}_2^{s',t'} \to \overline{E}_2^{s',t'}$ for $t' \geq t-(r-2)$. Via the maps of spectral sequences, differentials in the original spectral sequence enforce corresponding differentials in the $\widetilde{E}$-spectral sequence in the range $t-(r-2)\leq t' \leq t+p$. In particular, $E_r^{s,t}$ injects into $\widetilde{E}_r^{s,t}$. Note moreover that the $\widetilde{E}$-spectral sequence converges to $\underline{\pi}_*X^{t+p}_{t-(r-2)}$. 
      
      Our assumptions imply that $\widetilde{E}_{\infty}^{s+k, t+k}(G/H) = 0$ for $0<k<p$ and moreover $\widetilde{E}_r^{s,t} = \widetilde{E}_{\infty}^{s,t}$. Thus, we can lift the image of $x$ in $\widetilde{E}_r^{s,t}$ uniquely to $\pi_{t-s}^HX^{t+p}_{t-(r-2)}$ modulo $\widetilde{E}_{\infty}^{s+p,t+p}$. The latter term is a quotient of $\widetilde{E}_2^{s+p,t+p} = E_2^{s+p,t+p}$.  
      
     In summary, we have shown that we can lift $x$ uniquely to $\widetilde{x} \in \pi_{t-s}^HX^{t+p}_{t-(r-2)}$ modulo the image from $\pi_{t-s}^HX^{t+p}_{t+p}$. Thus, $w(\widetilde{x}) \in  \pi_{t-s}^KX^{t+p}_{t-(r-2)}$ is indeed well-defined modulo the image of $w(\pi_{t-s}^HX^{t+p}_{t+p}) = w(E_2^{s+p, t+p})$. 
\end{proof}

\begin{rmk}\rm
One can probably formulate a sharper criterion for the uniqueness of exotic $w$-operations, without requiring that all classes between $x$ and its target vanish. The essential point is to require that there are no interleaving $w$-operations such as classes in  $E_r^{s+k, t+k}$ with $0<k<p$ that admit nonzero $w$-operations of filtration jump smaller than $p-k$. Moreover, one would have to enlarge $I$ to include exotic $w$-operations as well. We refrain from making this precise.
\end{rmk}

\begin{prop}\label{prop:w-op}
	Let $x\in E_r^{s, t}(G/H)$ and $z$ a class with $d_r(z) =x$. Suppose $d_{r+q}(w(z))$ is zero for $q<p$. Then $d_{r+p}(w(z))$ is a $w$-operation of $x$ of filtration jump $p$ and page jump $p$. 
	
\end{prop}
\begin{proof}
	We choose a lift of $z \in \pi^H_{t-s+1}X^{t-r+1}_{t-2r+3}$ to $\widetilde{z} \in \pi^H_{t-s+1}X^{t-1}_{t-r+1}$. As $\delta(\widetilde{z})$ in the diagram below is a lift of $x$, contemplating the fate of $w(\widetilde{z})$ passing along the two different travel paths from the upper left corner to the lower right corner proves the proposition. 
	\[
	\qedxymatrix{
\pi^K_{t-s+1}X^{t-1}_{t-r+1} \ar[r]^{\delta} \ar[d] & \pi^K_{t-s} X^{t+r-2}_{t} \ar[d] \\
E^{s-r, t-r+1}_r\subseteq \pi^K_{t-s+1}X^{t-r+1}_{t-2r+3}\ar[d] & \pi^K_{t-s}X^{t+p}_{t+p-r+2} \supseteq E^{s+p, t+p}_{r} \ar[d] \\
E^{s-r, t-r+1}_{r+p}\subseteq\pi^K_{t-s+1}X^{t-r+1}_{t-2r+3-p} \ar[r]^-{\delta} & \pi^K_{t-s}X^{t+p}_{t-r+2} \supseteq E^{s+p, t+p}_{r+p}
}\]
\end{proof}

While our definition and results so far are very general (and our proofs would also apply to other settings than equivariant homotopy theory), we will now formulate a result that is specific to cyclic $2$-groups. For the following proposition, both the statement and the proof are variants of \cite[Theorem 4.4]{HHR:C4}, but also work for exotic transfers and restrictions on finite pages and circumvent a mistake in \cite[Lemma 4.5]{HHR:C4}.\footnote{With notation as in the cited lemma, a counterexample is the following: Fix an object $A$. Take $A_{i,j}$ to be  $\Sigma^{-1}A$, zero or $A$, depending on whether $i+j$ is smaller, equal or larger than $2$. The $a_{i,j}$ and $b_{i,j}$ are $\id$ if possible, with the exception of $a_{2,1}$ being an arbitrary self-equivalence of $A$, which is not equivalent to $\pm\id$. Take further $W = A$ and $f_3 = \id$. Then $f_1$ exists (and can be taken to be $\id$), but $f_1$ and $f_2$ cannot simultaneously exist. Strictly speaking, the cited lemma is ambiguous on whether it claims that $f_1$ and $f_2$ exist \emph{simultaneously} if $f_1$ exists, but this seems to be the way that it is later used in \cite[Theorem 4.4]{HHR:C4}.}

\begin{prop}\label{prop-transfer}
Let $G$ be a cyclic 2-group, $H\subset G$ an index $2$ subgroup, and $V \in RO(G)$. 
    \begin{enumerate}[(i)]
    \item Let $y\in E_{r+1}^{s,V}(G/G)$ with $a_{\sigma}y = 0 \in E_{r+1}^{s+1, V+1-{\sigma}}(G/G)$. Then $y$ is an (exotic) transfer of filtration jump (at most) $r-1$.\footnote{{The ``at most'' is actually unnecessary here, as the proof shows that $y$ is an exotic transfer of filtration jump $r-1$. We write it for emphasis though since $y$ might be very well also an exotic transfer of smaller filtration jump. This is related to the non-uniqueness described in \cref{item:Non-uniquness} of \cref{rem:bigremark}. Thus the statement is best used in conjunction with a uniqueness result like \cref{lem:UniquenessOfExoticW}.}}
    \item Let $z \in E_{r+1}^{s,V}(G/H)$ with $\Tr(z) = 0 \in E_{r+1}^{s,V}(G/G)$. Then $z$ is an (exotic restriction) from $E_{r+1}^{s-(r-1), V-(r-1) +(1-\sigma)}$ of filtration jump (at most) $r-1$.
    \end{enumerate}
\end{prop}

\begin{proof}
    For the first part, by shifting the tower and applying suspension if necessary, we can fix the bidegree of $y$ to be $(r-1,r-1)$. The term $E_{r+1}^{r-1,r-1}(G/G)$ injects into $\pi_0^GX_0^{r-1}$. Smashing the long exact sequence associated with the cofiber sequence $G/H_+ \to S^0 \xrightarrow{a_{\sigma}} S^{\sigma}$ with $X_0^{r-1}$ and taking homotopy groups, we get the long exact sequence 
    \[\pi_1^G X_0^{r-1} \xrightarrow{a_\sigma} \pi_{1-\sigma}^G X_0^{r-1} \xrightarrow{\Res} \pi_0^H X_0^{r-1} \xrightarrow{\Tr} \pi_0^G X_0^{r-1} \xrightarrow{a_\sigma} \pi_{-\sigma}^G X_0^{r-1}.\]
    From this long exact sequence, we see that $a_{\sigma}y = 0$ implies $y = \Tr(\widetilde{w})$ with $\widetilde{w} \in \pi_0^HX_0^{r-1}$. By definition, this defines an element $w\in E_{r+1}^{0,0}(G/H)$ such that $y$ is an exotic transfer of $w$ of filtration jump $r-1$.
    
    For the second part, we can fix the bidegree of $z$ to be $(r-1,r-1)$ by shifting the tower and applying suspension if necessary to view $z$ as an element in $\pi_0^HX_0^{r-1}$. Using the long exact sequence induced by $G/H_+ \to S^0 \xrightarrow{a_{\sigma}} S^{\sigma}$ again, we see that $z$ is the restriction of some $\widetilde{v}\in \pi_{1-\sigma}^GX^{r-1}_0$. By definition, this defines an element $v\in E_{r+1}^{0,1-{\sigma}}(G/G)$ such that $z$ is an exotic restriction of $v$ of filtration jump $r-1$. 
\end{proof}

Let us give an example of a possible workflow working with exotic transfers, which we will apply in \cref{prop-d13}.
\begin{workflow}\label{workflow:ExoticTransfers}\rm
Let $G$ be a cyclic $2$-group and $H\subset G$ of index $2$. Let $y \in E_r^{s+r-1,t+r-1}(G/G)$ and $r'>r$. We assume the following:
\begin{enumerate}
    \item \label{w1} $a_{\sigma}y$ is nonzero and is hit by a $d_r$-differential;
    \item \label{w2} $y$ persists to a nonzero class in the $E_{r'+r-1}$-page, which we denote by the same name; 
    \item \label{w3} every class in $E_2^{s+k, t+k}(G/H)$ for $0<k<r-1$ is either hit by a differential of length at most $r+k-1$ or supports a differential of length at most $r-k$;
    \item \label{w4} $y\in E_{2r-1}^{s+r-1,t+r-1}$ is not the image of a $(2r-1)$-cycle in $E_2$ which is the transfer of a class in $E_2^{s+r-1, t+r-1}$. \\
    
\end{enumerate} 
By (1), $a_{\sigma}y$ vanishes on $E_{r+1}$. Thus, by \cref{prop-transfer}, there exists $x\in E_{r+1}^{s, t}(G/H)$ such that $y \in E_{r+1}^{s+r-1,t+r-1}(G/G)$ is an exotic transfer of $x$ of filtration jump $r-1$. Applying \cref{lem:UniquenessOfExoticW} in conjunction with (3) and (4), we see that $x$ cannot be zero (as zero is the unique exotic transfer of zero under our assumptions); in case that there is only one non-zero element in the relevant bidegree, this already uniquely determines $x$. Suppose now further that:
\begin{enumerate}\setcounter{enumi}{4}
    \item $x = d_{r'}(a)$;
    \item $d_{r'+q}(\Tr_H^Ga) = 0$ for $0\leq q < r-1$. 
\end{enumerate}Then \cref{prop:w-op} implies that $d_{r'+r-1}(\Tr_H^G(a))$ is an exotic transfer of $x$ in the same degree as $y\in E_{r'+r-1}^{s,t}$ and thus must be $y$ by \cref{lem:UniquenessOfExoticW} again.
\end{workflow}

\subsection{The behaviour of norms} \label{sec-norm}

This section is about the behaviour of norms in the (regular) slice spectral sequence and its localized variant. We will formulate a generalization of \cite[Chapter~I.5]{Ullman:Thesis} and then discuss how it applies to Ullman's original setting (the regular slice spectral sequence), to the localized slice spectral sequence and the homotopy fixed point spectral sequence. 

We will first work in an abstract setting: 
 Let $(X^i)$ be a tower of $G$-spectra and $E_*^{*,*}$ be the associated spectral sequence as in the preceding subsection. Set $X^{\infty} = \lim_i X^i$ and $X_n = X^{\infty}_n$. 
 
 Let $H\subset G$ be a subgroup of index $h$. We assume that we have maps $N_H^GX_n\to X_{hn}$ and $N_H^GX_n^n \to X_{hn}^{hn}$ that are (up to homotopy) compatible with the maps $X_n \to X_{n-1}$ and $X_n \to X^n_n$. (Here we leave the restriction maps implicit.) We call this a \emph{norm structure}. It induces norm maps $N_H^G\colon E_2^{s, V+s} \to E_2^{hs, \ind_H^GV +hs}$. 

\begin{prop}\label{prop-norm}
	Let $x \in E_2(G/H)$ be an element representing zero in $E_{r+2}(G/H)$. Then $N_H^G(x)$ represents zero in $E_{rh+2}(G/G)$. 
\end{prop}
\begin{proof}
	The proof is the same as that of \cite[Proposition I.5.17]{Ullman:Thesis}.
\end{proof}

\begin{exam}\label{exam:regslice}\rm
Our first example of this setting is the regular slice tower of \cite{Ullman:Thesis}, which coincides with the slice tower of \cite{HHR} for norms of $\MUR$ and $\BPR$ -- thus there should be no danger of confusion if we use the same notation $P^iX$ for the regular slice tower. 

Ullman constructs in \cite[Corollaries I.5.10 and I.5.11]{Ullman:Thesis} for every $H$-spectrum $X$ natural compatible maps $N_H^GP_n X \to P_{nh}N_H^GX$ and $N_H^GP_n^nX \to P_{nh}^{nh}N_H^GX$. Moreover the square
\[
\xymatrix{N_H^GP_nX \ar[r]\ar[d] & P_{hn}N_H^GX \ar[d] \\
N_H^GP_{n-1}X \ar[r] & P_{hn-h} N_H^GX
}
\]
commutes, as $N_H^GP_nX$ is $\geq hn$ by \cite[Corollary I.5.8]{Ullman:Thesis} and both maps into $N_H^GP_nX \to P_{hn-h}N_H^GX$ are compatible with the respective maps to $N_H^GX$.

Let $R$ be a $G$-spectrum with a map $N_H^G\res_H^G R \rightarrow R$. The composite $N_H^G P_n \res_H^G R \to P_{nh}N_H^G\res_H^GR \to P_{nh}R$ and its analogue for $P_n^n$ define a norm structure on the regular slice tower of $R$. This applies in particular if $R$ is a $G$-commutative ring spectrum. 
\end{exam}

\begin{exam}\label{exam:locslice}\rm
Let $R$ be a $G$-commutative ring spectrum with $G = C_{2^n}$. We will define a norm structure on the tower $X^i = a_{\lambda}^{-1}P^iX$ defining the localized regular slice spectral sequence. Using the observations above for the regular slice spectral sequence, it suffices to produce natural maps $N_H^G\res_H^G a_{\lambda}^{-1} P_nR \to a_{\lambda}^{-1}N_H^G\res_H^G P_{hn}R$ and similarly for $P_n^n$. As $N_H^G$ and $\res_H^G$ are monoidal, by \cref{lem:NormResSphere} it thus suffices to provide a natural map 
 \[a_{\lambda}^{-1}\mathbb{S}_G \simeq a_{\lambda}^{-1}N_H^G \mathbb{S}_H \to N_H^G \res_H a_{\lambda^{-1}}\mathbb{S}_G \simeq a_{\ind_H^G \res_H^G \lambda}^{-1} \mathbb{S}_G\]
As observed before, $\ind_H^G \res_H^G \lambda$ is a multiple of $\lambda$ if $H\neq e$ and contains a trivial summand if $H = e$. This produces the norm structure if $H\neq e$. 
In contrast for $H =e$, all norms would have to be zero.

We remark that we have not used the full strength of our considerations in \cref{sec:MultLoc} here, but we expect that these will be necessary for deeper considerations about norms.
\end{exam}

\begin{exam}\label{exam:hfpss}\rm
Lastly we define a norm structure on the homotopy fixed point spectral sequence. Observe first that there is for $H$-spectra $X$ a natural map 
\[N_H^GX^{EH_+} \to (N_H^GX^{EH_+})^{EG_+} \xleftarrow{\simeq} (N_H^G X)^{EG_+},\] 
where the latter map is an equivalence as $\res_e^G N_H^G X \to \res_e^G N_H^G X^{EH_+}$ is an equivalence. 

Recall that the tower defining the homotopy fixed point spectral sequence for a spectrum $R$ is defined by $X^n = (\tau_{\leq n}R)^{EG_+}$. We observe that we have natural equivalences $X_n \simeq (P_nR)^{EG_+}$ and $X_n^n \simeq  (P^n_nR)^{EG_+}$ for $(P^nR)_n$ the regular slice tower. Combining these equivalences with the natural map from the last paragraph, the norm structure from \cref{exam:regslice} induces a norm structure on the homotopy fixed point spectral sequence. 
\end{exam}

We will use the following proposition without further comment.
\begin{prop}
	Both in the regular slice spectral sequence and in the localized regular slice spectral sequence of a $G$-commutative ring spectrum, the norms are multiplicative: $N_H^G(xy) = N_H^G(x) N_H^G(y)$. 
\end{prop}
\begin{proof}This follows from the commutativity of 
\[	\xymatrix{N_H^G(P_mX\wedge P_nY) \ar[r]\ar[d] & N_H^G(P_{m+n}X\sm Y) \ar[d] \\
P_{hm}N_H^GX \sm P_{hn}N_H^GY \ar[r] & P_{hm+hn} N_H^G(X\sm Y)
}\]
for $G$-spectra $X$ and $Y$. This in turn follows as there is up to homotopy just one map \[N_H^G(P_mX\wedge P_nY) \to P_{hm+hn} N_H^G(X\sm Y)\]
compatible with the maps to $N_H^G(X\sm Y)$ as $N_H^G(P_mX \sm P_nY) \geq h(m+n)$ by \cite[Corollaries I.4.2 and I.5.8]{Ullman:Thesis}. 
\end{proof}

Given two towers $(X^n)$ and $(Y^n)$ with norm structures, a morphism of towers $(X^n) \to (Y^n)$ is \emph{compatible with the norm structures} if the diagrams
\[
\xymatrix{N_H^G X_n \ar[r]\ar[d] & X_{hn} \ar[d] \\
N_H^G Y_n \ar[r] & Y_{hn}
}
\]

commute for all $n$ and similarly for $X_n^n$ and $Y^n_n$. Such a morphism induces in particular a morphism of spectral sequence that is compatible with the norms on the $E_2$-terms. 

\begin{exam}\label{exam:slicehfpss}\rm
Given any spectrum $X$, there is a natural map from the regular slice tower to the tower defining the homotopy fixed point spectral sequence, namely $P^nX \to (P^nX)^{EG_+}$. In case that $X$ is a $G$-commutative ring spectrum (or more generally a spectrum admitting a map $N_H^G\res_H^GX \to X)$, this map of towers is (essentially by construction) compatible with the norm structures introduced in \cref{exam:regslice} and \cref{exam:hfpss}.
\end{exam}

\subsection{Comparison of spectral sequences}\label{sec:compspec}
When computing localizations of a norm, we can apply different spectral sequences. For instance, in the isomorphism
\[
\EF[H] \wedge N_H^G X \simeq P^*_{G/H}(N_e^{G/H} \Phi^H(X))
\]
of \cref{prop-main}, the left hand side $\EF[H]\wedge N_H^G X$ can be computed by the localized slice spectral sequence we just built, while the right hand side can be computed by the pullback of the $(G/H)$-equivariant slice spectral sequence of $N_e^{G/H}\Phi^{H}X$. In this section, we give a comparison map between these spectral sequences, which we will use in understanding the homotopy fixed points and the Tate spectral sequence of $N_e^{G/H}\Phi^{H}X$.

Such comparison can only be made by regrading the slice tower. In the cases of relevance for us this takes the shape of the following doubling process: Let $P^{\bullet}$ be a tower, we define $\DP^{\bullet}$, the doubled tower of $P^{\bullet}$, as
\[
\DP^{2n+\epsilon} := P^{n}
\]
for $\epsilon = 0, 1$. We also use $\D$ as a prefix of a spectral sequence obtained from a tower as the spectral sequence of the doubled tower.

In the following theorem we will use both the slice tower $P^{\bullet}$ and the pullback functor $P^*_{G/C_2}$ from \cref{sec:GeometricFixedPoints}; the double usage of $P$ will hopefully not cause any confusion to the reader.

\begin{thm}\label{thm:comparison}
    Let $G = C_{2^n}$, $X\in Sp^G$ and $Y = \Phi^{C_2}X \in Sp^{G/C_2}$. Let $P^{\bullet}X$ and $P^{\bullet} Y$ be their slice towers in the corresponding categories. Then there is a commutative diagram of towers
\begin{center}
\begin{tikzcd}
P^{\bullet}X \arrow[d] \arrow[r]          & P^*_{G/C_2}\DP^{\bullet}Y \\
\tilde{E}G \wedge P^{\bullet}X \arrow[ru] &                         
\end{tikzcd}
\end{center}
such that the map $\tilde{E}G \wedge P^{\bullet}X \rightarrow P^*_{G/C_2}\DP^{\bullet}Y$ converges to the $G$-equivalence $\tilde{E}G \wedge X \rightarrow P^*_{G/C_2}Y$.

In particular, the induced map on the $C_2$-level spectral sequences of $P^{\bullet}X \rightarrow P^*_{G/C_2}\DP^{\bullet}Y$ converges to the geometric fixed points map
\[
    \Phi^{C_2}: \pi_{\star}^{C_2} X \rightarrow \pi_* Y.
\]
\end{thm}

\begin{proof}
    To construct the map $P^{\bullet} X \rightarrow P^{*}_{G/C_2}\DP^{\bullet}Y$, consider the composition 
    \[
    X \rightarrow a_{\lambda}^{-1}X \simeq P^{*}_{G/C_2}Y \rightarrow P^{*}_{G/C_2}\DP^{2i}Y.
    \]
    We only need to show $P^{*}_{G/C_2}\DP^{2i}Y \leq_{Slice} 2i$ for each $i$ (the analogous statement for $\DP^{2i+1}$ follows from this). This can be checked by testing against slice cells of dimension more than $2i$. By induction we can assume that our claim is true after restriction to any proper subgroup of $G$, so we can ignore induced slice cells. Thus it suffices to check that $[S^{k\rho_G}, P^{*}_{G/C_2}\DP^{2i}Y]_G = 0$ for $k|G| > 2i$. 
    
    The following equivalence of $G$-spectra is essential to our proof:
    \[\tilde{E}G \wedge S^{k\rho_G} \simeq P^*_{G/C_2}S^{k\rho_{G/C_2}}.\]
    It comes from the fact that both sides are equivalent to the representation sphere $S^{\infty \lambda + \rho_{G/C_2}}$. The left hand side of the equivalence is a localization of a slice cell of dimenson $k|G|$ while the right hand side is a pullback of a slice cell of dimension $\frac{k|G|}{2}$. This difference is the reason of doubling the tower of $Y$.

    Using this equivalence, we have a series of equivalences of mapping sets:
    \begin{align*}
        [S^{k\rho_G}, P^{*}_{G/C_2} \DP^{2i} Y]_G & \cong [\tilde{E}G \wedge S^{k\rho_G}, P^{*}_{G/C_2} \DP^{2i} Y]_G\\
        & \cong [P^{*}_{G/C_2}S^{k\rho_{G/C_2}}, P^{*}_{G/C_2} \DP^{2i} Y]_G\\
        & \cong [S^{k\rho_{G/{C_2}}}, \DP^{2i} Y]_H\\
        & \cong [S^{k \rho_{G/C_2}},P^i Y]_H\\
        & \cong 0.
    \end{align*}
    The change-of-group isomorphism comes from the fact that $P^*_{G/C_2}$ is fully faithful on homotopy categories, and the last isomorphism is because $S^{k\rho_{G/C_2}}$ is a slice cell of dimension $>i$ in $G/C_2$-spectra.
    
    By construction, the map $P^{\bullet}X \rightarrow  P^{*}_{G/C_2}\DP^{\bullet}Y$ converges to the map $X \rightarrow a_{\lambda}^{-1}X \simeq P^*_{G/C_2}Y$. Since everything in the tower $ P^{*}_{G/C_2}\DP^{\bullet}Y$ is already $a_{\lambda}$-local, the tower map factors through the $a_{\lambda}$-localization $\tilde{E}G \wedge P^{\bullet}X$.
\end{proof}

\begin{prop}\label{prop:compare_norm}
    Let $G = C_{2^n}$ and $X\in Sp^G$ a $G$-commutative ring spectrum. Then the tower $P^*_{G/C_2}\DP^{\bullet}\Phi^{C_2}X$ has a norm structure in the sense of \cref{sec-norm} and the maps 
    \[P^{\bullet}X \to P^*_{G/C_2}\DP^{\bullet}\Phi^{C_2}X \quad\text{and}\quad a_{\lambda}^{-1}P^{\bullet}X \to P^*_{G/C_2}\DP^{\bullet}\Phi^{C_2}X\]
    from \cref{thm:comparison} are compatible with norms from subgroups containing $C_2$.
\end{prop}
\begin{proof}
    Let $H \subset G$ be a subgroup of index $h$ such that $C_2\subset H$. Then we obtain maps 
    \[N_H^GP^*_{G/C_2} P_{2n} \Phi^{C_2}X \simeq P^*_{G/C_2} N_{H/C_2}^{G/C_2} P_{2n} \Phi^{C_2}X \to P^*_{G/C_2}P_{2hn} \Phi^{C_2}X\]
    and 
    \[N_H^GP^*_{G/C_2} P_{2n}^{2n} \Phi^{C_2}X \simeq P^*_{G/C_2} N_{H/C_2}^{G/C_2} P_{2n}^{2n} \Phi^{C_2}X \to P^*_{G/C_2}P_{2hn}^{2hn} \Phi^{C_2}X,\]
    which are compatible in the necessary sense. Here we use the norm structure on the regular slice tower from \cref{exam:regslice}, the $G/C_2$-commutative ring structure on $\Phi^{C_2}X$ from \cref{exam-MU} and the commutation of norms and pullbacks from \cref{prop:normpullback}. 
    
    To show that $P^{\bullet}X \to P^*_{G/C_2}\DP^{\bullet}\Phi^{C_2}X$ is compatible with norm structures, note first that the diagram
    \[
    \xymatrix{N_H^GX \ar[r]\ar[d]& N_{H}^{G}P^*_{G/C_2}\Phi^{C_2}X \simeq P^*_{G/C_2}N_{H/C_2}^{G/C_2}\Phi^{C_2}X  \ar[d]\\
    X \ar[r] & P^*_{G/C_2} \Phi^{C_2}X
    }
    \]
    commutes since $X \to P^*_{G/C_2}\Phi^{C_2}X \simeq a_{\lambda}^{-1}X$ is a morphism of $\cO$-algebras by \cref{def:localizationO} and \cref{exam-MU}, where  $\cO$ is an $N_{\infty}$-operad arising as the pullback of a $G/C_2$-$E_{\infty}$-operad. Next consider the diagram
    \[
    \xymatrix{\Phi^{C_2}N_H^GP_nX \ar[r] \ar[d] & N_{H/C_2}^{G/C_2}P_{2n}\Phi^{C_2}X \ar[r]\ar[d]& N_{H/C_2}^{G/C_2} \Phi^{C_2}X \ar[d]\\
    \Phi^{C_2}P_{hn}X \ar[r] & P_{2hn} \Phi^{C_2}X \ar[r] & \Phi^{C_2}X
    }
    \]
    The outer rectangle is obtained from the previous diagram by applying $\Phi^{C_2}$ (and using the maps $P_nX \to X$ and $P_{hn}X \to X$) and thus commutes. Given the connectivity estimate \cite[Corollary I.5.8]{Ullman:Thesis} and the universal property of $P_{2hn}$, we see that $\Phi^{C_2}N_H^GP_nX \to \Phi^{C_2}X$ factors through $P_{2hn}\Phi^{C_2}$ in an essentially \emph{unique} way, so the left square also has to commute. By the adjointness of $\Phi^{C_2}$ and $P^*_{G/C_2}$ this implies the commutativity of 
    \[\xymatrix{
    N_H^GP_nX \ar[r] \ar[d] & N_H^GP^*_{G/C_2}P_{2n}\Phi^{C_2}X \ar[d] \\
    P_{hn} X \ar[r]& P^*_{G/C_2}P_{2hn} \Phi^{C_2}X
    }\]
    The proof of the commutativity for the corresponding square for $P^n_n$ is completely analogous. The $a_{\lambda}$-inverted case follows again because the target is $a_{\lambda}$-local.
\end{proof}
	

\section{The localized slice spectral sequences of $BP^{(\!(G)\!)}$: summary of results} \label{sec:4}

We now turn to analyze the localized slice spectral sequence of $BP^{(\!(G)\!)}$ for $G = C_{2^n}$. From now on, everything will be implicitly 2-localized.  In this section, we list our main results and give an outline of the computation.  Detailed computations of the results stated in this section are in Section~\ref{section:sliceSScomputation}.

As we discussed in Section~\ref{sec:sliceSSBackground}, the Slice Theorem \cite[Theorem~6.1]{HHR} implies that the slice associated graded of $BP^{(\!(C_{2^n})\!)}$ is
\[
    H\UZ[G\cdot \ot_1, G\cdot \ot_2,\ldots],
\]
    where $\ot_i \in \pi^{C_2}_{(2^{i}-1)\rho_2} BP^{(\!(C_{2^n})\!)}$ (see also \cite[Section~2.4]{HHR} for details).

For the rest of the paper, we use $\lambda$ for the 2-dimensional real representation of $C_{2^n}$ which is rotation by $\left(\frac{\pi}{2^{n-1}}\right)$, and $\sigma$ for the $1$-dimensional sign representation of $G$. We use $\sigma_2$ for the sign representation of the unique \emph{subgroup} $C_2$ in $G$. Let $i < j \leq n$, we will use $\Res_{2^i}^{2^j}$, $\Tr_{2^i}^{2^j}$ and $N_{2^i}^{2^j}$ for restrictions, transfers and norms between $C_{2^i}$ and $C_{2^j}$ as subgroups of $G$. If their subscript and superscript are omitted, they mean the restriction, transfer and norm between $C_2$ and $C_4$.

\begin{thm}\label{thm-main} \hfill

		\begin{enumerate}
			\item Let $G = C_{2^n}$ and $H = C_{2}$ be the subgroup of order $2$ inside $G$. There is a $RO(G/H)$-graded spectral sequence of  Mackey functors $a_{\lambda}^{-1}\SliceSS(BP^{(\!(G)\!)})$ that converges to the $RO(G/H)$-graded homotopy Mackey functor of $N_e^{G/H} H\FF_2$. The $E_2$-page of this spectral sequence is
			\[
			a_{\lambda}^{-1}H\UZ_{\star}[G\cdot \ot_1, G\cdot \ot_2,\cdots].
			\]
			\item The integral $E_2$-page of $a_\lambda^{-1}\SliceSS(\BPG)$ is bounded by the vanishing lines $s = (2^n - 1)(t-s)$ and $s = -(t-s)$ in Adams grading. In other words, at stem $t-s$, the classes with filtrations greater than $(2^{n}-1)(t-s)$ or less than $-(t-s)$ are all zero.
			\item On the integral $E_2$-page, the $a_{\lambda}$-localizing map
			\[
			\SliceSS(BP^{(\!(G)\!)}) \to a_{\lambda}^{-1}\SliceSS(BP^{(\!(G)\!)})
			\]
			induces an isomorphism of classes in positive filtrations.  The kernel of this map consists of transfer classes in $\SliceSS(\BPG)$ from the trivial subgroup in filtration $0$.  These classes are all permanent cycles.
		\end{enumerate}
	\end{thm}
	
	\begin{proof}
		By \cref{thm:invertingAclassE2page}, $a_{\lambda}^{-1}\SliceSS(BP^{(\!(G)\!)})$ computes the homotopy of $\tilde{E}G \wedge BP^{(\!(G)\!)}$. By \cref{prop-main} and the fact that $\Phi^{C_2}(BP_{\RR})\simeq H\FF_2$,
		\[
		\tilde{E}G \wedge BP^{(\!(G)\!)} \simeq P^*_{G/C_2}(N_1^{G/C_2}H\FF_2).
		\]
        Since the $E_2$-page of the slice spectral sequence of $\BPG$ has the form
        \[
            H\UZ_{\star}[G\cdot \ot_1, G\cdot \ot_2,\ldots],
        \]
        the $E_2$-page of $a_\lambda^{-1}\SliceSS(\BPG)$ is
        \[
        a_{\lambda}^{-1}\HZ_\star[G \cdot \ot_1, G \cdot \ot_2, \ldots]
        \]
		Together with \cref{prop-main} and \cref{thm:invertingAclassE2page} this proves (1).
		
		The top vanishing line $s = (2^n - 1)(t-s)$ follows from the fact that $\U{\pi}_i(S^{k\rho_G + l\lambda} \wedge H\UZ) = 0$ for $k,l \geq 0$ and $i < k$ (See \cite[Theorem~4.42]{HHR}). For the second vanishing line $y = -x$, note that in stem $t-s$, classes in filtration less than $-(t-s)$ are contributed by slices of negative dimension, but $BP^{(\!(G)\!)}$ has no negative slices.  This proves (2).
		
		To prove (3), by unpacking the description of the $E_2$-page, we need to show that for $k,l \geq 0$, the $a_\lambda$-multiplication map
		\[
		a_{\lambda}: \pi^{G}_i(S^{k\rho_G + l \lambda} \wedge H\UZ) \longmapsto \pi^{G}_i(S^{k\rho_G + (l+1)\lambda} \wedge H\UZ)
		\]
		is an isomorphism for $k \leq i < k|G| + 2l$ and is surjective with kernel consisting of transfer classes from trivial subgroup for $i = k|G| + 2l$. Using the cellular structures and their corresponding chain complexes described in \cite[Section~3]{HHR:C4}, we see that when $k \leq i \leq k|G| + 2l$, $a_{\lambda}$ induces isomorphism on the cellular chain complexes, therefore it induces isomorphism on homology for $k \leq i < k|G| + 2l$ and surjection on homology for $i = k|G| + 2l$ with the kernel exactly the image of $Tr_1^{2^n}$. Since the underlying tower of the slice tower is the Postnikov tower, all the class in the trivial subgroup and their transfers are permanent cycles.
	\end{proof}
	
	\begin{rmk}\rm
	    In fact, $(2)$ and $(3)$ of \cref{thm-main} hold in a greater generality. For instance, they are true for any $(-1)$-connected $G$-spectrum. We will investigate properties of the localized slice spectral sequences in a future paper.
	\end{rmk}
	
	By \cite{LNR} and \cite{BBLNR}, all $C_{2^n}$ norms of $H\FF_2$ are cofree, therefore we will not distinguish between their fixed points and homotopy fixed points.

	\begin{cor}
		The $0$-th homotopy group of $(N_1^{2^{n-1}}H\FF_2)^{hC_{2^{n-1}}}$ is isomorphic to $\ZZ/2^n$.
	\end{cor}
	
	\begin{proof}
		In $a_{\lambda}^{-1}\SliceSS(BP^{(\!(G)\!)})$, the only Mackey functor contributing to the $0$-stem is $\U{\pi}_0(a_{\lambda}^{-1}H\UZ)$, and we claim that
		\[
		\pi_0^G(a_{\lambda}^{-1}H\UZ)(G/G) \cong \ZZ/2^{n}.
		\]
		Indeed, the maps $\pi_0^G(S^{n\lambda} \sm H\UZ) \to \pi_0^G(S^{(n+1)\lambda} \sm H\UZ)$ are isomorphisms for $n\geq 1$ and $\pi_0^G(S^{\lambda} \sm H\UZ)$ is the cokernel of the transfer $\Tr_1^{2^n}\colon \pi_0^eH\UZ \to \pi_0^{C_{2^n}}H\UZ$, i.e.\ of multiplication by $2^n$ on $\ZZ$. 
	\end{proof}

	For the rest of the paper, we focus on the case $G = C_4$.
	
	\begin{thm}\label{thm:htpy}
		The first $8$ stems of $\pi_{*}^{C_4}(a_{\lambda}^{-1}\BPfour)\cong \pi_*^{C_2}N_1^2H\FF_2$ are shown in the following chart:
		\smallskip
		\begin{center}
			\scalebox{0.9}{
				\begin{tabular}{|c|c|c|c|c|c|c|c|c|c|}\hline
					i & $0$ & $1$ & $2$ & $3$ & $4$ & $5$ & $6$ & $7$ & $8$\\ \hline
					$\pi_i$ & \color{red}$\ZZ/4$ & $\color{red}\ZZ/2$ & $\color{red}\ZZ/4$ & $\color{red}\ZZ/2 \color{black}\oplus \ZZ/2$ & $\ZZ/2$ & $\ZZ/2$ & $\color{red}\ZZ/4 \color{black}\oplus \ZZ/2$ & $\color{red}\ZZ/2 \oplus \ZZ/2 \color{black} \oplus \ZZ/2 \oplus \ZZ/2$ & ${\color{red}\ZZ/2} \oplus \ZZ/2 \oplus \ZZ/2$\\\hline
			\end{tabular}}
		\end{center}
	\medskip
	    On the $E_\infty$-page of the localized spectral sequence, the black subgroups are those generated by non-exotic transfers from $\mathcal{A}_* = \pi_*(H\mathbb{F}_2 \wedge H\mathbb{F}_2)$, and the red subgroups consist of everything else. For the Mackey functor structure, see \cref{fig-E11E16}.
		
Modulo transfers from $\mathcal{A}_*$, the homotopy groups have the following generators:
 \begin{enumerate}
\item $\pi_1$ is generated by $\eta = N(\ot_1)a_{\lambda}a_{\sigma}$, the image of the first Hopf invariant one element under the composition $\mathbb{S} \rightarrow (\BPfour)^{C_4} \rightarrow (a_{\lambda}^{-1}\BPfour)^{C_4}$; 
\item $\pi_2$ is generated by $\frac{\eta^2}{2} = 2u_{\lambda}{a_{\lambda}^{-1}}$;
\item $\pi_3$ is generated by $\nu = N(\ot_2)a_{\lambda}^3a_{\sigma}^3$, the image of the second Hopf invariant one element;
\item $\pi_6$ is generated by $\frac{\nu^2}{2} = 2u_{\lambda}^3{a_{\lambda}^{-3}}$;
\item $\pi_7$ is generated by $N(\ot_3)a_{\lambda}^7a_{\sigma}^7$ and $N(\ot_2)u_{\lambda}u_{2\sigma}a_{\lambda}^2a_{\sigma}$, and one of them detects the third Hopf invariant one element $\sigma$.
\item $\pi_8$ is generated by $\Tr_2^4(\ot_2^2 \ot_1^2a_{\sigma_2}^8)+\Tr_2^4(\ot_3 \ot_1a_{\sigma_2}^8) + N(\ot_2)N(\ot_1)u_{2\sigma}^{2}a_{\lambda}^4$.
 \end{enumerate}
\end{thm}

In \cite{Rog:com}, Rognes shows that the unit map $S^0 \rightarrow (N_1^2 H\FF_2)^{hC_2}$ induces a splitting injection on mod $2$ homology as an $\mathcal{A}_*$-comodule thus a splitting injection on the $E_2$-page of the Adams spectral sequence. Therefore, the ring spectrum $(N_1^2 H\FF_2)^{hC_2} \simeq (a_{\lambda}^{-1}\BPfour)^{C_4}$ detects all Hopf invariant one elements. They all restrict to $0$, since the underlying Adams spectral sequence of $H\FF_2 \wedge H\FF_2$ is concentrated in filtration $0$. Therefore, they are detected by red subgroups in the corresponding degree.
	
The proof of \cref{thm:htpy} is by computing $a_\lambda^{-1}\SliceSS(\BPfour)$ and is given in the next section.  The most relevant differentials in the spectral sequence are listed in the following table:
\begin{table}[H]
\begin{tabular}{|l|l|l|}
\hline
Differential & Formula                                                                                                                                                                                                & Proof                                                                                                                                                      \\ \hline
$d_3$        & \begin{tabular}[c]{@{}l@{}}$d_3(u_{2\sigma_2}) = a_{\sigma_2}^3 (\ot_1 + \gamma \ot_1)$\\ $d_3(u_{\lambda}) = \Tr_2^4(a_{\sigma_2}^3\ot_1)$\end{tabular}                & \cref{prop-d3}                                                                                                                 \\ \hline
$d_5$        & $d_5(u_{2\sigma}) = N(\ot_1)a_{\lambda}a_{\sigma}^3$                                                                                                                                                   & \cref{thm-HHRd}                                                                                                                    \\ \hline
$d_5$        & $d_5(u_{\lambda}^2) = N(\ot_1)u_{\lambda}a_{\lambda}^2a_{\sigma}$                                                                                                                                      & \cref{prop-d5}                                                                                                                 \\ \hline
$d_7$        & \begin{tabular}[c]{@{}l@{}}$d_7(u_{2\sigma_2}^2) = a_{\sigma_2}^7(\ot_2 + \ot_1^3 + \gamma \ot_2)$\\ $d_7(2u_{\lambda}^2) = \Tr_2^4(a_{\sigma_2}^7\ot_1^3)$\end{tabular} & \begin{tabular}[c]{@{}l@{}}\cref{thm-C_2}\\ \cref{prop-d7}\end{tabular} \\ \hline
$d_7$        & $d_7(u_{\lambda}^4) = \Tr_2^4(\ot_1^3u_{2\sigma_2}^2a_{\sigma_2}^7)$                                                                                                                                      & \cref{prop-d7-2}                                                                                                               \\ \hline
$d_{13}$     & $d_{13}(u_{\lambda}^4a_{\sigma}) = N(\ot_2 + \ot_1^3 + \gamma(\ot_2))u_{2\sigma}^2a_{\lambda}^7$                                                                                                      & \cref{prop-d13-2}                                                                                                              \\ \hline
$d_{15}$     & $d_{15}(2u_{\lambda}^4) = \Tr_2^4(\ot_3^{C_2}a_{\sigma_2}^{15})$                                                                                                                                        & \cref{prop-d15}                                                                                                                \\ \hline
\end{tabular}
\end{table}

	\section{Computing the localized slice spectral sequences of $BP^{(\!(G)\!)}$} \label{section:sliceSScomputation}

In this section, we compute $a_\lambda^{-1}\SliceSS(\BPfour)$ and prove \cref{thm:htpy}. Our approach is similar to that of \cite{HHR:C4} and \cite{HSWX}.  When going through the computations in this section, the following guiding principles are useful to keep in mind.  We hope these points would serve as a road map that will be helpful to the readers who are new to these types of computations.
\begin{enumerate}
\item The $E_2$-page of the spectral sequence can be obtained by computing the $RO(C_4)$-graded homotopy groups of $a_\lambda^{-1}\HZ$.
\item The $C_2$-level spectral sequence, $a_{\sigma_2}^{-1}\SliceSS(\BPR \wedge \BPR)$, is easy to compute, as it is completely determined by the Hill--Hopkins--Ravenel slice differentials.  
\item In the positive cone part of $a_\lambda^{-1} \SliceSS(\BPfour)$ (which includes the entire integer-graded spectral sequence), the only algebra generators that are not permanent cycles are essentially classes of forms $u_V$ and $u_Va_\sigma$.  Therefore, we only need to focus on finding differentials on these classes, and then use the Leibniz rule.  This is why even though the integer-graded spectral sequence is the computation of interest, we often move to analyze certain classes in $RO(C_4)$-degrees.  
\item Many of the differentials are proven by using the $C_2$-level spectral sequence, and using the restrictions and transfers on the $E_2$-page.  More precisely, if one knows that $d_r(\Res_{C_2}^{C_4}x) = y$, then $x$ must support a differential of length at most $r$.  Similarly, if $d_r(x) = y$, and $\Tr_{C_2}^{C_4}(y)$ is not zero on the $E_2$-page, then it must be killed by a differential of length at most $r$.  
\item The remaining differentials and extension are proven by using the Hill--Hopkins--Ravenel norm and the theory of exotic restrictions and transfers.  
\end{enumerate}
We would like to also remark that the differentials proven in this section determine all the differentials in the integer-graded spectral sequence in our range of interest.  There are other differentials in the $RO(C_4)$-graded page (both in the positive cone and outside the positive cone) that don't influence the integer-graded page of the spectral sequence.    

\subsection{Computing the $E_2$-page}	
	 We will first give a complete algebraic description of the $E_2$-page of $a_{\lambda}^{-1}\SliceSS(\BPfour)$ in terms of generators and relations. To do so, by \cref{thm-main}, we need to describe the $C_2$-homotopy groups $\U{\pi}_{\star} (a_{\sigma_2}^{-1}H\UZ)$ and the $C_4$-homotopy groups $\U{\pi}_{\star}(a_{\lambda}^{-1}H\UZ)$.
	
	\begin{prop}\label{prop-C_2coe}
We have
		\[
		\pi_{\star}^{C_2}(a_{\sigma_2}^{-1}H\UZ) = \FF_2[u_{2\sigma_2},a_{\sigma_2}^{\pm 1}].
		\]
     The Mackey functor structure is determined by the contractibility of the underlying spectrum.
	\end{prop}
    This proposition is proved by a standard Tate cohomology computation, see \cite[Section~2.C]{Greenlees-Four} for details.
	
Let $S$ be the subring of
		\[
		R = \ZZ/4[a_{\sigma},u_{2\sigma}^{\pm 1},u_{\lambda}a_{\lambda}^{-1}]/(2a_{\sigma},  u_{\lambda}a_{\lambda}^{-1}a_{\sigma}^2 = 2u_{2\sigma})
		\]
generated by the elements $\{a_{\sigma}, u_{2\sigma}, u_{\lambda}a_{\lambda}^{-1}, 2u_{2\sigma}^k, u_{2\sigma}^ku_{\lambda}a_{\lambda}^{-1}\,|\, k < 0 \}$, and let $M = \ZZ/2[u_{2\sigma}^{\pm 1},u_{\lambda}a_{\lambda}^{-1}, a_{\sigma}^{\pm 1}]/(u_{2\sigma}^{\infty},a_{\sigma}^{\infty})$ be considered as a module over $S$.  Here, $R[x^{\pm 1}]/(x^{\infty})$ is the cokernel of the map $R[x] \rightarrow R[x^{\pm 1}]$.
	
\begin{prop}\label{prop-C_4coe}
We have
		\[
		\pi_{\star}^{C_4}(a_{\lambda}^{-1}H\UZ) = (S \oplus \Sigma^{-1}M)[a_{\lambda}^{\pm 1}],
		\]
		where $S \oplus \Sigma^{-1}M$ is the square-zero extension of $M$ over $S$ of degree $-1$.
		
		The Green functor structure is determined by the following facts:
    \begin{enumerate}
        \item The $C_2$-restriction of $a_{\lambda}^{-1}H\UZ$ is the spectrum $a_{\sigma_2}^{-1}H\UZ$ in \cref{prop-C_2coe}.
        \item The $C_2$-restrictions of the classes $u_{\lambda}$ and $u_{2\sigma}$ are $u_{2\sigma_2}$ and $1$, respectively.
        \item Given $V \in RO(C_4)$, there is an exact sequence (see \cite[Lemma~4.2]{HHR:C4})
            \[
                \pi^{C_2}_{i^*_{C_2}V} X \xrightarrow{\Tr_2^4} \pi^{C_4}_{V} X \xrightarrow{a_{\sigma}} \pi^{C_4}_{V - \sigma} X \xrightarrow{\Res^4_2} \pi^{C_2}_{i^*_{C_2}V - 1} X.
            \]
In other words, the kernel of $a_{\sigma}$-multiplication is the image of the transfer from $C_2$ to $C_4$, and the image of $a_{\sigma}$-multiplication is the kernel of the restriction from $C_4$ to $C_2$.
    \end{enumerate}
	\end{prop}

The proof of Proposition~\ref{prop-C_4coe} and a more explicit presentation of the Mackey functor are given in \cite[Proposition~6.7]{Zeng:HZ}.  Fortunately, in most of the paper we only need the "positive cone" of the coefficient Green functor, that is, the part $\star = a +b\sigma + c\lambda$ for $b \leq 0$. The Green functor structure of this part is computed in \cite[Section~3]{HHR:C4}. However, the other part also plays an important role on the computation, see for example the proofs of \cref{prop-ext4} and \cref{prop-ext12}. 

The relation $u_{\lambda}a_{\lambda}^{-1}a_{\sigma}^2  = 2u_{2\sigma}$ and its integral version $u_{\lambda}a_{\sigma}^2  = 2u_{2\sigma}a_{\lambda}$ are commonly called the \emph{gold relation} (see \cite[Lemma~3.6]{HHR:C4}).

Figure~\ref{fig-C4MF} gives the Lewis diagrams (first introduced in \cite{Lewis:ROG}) we use for $C_4$-Mackey functors, where restrictions $\Res^G_H$ map downwards and transfers $\Tr^G_H$ map upwards. These notations are consistent with \cite[Section~5]{HHR:C4}.
	
	\begin{figure}[H]
		\includegraphics{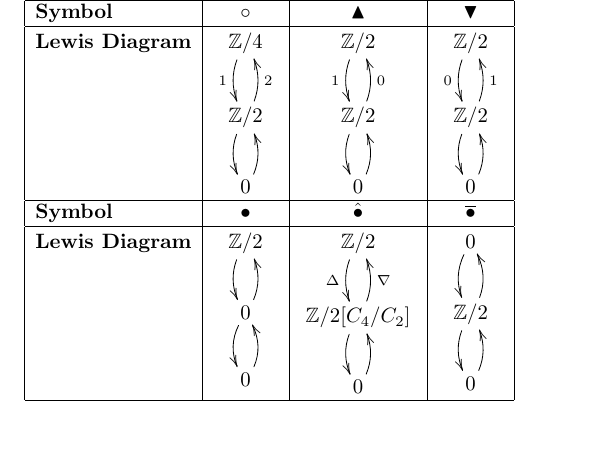}
		\caption{\label{fig-C4MF}Table of $C_4$-Mackey functors}
	\end{figure}

	Figure~\ref{fig-lambda} shows $\U{\pi}_{a+b\sigma}(a_{\lambda}^{-1}H\UZ)$ in the range $-6 \leq a,b \leq 6$.  In the figure, the horizontal coordinate is $a$ and the vertical coordinate is $b$. Vertical lines are $a_{\sigma}$-multiplications, where solid lines are surjections and the dashed lines represent maps of the form $\ZZ/2 \hookrightarrow \ZZ/4$.
	\begin{figure}[H]
		\includegraphics{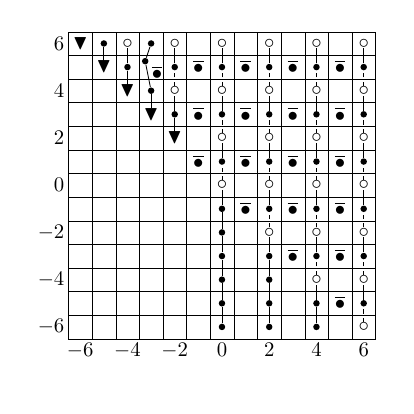}
		\caption{$\U{\pi}_{a+b\sigma}(a_{\lambda}^{-1}H\UZ)$ for $-6 \leq a, b \leq 6$.}\label{fig-lambda}
	\end{figure}

Although we mostly care the most about the $C_4$-equivariant homotopy groups of $a_\lambda^{-1}\BPfour$, there are two advantages for computing $a_\lambda^{-1}\SliceSS(\BPfour)$ as a spectral sequence of Mackey functors:
\begin{enumerate}
\item The Mackey functor structure can transport certain differentials on the $C_2$-level to differentials on the $C_4$-level.
\item The Mackey functor structure and $d_r$-differentials can result in exotic extensions of filtration $r-1$ (see Section \ref{sec-transfer}).
\end{enumerate}

We will see (1) in the computations of $d_3$, $d_7$, and $d_{15}$-differentials below.  (2) will be used to prove certain extensions forming the $(\ZZ/4)$s in \cref{thm:htpy}, see \cref{prop-ext4} and \ref{prop-ext12}.

\begin{nota}\label{nota:transfer}
Let $V \in RO(H)$ be a virtual representation that is in the image of the restriction $i^*_H:RO(G) \rightarrow RO(H)$. Then for any preimage $W$ of $V$, there is a transfer map
\[
    \Tr_H^{G,W} : \pi^H_V X \rightarrow \pi^G_W X,
\]
as a part of the homotopy Mackey functor structure. In our computation we will omit writing $W$ when it is clear from the context what $W$ is.
\end{nota}

\subsection{The $C_2$-spectral sequence}
    We start our computation with the $C_2$-underlying spectral sequence of $a_{\lambda}^{-1}\SliceSS(\BPfour)$.
	
	\begin{thm}\label{thm-C_2}\hfill
		\begin{enumerate}
			\item The underlying $C_2$-spectral sequence of $a_{\lambda}^{-1}\SliceSS(\BPfour)$ is ${a_{\sigma_2}^{-1}\SliceSS(BP_{\RR} \wedge BP_{\RR})}$. Its $E_2$-page is
			\[
			a_{\sigma_2}^{-1}H\UZ_{\star}[\ot_1,\gamma \ot_1,\ot_2,\gamma \ot_2,\cdots].
			\]
			More precisely, the $E_2$-page of the underlying non-equivariant spectral sequence is trivial, and the $E_2$-page of the $C_2$-spectral sequence is
			\[
			\FF_2[u_{2\sigma_2},a_{\sigma_2}^{\pm 1}][\ot_1,\gamma \ot_1, \ot_2, \gamma \ot_2,\cdots].
			\]
				The elements $u_{2\sigma_2}, \ot_i$ and $\gamma\ot_i$ have filtration $0$, while $a_{\sigma_2}$ has filtration $1$.\footnote{We recall the convention here that the filtration of an element in $\pi_V^HP^n_nX$ in the slice spectral sequence for some $X$ is in filtration $n-\dim_{\mathbb{R}}V$. In particular the classes $a_V$ will be always in filtration $\dim_{\mathbb{R}}V$.}
			\item All the differentials in $a_{\sigma_2}^{-1}\SliceSS(\BPR \wedge \BPR)$ are determined by $a_{\sigma_2}$, $\ot_i$ and $\gamma \ot_i$ being permanent cycles, the differentials
			\[
			d_{2^{k+1}-1}(u_{2\sigma_2}^{2^{k-1}}) = a_{\sigma_2}^{2^{k+1}-1} \sum_{i = 0}^{k} \ot_{k-i}^{2^i}\gamma \ot_i, \,\,\,\,\, k \geq 1
			\]
			and the Leibniz formula (for notational convenience, we let $\ot_0 = \gamma \ot_0 = 1$).  The $E_{2^{k+1}}$-page has the form
			\[
			\FF_2[u_{2{\sigma_2}}^{2^k},a_{\sigma_2}^{\pm 1}][\ot_1,\gamma \ot_1,\cdots]/(\ov_1,\ov_2,\cdots,\ov_{k})
			\]
			where $\ov_k = \sum\limits_{i = 0}^{k} \ot_{k-i}^{2^i}\gamma \ot_i$.
			\item The $E_{\infty}$-page of $a_{\sigma_2}^{-1}\SliceSS(\BPR \wedge \BPR)$ is
			\[
			\FF_2[a_{\sigma_2}^{\pm 1}][\ot_1,\gamma \ot_1,\cdots]/(\ov_1,\ov_2,\cdots)
			\]
			In particular, in the integral grading, all the stem-$n$ non-trivial permanent cycles are located in filtration $n$.
		\end{enumerate}
	\end{thm}

	\begin{proof}
		For $(1)$, note that since $i^*_{C_2}\BPfour = BP_{\RR} \wedge BP_{\RR}$, the $C_2$-underlying slice spectral sequence of $\SliceSS(\BPfour)$ is $\SliceSS(\BPR \wedge \BPR)$.  Moreover, $i^*_{C_2}a_{\lambda} = a_{\sigma_2}^2$.  Therefore inverting $a_{\lambda}$ in the $C_4$-spectral sequence inverts $a_{\sigma_2}$ in the underlying $C_2$-spectral sequence.
		
		For $(2)$, we use the Hill--Hopkins--Ravenel slice differential theorem \cite[Theorem~9.9]{HHR} and the formula in \cite[Theorem~1.1]{BHSZ} that expresses the $\bar{v}_i$-generators in terms of the $\ot_i$-generators (our $\ov_i$ and $\ot_i$ are $\ot_i^{C_2}$ and $\ot_i^{C_4}$ respectively in \cite{BHSZ}). The Hill--Hopkins--Ravenel slice differential theorem states that in the slice spectral sequence of $BP_{\RR}$, there are differentials
		\[
		d_{2^{k+1}-1}(u_{2\sigma_2}^{2^{k-1}}) = \ov_k a_{\sigma_2}^{2^{k+1}-1}, \,\,\,\,\, k \geq 1.
		\]
		The formula in \cite[Theorem~3.1]{BHSZ} shows that under the left unit map $BP_{\RR} \rightarrow BP_{\RR} \wedge BP_{\RR}$,
		\[
		\ov_k = \sum\limits_{i = 0}^k \ot_{k-i}^{2^i} \gamma \ot_i \textrm{ mod $(2,\ov_1,\cdots,\ov_{k-1})$}.
		\]
		The left unit map induces a map
		$$a_{\sigma_2}^{-1}\SliceSS(\BPR) \longrightarrow a_{\sigma_2}^{-1}\SliceSS(\BPR \wedge \BPR)$$
		of spectral sequences.  We will use naturality and induction to obtain the differentials and the description of the $E_{2^{k+1}}$-page.
		
To start the induction process, note that the description of the $E_2$-page is already given in (1). Now assume that we have obtained a description of the $E_{2^k}$-page.  For degree reasons, the next potential differential is of length exactly ${2^{k+1}-1}$. The differential formula for $a_{\sigma_2}^{-1}\SliceSS(\BPR)$ above shows that for any polynomial $P \in \FF_2 [\ot_1,\gamma \ot_1,\cdots]/(\ov_1,\ov_2,\cdots,\ov_{k-1})$ and $l$ an odd number, we have the differential
		\[
		d_{2^{k+1}-1}(Pu_{2\sigma_2}^{2^k l}) = P\ov_k u_{2\sigma_2}^{2^{k-1} l}a_{\sigma_2}^{2^{k+1}-1}
		\]
in $a_{\sigma_2}^{-1}\SliceSS(\BPR \wedge \BPR)$. The source and the target of this differential are always non-zero on the $E_{2^{k}}$-page because the sequence $(\ov_1,\ov_2,\cdots)$ is a regular sequence in the polynomial ring $\FF_2[\ot_1,\gamma \ot_1,\cdots]$.  Taking the quotient of the kernel and cokernel of this differential, we see that the $E_{2^{k+1}}$-page has the above description.
		
		(3) is a direct consequence of (2) by letting $k \rightarrow \infty$. See \cref{fig-C_2} for the integral $E_2$ and $E_{\infty}$-pages of this spectral sequence.
	\end{proof}

\begin{figure}
\begin{center}
    \includegraphics[scale = 0.7]{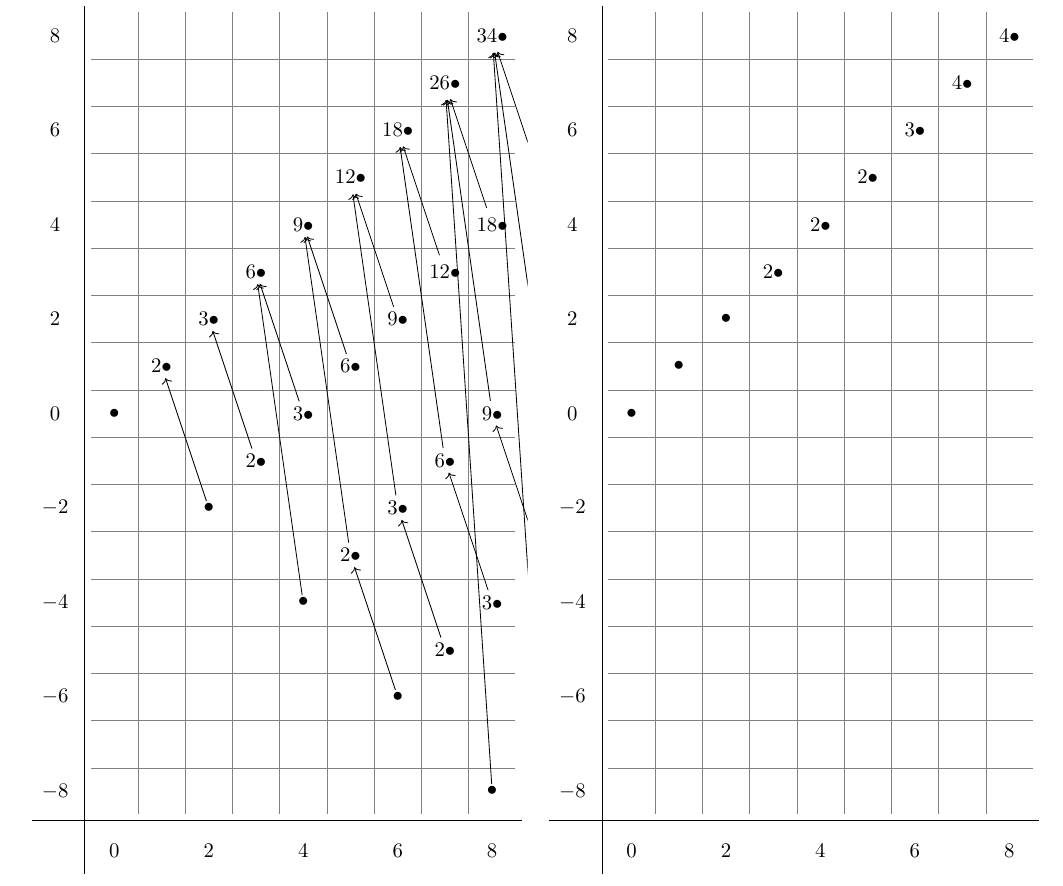}
\end{center}
    \caption{\label{fig-C_2}The integral $E_2$- and $E_{\infty}$-pages of $a_{\sigma_2}^{-1}\SliceSS(BP_{\RR} \wedge BP_{\RR})$}
\end{figure}

\begin{rmk}\rm
		In \cref{prop:ti,} we show that the $C_2$-geometric fixed points of the $\ot_i$ and $\gamma \ot_i$ generators are the $\xi_i$ and $\zeta_i$ generators in the mod $2$ dual Steenrod algebra $\mathcal{A}_*$. Therefore, the formula 
		\[
		\ov_k = \sum\limits_{i = 0}^k \ot_{k-i}^{2^i} \gamma \ot_i \textrm{ mod $(2,\ov_1,\cdots,\ov_{k-1})$}
		\]
		reduces to Milnor's conjugation formula $0 = \sum\limits_{i = 0}^k \xi_{k-i}^{2^i} \zeta_i$ in $\mathcal{A}_*$.
\end{rmk}

\subsection{The $C_4$-spectral sequence: $d_3$, $d_5$ and $d_7$-differentials}
	The rest of this section is dedicated to computing the first $8$ stems of the $C_4$-Mackey functor homotopy groups of $a_{\lambda}^{-1}\BPfour$.  The result is stated in \cref{thm:htpy}. By Section \ref{sec-norm}, we are free to use the norm structure from $C_2$ to $C_4$ in the localized slice spectral sequence.
	
	As a consequence of the slice theorem \cite[Theorem~6.1]{HHR}, the $0$-th slice of $MU^{(\!(G)\!)}$ is $H\UZ$ and $\U{\pi}_0MU^{(\!(G)\!)} \cong \UZ$. Therefore, every Mackey functor in the (localized) slice spectral sequence  and the homotopy of any $MU^{(\!(G)\!)}$-module is a module over $\UZ$. By \cite[Theorem~16.5]{Thevenaz-Webb}, we have the following proposition.
	\begin{prop}\label{prop-cohomological}
	    Let $K \subset H \subset G$, and $x$ be an element in the $G/H$-level of a Mackey functor either in the (localized) slice spectral sequence or the homotopy of a $MU^{(\!(G)\!)}$-module, then 
	    \[
	    \Tr_K^H(\Res_K^H(x)) = [H:K]x.
	    \]
	\end{prop}

	Before getting to the page-by-page computation, we note that all the differentials on the classes $u_{2\sigma}^{2^k}$ for $k \geq 0$ are already known by the work of Hill--Hopkins--Ravenel. Their theorem is originally formulated for the slice spectral sequence for $MU^{(\!(C_4)\!)}$ and the exact same statement and proof carries over to $\SliceSS(\BPfour)$ and $a_{\lambda}^{-1}\SliceSS(\BPfour)$. 
	
\begin{thm}[{\cite[Theorem~9.9]{HHR}}]\label{thm-HHRd}
	For $k \geq 0$ and $i < 2^{k+3} - 3$, $d_i(u_{2\sigma}^{2^k}) = 0$ and
	\[
	d_{2^{k+3} - 3}(u_{2\sigma}^{2^k}) = N(\ot_{k+1})a_{\lambda}^{2^{k+1}-1}a_{\sigma}^{2^{k+2}-1}.
	\]
\end{thm}
		
	\begin{figure}
		\includegraphics[scale = 0.5,page = 2]{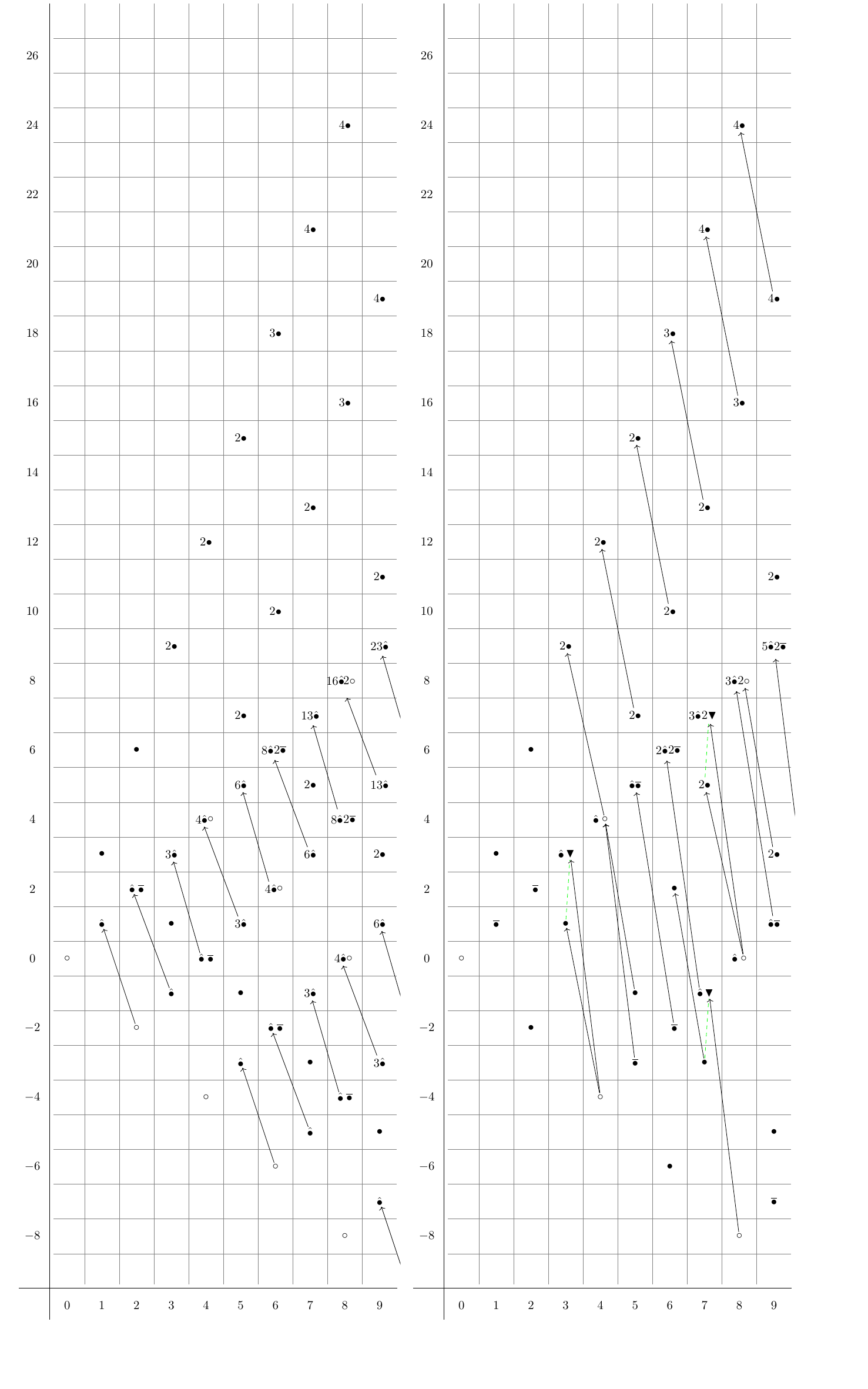}
        \caption{\label{fig-E2E4}Left: $d_3$-differentials in $a_{\lambda}^{-1}\SliceSS(\BPfour)$.\\Right: $d_5$- and $d_7$-differentials in $a_{\lambda}^{-1}\SliceSS(\BPfour)$.}
	\end{figure}
	
	Now we start the page-by-page computation.  First, note that for degree reasons all the differential lengths will be odd. 
	
	\begin{prop}\label{prop-d3}
		\[
		d_{3}(u_{\lambda}) = \Tr_2^4(\ot_1 a_{\sigma_2}^3)
		\]
	\end{prop}
	
	\begin{proof}
		By \cref{thm-C_2}, the restriction $\Res^4_2(u_{\lambda}) = u_{2\sigma_2}$ supports the differential
		\[
		d_3(u_{2\sigma_2}) = (\ot_1 + \gamma \ot_1)a_{\sigma_2}^3
		\]
		in the $C_2$-spectral sequence. By naturality and degree reasons, the class $u_{\lambda}$ must also support a $d_3$-differential in the $C_4$-spectral sequence whose target restricts to the class $(\ot_1 + \gamma \ot_1)a_{\sigma_2}^3$. The only class that restricts to $(\ot_1 + \gamma \ot_1)a_{\sigma_2}^3$ with $RO(C_4)$-degree $1-\lambda$ is $\Tr_2^4(\ot_1 a_{\sigma_2}^3)$.
	\end{proof}
	
	In \cref{fig-E2E4}, this proposition gives all $d_3$ coming out of $\circ$, namely $u_{\lambda}a_{\lambda}^{-1}$ at $(2,-2)$, $N(\ot_1)^2u_{\lambda}u_{2\sigma}a_{\lambda}$ at $(6,2)$ and $u_{\lambda}^3 a_{\lambda}^{-3}$ at $(6,-6)$.
	
	\begin{cor}\label{cor-d3}
	    Let $P$ be a polynomial of $\ot_i$, $\gamma \ot_i$, $a_{\sigma_2}$, then
	    \[
	    d_3(u_{\lambda}^{2k+1}Tr_2^{4,V}(P)) =u_{\lambda}^{2k} \Tr_2^{4,V  - \lambda + 1}(P(\ot_1 +\gamma \ot_1)a_{\sigma_2}^3)
	    \]
	    for all $k \geq 0$ and any $V \in RO(C_4)$ restricting to the $RO(C_2)$-degree of $P$.
	\end{cor}
	
	\begin{proof}
	    This is a direct consequence of \cref{prop-d3}, the Frobenius relation \cite[Definition~2.3]{HHR:C4} and the Leibniz rule.
	\end{proof}
	
	As displayed in \cref{fig-E2E4}, this corollary gives all other $d_3$-differentials. We now explain them in detail.
	
	In terms of Mackey functors, the $d_3$-differentials give the following exact sequences:
	\begin{align*}
	0 \rightarrow \bullet \rightarrow & \circ \xrightarrow{d_3} \hat{\bullet} \rightarrow \Ov{\bullet} \rightarrow 0\\
	0 \rightarrow & \hat{\bullet} \xrightarrow{d_3} \hat{\bullet} \rightarrow 0\\
	0 \rightarrow & \Ov{\bullet} \xrightarrow{d_3} \hat{\bullet} \rightarrow \JJ \rightarrow 0.
	\end{align*}
	
	Here are examples of $d_3$-differentials corresponding to each exact sequence above:
	\begin{align*}
	d_{3}(u_{\lambda}) & = \Tr_2^4(\ot_1 a_{\sigma_2}^3)\\
	d_{3}(\Tr_2^4(\ot_1 a_{\sigma_2})u_{\lambda}) &= \Tr_2^4(\ot_1 (\ot_1 + \gamma \ot_1) a_{\sigma_2}^4) 
	\\
	d_{3}(u_{2\sigma_2} a_{\sigma_2}) & = (\ot_1 + \gamma\ot_1) a_{\sigma_2}^4.
	\end{align*}
	
	Note that the last differential is a $C_2$-differential, but it has an effect on the $C_4$-level Mackey functor structure. By results in Section \ref{sec-transfer}, the $d_3$-differentials also give certain exotic restrictions of filtration jump at most $2$ (that is, the image of the restriction is of filtration at most 2 higher than the source). For example, consider the element $N(\ot_1)u_{\lambda}a_{\sigma}$ at $(3,1)$.  This class is a $d_3$-cycle. By Proposition~\ref{prop-d3}, the class $N(\ot_1)u_{\lambda}$ supports the $d_3$-differential
	$$d_3(N(\ot_1)u_{\lambda})=\Tr_2^4(\ot_1^2 \gamma \ot_1 a_{\sigma_2}^3).$$ 
	By \cref{prop-transfer}, the class $\ot_1^2 \gamma \ot_1 a_{\sigma_2}^3$ receives an exotic restriction of filtration jump at most $2$ in integral degree, and the only possible source is $N(\ot_1)u_{\lambda}a_{\sigma}$. The same argument applies to all $2$-torsions classes with $(t-s, s)$-bidegrees $(3+4i+4j,1+4i-4j)$ for $i,j \geq 0$. The exotic restrictions are represented by the vertical green dashed lines in \cref{fig-E2E4}.   
	
	\begin{rmk}\label{rmk-exotic} \rm
		These exotic restrictions are the first family of examples of an interesting phenomenon in the $RO(G)$-graded spectral sequence of Mackey functors. Exotic restrictions and transfers can imply nontrivial abelian group extensions. By \cref{prop-cohomological}, the transfer of a restriction of a class must be the multiple of this class by the index of the subgroup. Therefore, As Mackey functors, these extensions are of the form
		\[
		0 \rightarrow \bullet \rightarrow \circ \rightarrow \JJ \rightarrow 0,
		\]
		which represents a nontrivial extension
		\[
		0 \rightarrow \ZZ/2 \rightarrow \ZZ/4 \rightarrow \ZZ/2 \rightarrow 0
		\]
		if one evaluates the exact sequence of Mackey functors at $C_4/C_4$. Notice that in the category of Mackey functors, there are essentially two nontrivial extensions between $\bullet$ and $\JJJ$, but only the one above fits into \cref{prop-cohomological}.
		
		For readers who are familiar with Lubin--Tate $E$-theories and topological modular forms, the family of 2-extensions above is a generalization  of the type of 2-extension between the class $\nu$ at $(3,1)$ and the class $2\nu$ at $(3,3)$ in the homotopy fixed points spectral sequences of $E_2^{hC_4}$ and $TMF_0(5)$ (see \cite{BBHS} and \cite{Behrens-Ormsby}).
	\end{rmk}
	
	In summary, the $d_3$-differentials can be described as follows: 
	\begin{enumerate}
		\item On $C_2$-level, it is the first differential in \cref{thm-C_2}.
		\item The Green functor structure of the spectral sequence gives $d_3$-differentials on the $C_4$-level, by \cref{prop-d3} and \cref{cor-d3}. After these $d_3$-differentials, there is no room for further $d_3$-differentials. 
		
		\item Every $d_3$-differential of the form $\Ov{\bullet} \rightarrow \hat{\bullet}$ gives an extension of filtration $2$ by the above remark.
	\end{enumerate}

	Now we will prove the $d_5$-differentials. There are two different types of $d_5$-differentials. The first type is given by \cref{thm-HHRd}:
	\[
	d_5(u_{2\sigma}) = N(\ot_1)a_{\lambda}a_{\sigma}^3.
	\]
	Since $N(\ot_1)$ and $a_{\lambda}$ are both permanent cycles, on the integral page for our range, it gives the following $d_5$-differential at $(4,4)$:
	\[
	d_5(N(\ot_1)^2u_{2\sigma}a_{\lambda}^2) = N(\ot_1)^2a_{\lambda}^3a_{\sigma}^3, 
	\]
	and it repeats by multiplying by $N(\ot_1)a_{\lambda}a_{\sigma}$. In Figure~\ref{fig-E2E4}, these are the $d_5$-differentials with sources on or above the line of slope $1$. 
	
	The second type of $d_5$-differentials is given by the following proposition. 
	\begin{prop}\label{prop-d5}
		\begin{align*}
		d_5(u_{\lambda}^2) & = N(\ot_1)u_{\lambda}a_{\lambda}^2a_{\sigma},\\
		d_5(u_{\lambda}^2a_{\sigma}) & = 2N(\ot_1)u_{2\sigma}a_{\lambda}^3.
		\end{align*}
	\end{prop}
	\begin{proof}
	    The restriction $\Res^4_2(u_{\lambda}^2) = u_{2\sigma_2}^2$ supports the $d_7$-differential
	    \[
	    d_7(u_{2\sigma_2}^2) = (\ot_2 +\gamma \ot_2 + \ot_1^3 )a_{\sigma_2}^7
	    \]
	    by \cref{thm-C_2}. By naturality, $u_{\lambda}^2$ must support a differential of length at most 7.  For degree reasons, the length of this differential can only be 5 or 7.  If the length of this differential is 7, the target must restrict to the class $(\ot_2 +\gamma \ot_2 +\ot_1^3 )a_{\sigma_2}^7$.  However, this class is not in the image of the restriction map $\Res^4_2$. Therefore, $u_{\lambda}^2$ must support a $d_5$-differential.  The only possible target of this $d_5$-differential is $N(\ot_1)u_{\lambda}a_{\lambda}^2a_{\sigma}$.  This proves the first $d_5$-differential
	    
	    Multiplying with $a_{\sigma}$ on both sides of the first $d_5$-differential gives
	    \[
	    d_5(u_{\lambda}^2a_\sigma) = N(\ot_1)u_{\lambda}a_{\lambda}^2a_{\sigma}^2.
	    \]
	    Applying the gold relation $u_\lambda a_{\sigma}^2 = 2u_{2\sigma}a_\lambda$ gives the second $d_5$-differential.
	\end{proof}
	
	In \cref{fig-E2E4}, the $d_5$-differentials in Proposition~\ref{prop-d5} can be seen on the following classes: 
	\begin{enumerate}
	\item $u_{\lambda}^2a_{\lambda}^{-2}$ at $(4,-4)$, 
	\item $N(\ot_1)u_{\lambda}^2a_{\lambda}^{-1}a_\sigma$ at $(5,-1)$, 
	\item $N(\ot_1)^2u_{\lambda}^2u_{2\sigma}$ at $(8,0)$, 
	\item $N(\ot_1)^3u_{\lambda}^2u_{2\sigma}a_\lambda a_\sigma$ and $N(\ot_2)u_{\lambda}^2u_{2\sigma}a_\lambda a_\sigma$ at $(9,3)$. 
	\end{enumerate}
	
	\text{}\begin{rmk}\rm
		Although $u_{\lambda}^2$ and $u_{\lambda}^2a_{\sigma}$ support differentials of the same length, this is not true in general. For example, we will see soon that $u_{\lambda}^4$ supports a $d_7$-differential, while $u_{\lambda}^4a_{\sigma}$ supports a $d_{13}$-differential.
	\end{rmk}
	
	\begin{cor}
		\[
		d_5(u_{\lambda}^3a_{\sigma}) = 2N(\ot_1)u_{\lambda}u_{2\sigma}a_{\lambda}^3.
		\]
	\end{cor}
	
	\begin{proof}
		First, we will show that $u_{\lambda}a_{\sigma}$ is a nontrivial permanent cycle. Since the target of the $d_3$-differential on $u_\lambda$ is a transfer class, it is killed by $a_{\sigma}$, and therefore $u_\lambda a_\sigma$ is a $d_3$-cycle. The only potential non-trivial differential that $u_\lambda a_\sigma$ can support is the $d_5$-differential 
		$$d_5(u_\lambda a_\sigma) =  N(\ot_1)a_{\lambda}^2a_{\sigma}^2.$$
		If this differential happens, then multiplying $a_\sigma$ on both sides and using the gold relation will produce the differential 
		$$d_5(2u_{2\sigma}a_{\lambda}) = N(\ot_1)a_{\lambda}^2a_{\sigma}^3.$$
		This is a contradiction to \cref{thm-HHRd}. 
		
		Applying the Leibniz rule on the first $d_5$-differential in Proposition~\ref{prop-d5} with the class $u_\lambda a_\sigma$ produces the $d_5$-differential 
		\[
		d_5(u_{\lambda}^3a_{\sigma}) = u_{\lambda}a_{\sigma}d_5(u_{\lambda}^2) = N(\ot_1)u_{\lambda}^2a_{\lambda}^2a_{\sigma}^2 = 2N(\ot_1)u_{\lambda}u_{2\sigma}a_{\lambda}^3.\qedhere
		\]
	\end{proof}

	In \cref{fig-E2E4}, this $d_5$-differential implies the $d_5$-differential on the class $N(\ot_1)u_{\lambda}^3 a_{\lambda}^{-2}a_\sigma$ at $(7,-3)$. Notice that the class $N(\ot_1)u_{\lambda}u_{2\sigma}a_{\lambda}^2$ supports a $d_3$-differential and the class $2N(\ot_1)u_{\lambda}u_{2\sigma}a_{\lambda}^2$ is killed by a $d_5$-differential. In the integral grading, this happens to the $\ZZ/4$ in $(6,2)$.

	There are extensions of filtration jump 4 induced by the $d_5$-differentials. 
	
	\begin{prop}\label{prop-ext4}
		There is an exotic transfer of filtration jump $4$ from $(2,2)$ to $(2,6)$:
		\[
		\Tr_2^4(\ot_1^2a_{\sigma_2}^2) = N(\ot_1)^2a_{\lambda}^2a_{\sigma}^2.
		\]
		There is an exotic restriction of filtration jump $4$, from $(2,-2)$ to $(2,2)$:
		\[
		\Res_2^4\left(2u_{\lambda}a_{\lambda}^{-1}\right) = \ot_1^2a_{\sigma_2}^2.
		\]
	\end{prop}
	
	\begin{proof}
	    We use \cref{prop-transfer} to prove both extensions.
	    
		For the first claim, note that $d_5(N(\ot_1)u_{2\sigma}a_{\lambda})=N(\ot_1)^2a_{\lambda}^2a_{\sigma}^3$, and $N(\ot_1)^2a_{\lambda}^2a_{\sigma}^2$ is a nontrivial $d_5$-cycle. Therefore, $N(\ot_1)^2a_{\lambda}^2a_{\sigma}^2$ is the target of an exotic transfer of filtration jump $4$ in $E_6$, and the only possible source is $\ot_1^2a_{\sigma_2}^2$.
		
		For the second claim, first note that by \cref{prop-C_4coe} (also see \cref{fig-lambda}) and the gold relation,
		\[
		2u_{\lambda}a_{\lambda}^{-1} = \left(\frac{u_{\lambda}^2}{u_{2\sigma}}a_{\lambda}^{-2}a_{\sigma}\right)a_{\sigma}.
		\]
		We have the $d_5$-differential
		\[
		d_5\left(\frac{u_{\lambda}^2}{u_{2\sigma}}a_{\lambda}^{-2}a_{\sigma}\right) = \Tr_2^4(\ot_1^2a_{\sigma_2}^2).
		\]
		To prove this differential, consider the class $\frac{u_{\lambda}^2}{u_{2\sigma}}a_{\lambda}^{-2}$.  This class supports a $d_5$-differential because after multiplying it by $u_{2\sigma}^2 a_{\lambda}^2$ (which is a $d_5$-cycle), the class $u_{\lambda}^2u_{2\sigma}$ supports the $d_5$-differential 
		$$d_5(u_{\lambda}^2u_{2\sigma}) = N(\ot_1)u_{\lambda}u_{2\sigma}a_{\lambda}^2a_{\sigma}$$
		by Proposition~\ref{prop-d5}.  Therefore
		\[
		d_5\left(\frac{u_{\lambda}^2}{u_{2\sigma}}a_{\lambda}^{-2}\right) = N(\ot_1)\frac{u_{\lambda}}{u_{2\sigma}}a_{\sigma}.
		\]
		Multiplying both sides by $a_{\sigma}$, we have
		\[
		d_5\left(\frac{u_{\lambda}^2}{u_{2\sigma}}a_{\lambda}^{-2}a_{\sigma}\right) = N(\ot_1)\frac{u_{\lambda}}{u_{2\sigma}}a_\sigma^2 = 2N(\ot_1)a_{\lambda} = \Tr_2^4(\Res_2^4(N(\ot_1)a_{\lambda}))=\Tr_2^4(\ot_1 \gamma\ot_1 a_{\sigma_2}^2) = \Tr_2^4(\ot_1^2 a_{\sigma_2}^2)
		\]
		The last equation holds because by \cref{thm-C_2}, $\ot_1 = \gamma \ot_1$ after the $d_3$-differentials in the $C_2$-spectral sequence.
		
		Therefore, $\ot_1^2 a_{\sigma_2}^2$ must receive an exotic restriction of filtration jump $4$ in the integral degree, and the only source of the restriction is $2u_{\lambda}a_{\lambda}^{-1}$.
	\end{proof}
	
	In Figure~\ref{fig-E11E16}, the exotic restrictions and transfers are the green and blue dashed lines, respectively.
	
	\begin{rmk}\rm
		Similar to Remark \ref{rmk-exotic}, the exotic restrictions and transfers also give extensions of abelian groups on the $C_4$-level. The situation is more subtle here because each individual exotic extension doesn't involve non-trivial extensions of abelian groups at any level. When we combine the two extensions together, however, we obtain an abelian group extension of filtration $8$ from $(2,-2)$ to $(2,6)$:
		\[
		0 \rightarrow \ZZ/2 \rightarrow \ZZ/4 \rightarrow \ZZ/2 \rightarrow 0,
		\]
		and $2(2u_{\lambda}a_{\lambda}^{-1}) = N(\ot_1)^2a_{\lambda}^2a_{\sigma}^2$ in homotopy. This extension is similar to the extension in the 22-stem of $E_2^{hC_4}$ and $TMF_0(5)$. (See \cite[Figure~10]{BBHS} and \cite[Section~2]{Behrens-Ormsby}).
	\end{rmk}
	
	We will now prove the $d_7$-differentials. While we state them in some $RO(C_4)$-graded page first, we recommend that the reader multiplies with appropriate powers of $a_{\lambda}$ whenever possible to visualize the arguments in \cref{fig-E2E4}.
	
	\begin{prop}\label{prop-d7}
	We have the following $d_7$-differentials
		\begin{eqnarray*}
		d_7(2u_{\lambda}^2) &=& \Tr_2^{4, 3-2\lambda}(a_{\sigma_2}^7\ot_1^3),\\
		d_7(2u_{\lambda}^2u_{2\sigma}) &=& \Tr_2^{4, 5-2\lambda-2\sigma}(a_{\sigma_2}^7\ot_1^3),
		\end{eqnarray*}
		(see Notation~\ref{nota:transfer} for the transfer notations).
	\end{prop}
	\begin{proof}
	We will prove the first differential.  The second differential is proven by the exact same method.  On the $C_2$-level, we have the $d_7$-differential
	$$d_7(u_{2\sigma_2}^2) = (\ot_2 + \ot_1^3 + \gamma\ot_2)a_{\sigma_2}^7$$
	by \cref{thm-C_2}. Taking transfer on the target and using naturality, the class 
	$$\Tr_2^{4, 3-2\lambda}(a_{\sigma_2}^7(\ot_2 + \ot_1^3 + \gamma\ot_2)) = \Tr_2^{4, 3-2\lambda}(a_{\sigma_2}^7\ot_1^3)$$ 
	must be killed by a differential of length at most 7. For degree reasons, it must be the $d_7$-differential with source $2u_{\lambda}^2$.
	\end{proof}

    \begin{rmk}\rm
        These differentials can also be proved by combining \cref{prop:w-op} and  \cref{rmk-exotic}. One sets the exotic $w$-operation to be multiplication by $2$, and \cref{rmk-exotic} shows that the exotic restriction gives such an exotic multiplication.
    \end{rmk}
	
	In \cref{fig-E2E4}, The $d_7$-differentials in Proposition~\ref{prop-d7} and the underlying $C_2$-level $d_7$-differentials in \cref{thm-C_2} are supported by the classes at $(4+i,-4+i)$ for $i \geq 0$.
	
	\begin{prop}\label{prop-d7-2}
		\[
		d_7(u_{\lambda}^4) = u_{\lambda}^2\Tr_2^4(\ot_1^3a_{\sigma_2}^7).
		\]
	\end{prop}
	
	\begin{proof}
	We will prove in Proposition~\ref{prop-d13-2} that there is a nontrivial $d_{13}$-differential on the class $u_\lambda^4a_\sigma$ (we can already prove it at this point, but for organization reasons we prove it later).  This implies that the class $u_{\lambda}^4$ must support a differential of length at most 13. For degree reasons, the claimed $d_7$-differential is the only possibility. 
	\end{proof}

	In \cref{fig-E2E4}, the $d_7$-differential in Proposition~\ref{prop-d7-2} gives the $d_7$-differential supported by the class $u_{\lambda}^4a_{\lambda}^{-4}$ at $(8,-8)$.

	\subsection{The $C_4$-spectral sequence: higher differentials and extensions}
    \begin{figure}
        \includegraphics[scale = 0.5, page = 3]{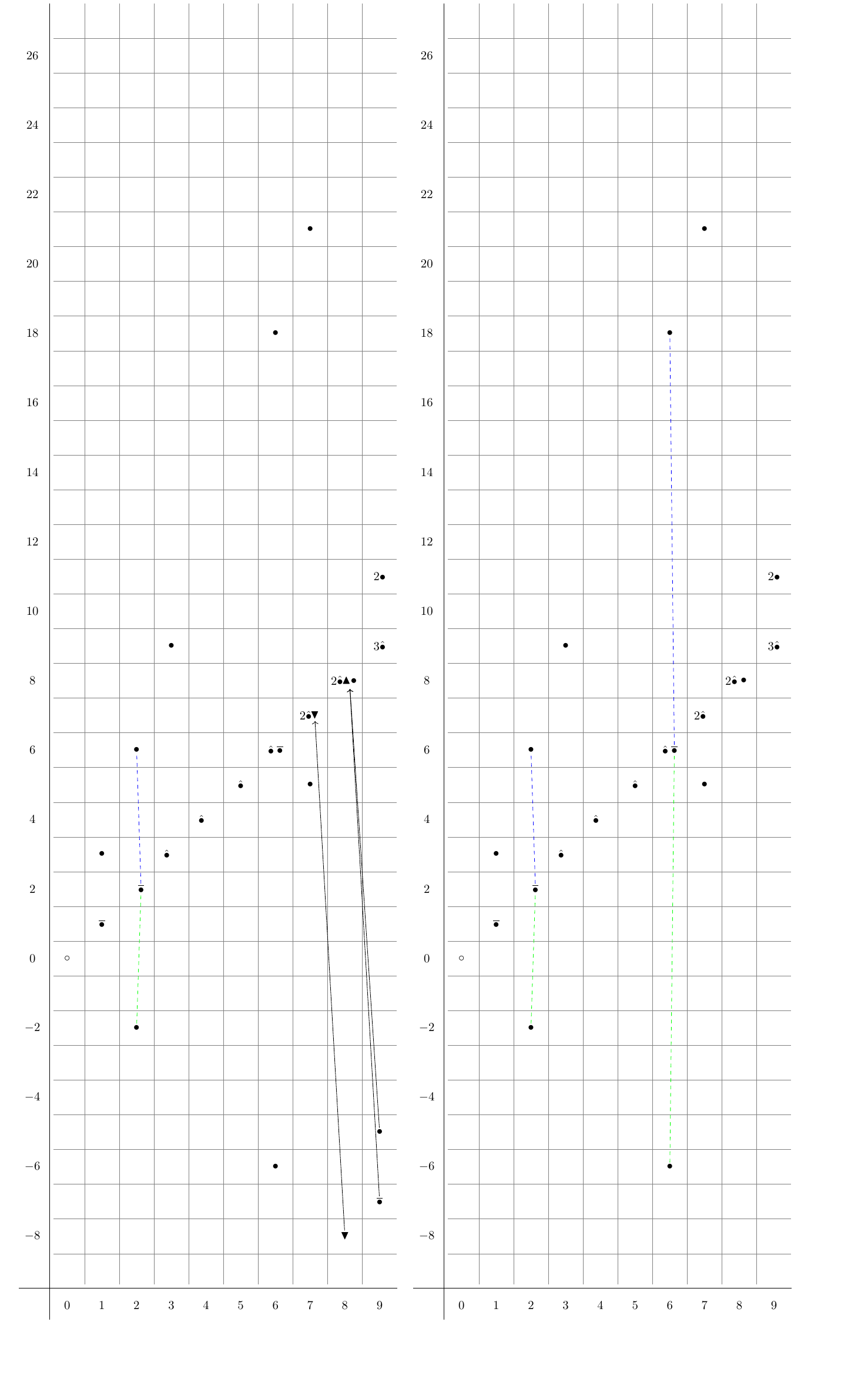}
        \caption{\label{fig-E11E16}Left: $d_{13}$- and $d_{15}$-differentials in $a_{\lambda}^{-1}\SliceSS(\BPfour)$.\\ Right: $E_{\infty}$-page of $a_{\lambda}^{-1}\SliceSS(\BPfour)$ with all extensions.}
	\end{figure}

    We will now prove the higher differentials in our range (see Figure~\ref{fig-E11E16}). The next possible differential is a $d_{13}$-differential from \cref{thm-HHRd}:
    \[
    d_{13}(u_{2\sigma}^2) = N(\ot_2)a_{\lambda}^{3}a_{\sigma}^{7}.
    \]
    However, we won't see this differential in Figure~\ref{fig-E11E16}.  This is because its first appearance in the integer graded spectral sequence is on the class $(10, 14)$, which is outside of our range. Note also that even though some classes at $(8,8)$ contain $u_{2\sigma}^2$, they don't support $d_{13}$-differentials. We will give a detailed discussion of the classes at $(8,8)$ in Section~\ref{sec-88}.

	\begin{prop}\label{prop-d13}
		\[
		d_{13}(u_{\lambda}^4u_{2\sigma}) = N(\ot_2)u_{\lambda}u_{2\sigma}^2a_{\lambda}^6a_{\sigma}
		\]
	\end{prop}
	
	\begin{proof}
		On the $C_2$-level, the restriction $\Res^4_2(u_{\lambda}^4u_{2\sigma}) = u_{2\sigma_2}^4$ supports a $d_{15}$-differential hitting the class $\ov_3a_{\sigma_2}^{15} =(\ot_3 + \ot_2^2\ot_1 + \ot_1^4 \gamma\ot_2 + \gamma\ot_3)a_{\sigma_2}^{15}$.  Since this class is not in the image of the restriction after the $d_3$-differentials, by naturality the class $u_{\lambda}^4u_{2\sigma}$ must support a differential of length shorter than 15. After computing the first few pages, we see that for degree reasons the potential targets are the following classes: 
		\begin{enumerate}
		\item $\Tr_2^4((\ot_2  + \ot_1^3 + \gamma\ot_2)u_{2\sigma_2}^2a_{\sigma_2}^7)$ in filtration $7$; 
		\item $N(\ot_1)^3u_{\lambda}u_{2\sigma}^2a_{\lambda}^6a_{\sigma}$ in filtration 13;
		\item $N(\ot_2)u_{\lambda}u_{2\sigma}^2a_{\lambda}^6a_{\sigma}$ in filtration $13$.
		\end{enumerate}
		
		We will first prove that the class $\Tr_2^4((\ot_2  + \ot_1^3 + \gamma\ot_2)u_{2\sigma_2}^2a_{\sigma_2}^7)$ supports the $d_{11}$-differential
		\[
		d_{11}(\Tr_2^4((\ot_2  + \ot_1^3 + \gamma\ot_2)u_{2\sigma_2}^2a_{\sigma_2}^7)) = N(\ot_1)^4u_{2\sigma}^2a_{\lambda}^8a_\sigma^2.
		\]
		To prove this, first note that 
		\[
		\Tr_2^4((\ot_2  + \ot_1^3 + \gamma\ot_2)u_{2\sigma_2}^2a_{\sigma_2}^7) = \Tr_2^4(\ot_1^3u_{2\sigma_2}^2a_{\sigma_2}^7) 
		\]
		since the class $(\ot_2 + \gamma \ot_2)a_{\sigma_2}$ transfers to $0$ in the homotopy.
		On the $C_2$-level, we have the $d_7$-differential
		\[
		d_7(\ot_1^3u_{2\sigma_2}^2a_{\sigma_2}^7) = \ot_1^3(\ot_2 + \ot_1^3 + \gamma \ot_2)a_{\sigma_2}^{14}.
		\]
		The transfer of the target, $\Tr_2^4(\ot_1^3(\ot_2 + \ot_1^3 + \gamma \ot_2)a_{\sigma_2}^{14}) = \Tr_2^4(\ot_1^6a_{\sigma_2}^{14})$, is zero.  This is because after the $C_2$-level $d_3$-differentials, the class $\ot_1^6 a_{\sigma_2}^{14}$ is identified with the class $\ot_1^3 \gamma \ot_1^3 a_{\sigma_2}^{14}$, which transfers to 0. We will show that the class $\ot_1^6a_{\sigma_2}^{14}$ actually supports an exotic transfer of filtration jump 4. Let $x = N(\ot_1)^3a_{\lambda}^7u_{2\sigma}^3$. We have the $d_5$-differential from \cref{thm-HHRd}
		\[
		d_5(x) = N(\ot_1)^4u_{2\sigma}^2a_{\lambda}^8a_{\sigma}^3.
		\]
		By \cref{prop-transfer}, $N(\ot_1)^4u_{2\sigma}^2a_{\lambda}^8a_{\sigma}^2$ receives an exotic transfer of jump $4$, and the only possible source is $\ot_1^6 a_{\sigma}^{14}$. Applying \cref{prop:w-op} to this exotic transfer and the $C_2$-level $d_7$, we prove the claimed $d_{11}$.
		
		The class $N(\ot_1)^3u_{\lambda}u_{2\sigma}^2a_{\lambda}^6a_{\sigma}$ in filtration $13$ is killed by a $d_5$-differential from \cref{prop-d5}:
		\[
		N(\ot_1)^3u_{\lambda}u_{2\sigma}^2a_{\lambda}^6a_{\sigma} = d_5(N(\ot_1)^2 u_{\lambda}^2u_{2\sigma}^2a_{\lambda}^4).
		\]
		It follows that the class $N(\ot_2)u_{\lambda}u_{2\sigma}^2a_{\lambda}^6a_{\sigma}$ is the only possible target.
	\end{proof}	

    \begin{rmk}\rm
    The class $u_{\lambda}^4u_{2\sigma}$ is a permanent cycle in the homotopy fixed points spectral sequence of $E_2^{hC_4}$ (see \cite[Proposition~5.23]{BBHS}) because $N(\ot_2)$ is zero there.
    \end{rmk}

	Although this $d_{13}$ doesn't imply any differentials in our range, it is used in proving extensions.
	\begin{prop}\label{prop-ext12}
	\begin{enumerate}
		\item There is an exotic transfer in stem $6$ of filtration $12$,
		\[
		\Tr_2^4(\ot_2 \gamma\ot_2a_{\sigma_2}^6) = N(\ot_2)^2a_{\lambda}^6a_{\sigma}^6.
		\]
		\item There is an exotic restriction in stem $6$ of filtration $12$,
		\[
		\Res_2^4(2u_{\lambda}^3a_{\lambda}^{-3}) = \ot_2 \gamma\ot_2a_{\sigma_2}^6.
		\]
		\end{enumerate}
	\end{prop}
	
	\begin{proof}
	    The proof is similar to that of \cref{prop-ext4}. The exotic transfer comes from applying \cref{prop-transfer} to the $d_{13}$-differential
		\[
		d_{13}\left(N(\ot_2)u_{2\sigma}^2a_{\lambda}^3\right) = N(\ot_2)^2a_{\lambda}^6a_{\sigma}^7
		\]
		in \cref{thm-HHRd}.
		
		For the exotic restriction, first note that $2u_{\lambda}^3a_{\lambda}^{-3} = \left(\frac{u_{\lambda}^4}{u_{2\sigma}}a_{\lambda}^{-4}a_{\sigma}\right)a_{\sigma}$ by the gold relation.  We will prove that the class $\frac{u_{\lambda}^4}{u_{2\sigma}}a_{\lambda}^{-4}a_{\sigma}$ supports a $d_{13}$-differential.  To do so, we multiply this class by $u_{2\sigma}^2a_{\lambda}^4$. After multiplying the differential in Proposition~\ref{prop-d13} by $a_\sigma$, we have
		\[
		d_{13}(u_{\lambda}^4 u_{2\sigma}a_{\sigma}) = 2N(\ot_2)u_{2\sigma}^3a_{\lambda}^7.
		\]
		As by the gold relation $u_{\lambda}^2$ kills $d_{13}(u_{2\sigma}^2a_{\lambda}^4)$, we can use the Leibniz rule to obtain the $d_{13}$-differential 
		\[d_{13}\left(\frac{u_{\lambda}^4}{u_{2\sigma}}a_{\lambda}^{-4}a_{\sigma}\right) = 2N(\ot_2)u_{2\sigma}a_\lambda^3.\]
		
		On the $E_2$-page, $2N(\ot_2)u_{2\sigma} a_{\lambda}^3 = Tr_2^4(\ot_2 \gamma\ot_2a_{\sigma_2}^6)$. By \cref{prop-transfer}, $\ot_2 \gamma\ot_2a_{\sigma_2}^6$ must receive an exotic restrion of filtration jump $12$, and the only possible source is $2u_{\lambda}^3a_{\lambda}^{-3}$ (see \cref{fig-E11E16}).
	\end{proof}
	
	In \cref{fig-E11E16}, they are the exotic restriction from the class $(6,-6)$ to $(6,6)$ and the exotic transfer from $(6,6)$ to $(6,18)$. Since these extensions involve elements containing $\ot_2$, we expect similar extensions in the homotopy fixed points spectral sequence of $E_4^{hC_4}$ by \cite[Theorem~1.1]{BHSZ}. 
	
    \begin{prop}\label{prop-d13-2}
	\[
	d_{13}(u_{\lambda}^4a_{\sigma}) = N(\ot_2 + \ot_1^2\gamma \ot_1 + \gamma\ot_2)u_{2\sigma}^2a_{\lambda}^7.
	\]
    \end{prop}

    \begin{proof}
    Consider the $C_2$-differential
    $$d_7(u_{2\sigma_2}^2) = \ot_2^{C_2}a_{\sigma_2}^7.$$
    Applying \cref{prop-norm} to its target, we see that its norm $N(\ot_2 + \ot_1^2 \gamma \ot_1 + \gamma(\ot_2))a_{\lambda}^7$ must be killed by a differential of length $13$ or shorter. Since the restriction of this element is killed by $d_7$, it must be killed by a differential of length between $7$ and $13$. Since $u_{2\sigma}^2$ supports a $d_{13}$, if $d_{r}(x) = N(\ot_2 + \ot_1^2 \gamma \ot_1 + \gamma(\ot_2))a_{\lambda}^7$ happens for $r < 13$, one can multiply both sides by $u_{2\sigma}^2$. However, for degree reasons $ N(\ot_2 + \ot_1^2\gamma \ot_1 + \gamma\ot_2)u_{2\sigma}^2a_{\lambda}^7$ cannot be hit by a differential shorter than a $d_{13}$. Thus this element and hence also $N(\ot_2 + \ot_1^2 \gamma \ot_1 + \gamma(\ot_2))a_{\lambda}^7$ must be hit by a $d_{13}$ and the only possible source is $u_{\lambda}^4a_{\sigma}$.
    \end{proof}

	On the integer graded page, this contributes to the $d_{13}$-differential supported by the class $N(\ot_1)u_\lambda^4a_\lambda^{-3}a_\sigma$ at $(9,-5)$.
	
	The last differential in our range is a $d_{15}$-differential. 
	
	\begin{prop}\label{prop-d15}
	We have the $d_{15}$-differential
		\[
		d_{15}(2u_{\lambda}^4) = \Tr_2^4(\ot_3^{C_2}a_{\sigma_2}^{15}).
		\]
	\end{prop}
	\begin{proof}
	In the the $C_2$-spectral sequence, we have the $d_{15}$-differential 
	\[d_{15}(u_{2\sigma_2}^4)= \ot_3^{C_2}a_{\sigma_2}^{15}.\]
	Applying the transfer shows that the class $\Tr_2^4(\ot_3^{C_2}a_{\sigma_2}^{15})$ must be killed by a differential of length at most 15.  By naturality and degree reasons, the only possible source is the class $2u_{\lambda}^4 = \Tr_2^4(u_{2\sigma_2}^4)$. 
	\end{proof}
	
	In \cref{fig-E11E16}, this contributes to the $d_{15}$-differential supported by the class $2u_{\lambda}^4a_{\lambda}^{-4}$ at $(8,-8)$ (the $d_{15}$-differential supported by the class at $(9,-7)$ is a $C_2$-level differential). 
	
	These are all the differentials and extensions in the first $8$ stems. Now we will discuss in detail the generators and relations in degree $(8,8)$ after each differential in order to illustrate the technical aspect of tracking differentials in the localized slice spectral sequences.
	
	\subsection{The classes at $(8,8)$}\label{sec-88} Since our discussion here focuses on a single degree, we will omit the powers of $a_V$ and $u_V$ classes on each monomial, except in formulas of differentials. That is, we omit $u_{2\sigma}^2a_{\lambda}^4$ on $C_4$-classes and $a_{\sigma_2}^8$ on $C_2$-classes.
	
	On the $E_3$-page, there are 2 $\circ$ and $16$ $\hat{\bullet}$.  The 2 $\circ$ are $N(\ot_1)^4$ and $N(\ot_2)N(\ot_1)$.  The 16 $\hat{\bullet}$ are 
	\begin{enumerate}
	    \item $\Tr_2^4(\ot_1^8)$, $\Tr_2^4(\ot_1^7 \gamma \ot_1)$, $\Tr_2^4(\ot_1^6 \gamma \ot_1^2)$, $\Tr_2^4(\ot_1^5 \gamma \ot_1^3)$;
	    \item $\Tr_2^4(\ot_2 \ot_1^5)$, $\Tr_2^4(\ot_2\ot_1^4 \gamma \ot_1)$, $\Tr_2^4(\ot_2\ot_1^3 \gamma \ot_1^2)$, $\Tr_2^4(\ot_2\ot_1^2 \gamma \ot_1^3)$, $\Tr_2^4(\ot_2\ot_1 \gamma \ot_1^4)$, $\Tr_2^4(\ot_2\gamma \ot_1^5)$;
	    \item $\Tr_2^4(\ot_2^2\ot_1^2)$, $\Tr_2^4(\ot_2^2\ot_1 \gamma \ot_1)$, $\Tr_2^4(\ot_2^2 \gamma \ot_1^2)$;
	    \item $\Tr_2^4(\ot_2 \gamma \ot_2 \ot_1^2)$; 
	    \item $\Tr_2^4(\ot_3 \ot_1)$, $\Tr_2^4(\ot_3 \gamma \ot_1)$.
	\end{enumerate}
	
	At the $C_2$-level, the $d_3$-differentials identifies $\ot_1$ with $\gamma \ot_1$.  At the $C_4$-level, the effect of the $d_3$-differentials are as follows: 
	\begin{enumerate}
	    \item All the classes in (1) are identified with $2 N(\ot_1)^4$;
	    \item all the classes in (2) are identified to be the same;
	    \item all the classes in (3) are identified to be the same; 
	    \item the class $\Tr_2^4(\ot_2 \gamma \ot_2 \ot_1^2)$ is identified with $2N(\ot_2)N(\ot_1)$;
	    \item all the classes in (5) are identified to be the same. 
	\end{enumerate}
	Therefore after the $d_3$-differentials, there are $2$ $\circ$, generated by $N(\ot_1)^4$ and $N(\ot_2)N(\ot_1)$, and $3$ $\hat{\bullet}$, generated by $\Tr_2^4(\ot_2 \ot_1^5)$, $\Tr_2^4(\ot_2^2\ot_1^2)$, and $\Tr_2^4(\ot_3 \ot_1)$. 
	
	On the $E_5$-page, by Proposition~\ref{prop-d5}, we have the following two $d_5$-differentials: 
	\begin{eqnarray*}
	d_5(N(\ot_1)^3u_\lambda^2 u_{2\sigma}a_\lambda a_\sigma) &=& 2 N(\ot_1)^4 u_{2\sigma}^2 a_\lambda^4, \\
	d_5(N(\ot_2)u_\lambda^2 u_{2\sigma}a_\lambda a_\sigma) &=& 2 N(\ot_2) N(\ot_1) u_{2\sigma}^2 a_\lambda^4. 
	\end{eqnarray*}
	It follows that after the $d_5$-differentials, the 2 $\circ$ become 2 $\JJJ$, with the same generators. In total, there are 2 $\JJJ$ and 3 $\hat{\bullet}$ at $(8,8)$ after the $d_5$-differentials (with the same generator as before). 
	
	Now we will discuss the $d_7$-differentials.  At $(9,1)$, there are two classes on the $E_7$-page: a $\hat{\bullet}$ generated by $\Tr_2^4(\ot_2 \ot_1^2)$ and a $\Ov{\bullet}$ generated by $\ot_1^5$ (it only exists on the $C_2$-level).  Since $\ov_2 = \ot_2 + \ot_1^3 + \gamma\ot_2$, the $d_7$-differential on the class $\Tr_2^4(\ot_2 \ot_1^2)$ hits the class 
	\[\Tr_2^4(\ot_2 \ot_1^2 (\ot_2 + \ot_1^3 + \gamma \ot_2)) = \Tr_2^4(\ot_2^2 \ot_1^2) + \Tr_2^4(\ot_2 \ot_1^5) + \Tr_2^4(\ot_2 \gamma \ot_2 \ot_1^2) = \Tr_2^4(\ot_2^2 \ot_1^2) + \Tr_2^4(\ot_2 \ot_1^5).\]
	In other words, it identifies the classes $\Tr_2^4(\ot_2^2 \ot_1^2)$ and $\Tr_2^4(\ot_2 \ot_1^5)$.
	
	The $d_7$-differential on the class $\ot_1^5$ hits the class 
	\begin{eqnarray*}
	\ot_1^5(\ot_2 + \gamma \ot_2 + \ot_1^3) &=& \ot_2 \ot_1^5 + \gamma \ot_2 \ot_1^5 + \ot_1^8 \\
	&=& \Res_2^4(\Tr_2^4(\ot_2 \ot_1^5)) + \Res_2^4(N(\ot_1)^4) \\
	&=& \Res_2^4(\Tr_2^4(\ot_2^2 \ot_1^2)) + \Res_2^4(N(\ot_1)^4).
	\end{eqnarray*}
	As Mackey functors, we have
	\[
	\hat{\bullet} \overline{\bullet} \xrightarrow{d_7} 2\JJJ 3\hat{\bullet} \twoheadrightarrow \bullet \JJJ 2\hat{\bullet}.
	\]
	In the quotient we need to choose our generators carefully: The $\bullet$ is generated by $N(\ot_1)^4 + \Tr_2^4(\ot_2^2\ot_1^2)$, because the image of $\Ov{\bullet}$ identifies the restriction of $N(\ot_1)^4$ with the restriction of $\Tr_2^4(\ot_2^2 \ot_1^2)$. Therefore their sum is the unique element in $C_4$-level that has trivial restriction. The $\JJJ$ is generated by $N(\ot_2)N(\ot_1)$, as it still has nontrivial restriction. The two $\hat{\bullet}$ are generated by $\Tr_2^4(\ot_2^2 \ot_1^2)$ and $\Tr_2^4(\ot_3 \ot_1)$. 
	
	The next differential is a $d_{13}$-differential supported by the class $N(\ot_1)u_{\lambda}^4a_{\lambda}^{-3}a_\sigma$ at $(9,-5)$.  By Proposition~\ref{prop-d13-2}, the target of this differential is the class $N(\ot_1)N(\ot_2 + \ot_1^3 + \gamma \ot_2)u_{2\sigma}^2a_{\lambda}^4$.  The restriction of this class is 
	\[\ot_1 \gamma \ot_1(\ot_2 + \ot_1^3 + \gamma \ot_2)(\gamma \ot_2 + \gamma \ot_1^3 - \ot_2),\]
	which, after the $d_3$-differentials, is 
	\[\ot_2^2 \ot_1^2 + \gamma \ot_2^2 \ot_1^2 + \ot_1^8 = \Res_2^4(\Tr_2^4(\ot_2^2 \ot_1^2)) + \Res_2^4(N(\ot_1)^4).\]
	As we have discussed above, this class is killed by the $d_7$-differentials supported by the class $\ot_1^5$. It follows that the target of the $d_{13}$-differential is the generator of $\bullet$, the unique nontrivial element that restricts to 0. 
	
	There is another possible $d_{13}$-differential supported by some classes at $(8,8)$ that is induced by the differential
	\[
	d_{13}(u_{2\sigma}^2) = N(\ot_2)a_{\lambda}^3a_{\sigma}^7.
	\]
	However, in $(8,8)$ every monomial containing $u_{2\sigma}^2$ also contains $N(\ot_1)$. By \cite[Corollary~9.13]{HHR},
	\[
	d_{13}(N(\ot_1)u_{2\sigma}^2) = N(\ot_1)N(\ot_2)a_{\lambda}^3a_{\sigma}^7 = d_{5}(N(\ot_2)u_{2\sigma}a_{\lambda}^2a_{\sigma}^4).
	\]
	This makes all elements containing $u_{2\sigma}^2$ in $(8,8)$ $d_{13}$-cycles.
	
	In summary, after the $d_{13}$-differentials, we have two $\hat{\bullet}$, generated by $\Tr_2^4(\ot_2^2 \ot_1^2)$ and $\Tr_2^4(\ot_3 \ot_1)$, and $\JJJ$, generated by $N(\ot_2)N(\ot_1)$. 
	
	Our final differential is a $d_{15}$-differential on the $C_2$-level supported by the class at $(9,-7)$: 
	\begin{eqnarray*}
	d_{15}(\ot_1u_{2\sigma_2}^4a_{\sigma_2}^{-7}) &=& \ot_1(\ot_3 + \ot_2^2 \ot_1 + \gamma\ot_2\ot_1^4 + \gamma\ot_3)a_{\sigma_2}^8 \\
	&=& (\ot_3 \ot_1 + \gamma \ot_3 \ot_1)a_{\sigma_2}^8 + (\ot_2^2 \ot_1^2 + \gamma \ot_2 \ot_1^5) a_{\sigma_2}^8 \\
	&=& \Tr_2^4(\ot_3 \ot_1) + (\ot_2^2 \ot_1^2 \gamma \ot_2^2 \ot_1^2 + \ot_2 \gamma \ot_2 \ot_1^2) a_{\sigma_2}^8 \\ 
	&=& \Tr_2^4(\ot_3 \ot_1) + \Tr_2^4(\ot_2^2 \ot_1^2) + \Tr_2^4 \Res_2^4 (N(\ot_2)N(\ot_1)).
	\end{eqnarray*}
	
	The map in Mackey functors is 
	\[
	\overline{\bullet} \xrightarrow{d_{15}} \JJJ 2 \hat{\bullet }\twoheadrightarrow \bullet 2 \hat{\bullet}.
	\]
	On the $E_\infty$-page, $(8,8)$ is given by $\bullet 2 \hat{\bullet}$. The generators for the two $\hat{\bullet}$ are $\Tr_2^4(\ot_2^2 \ot_1^2)$ and $\Tr_2^4(\ot_3 \ot_1)$.  The generator for $\bullet$ is $\Tr_2^4(\ot_2^2 \ot_1^2)+\Tr_2^4(\ot_3 \ot_1) + N(\ot_2)N(\ot_1)$. 
	
	\subsection{A family of permanent cycles}
	We will now present families of nontrivial permanent cycles in $a_{\lambda}^{-1}\SliceSS(\BPfour)$.  These families will be used in the proof of \cref{thm:TateDiff}.
	
		\begin{lem}\label{lem:sigmalambdacovering}
		In $\pi_\bigstar^{C_4}a_{\sigma}^{-1}\mathbb{S}$, the element $a_{\lambda}$ is invertible.
	\end{lem}
	
	\begin{proof}
		We have the following commutative diagram of pointed $C_4$-spaces
		\[
		\xymatrix{
			S^0 \ar[r]^{a_{\lambda}} \ar[rd]_{a_{\sigma}^2} & S^{\lambda} \ar[d]^{\theta} \\
			& S^{2\sigma}
		}
		\]
		where $\theta$ is the $C_4$-equivariant $2$-folded branched cover.  Since $\theta a_\lambda = a_\sigma^2$ is invertible, $a_\lambda$ is invertible. 
	\end{proof}
	
	\begin{prop}\label{prop-cycle}
		In $\pi_{\star}^{C_4}a_{\lambda}^{-1}\BPfour$, the classes $N(\ot_k)a_{\sigma}^i$ for $k > 0$ and $0 \leq i < 2^{k+1}-1$ are non-zero.
	\end{prop}
	
	\begin{proof}		
	By \cref{lem:sigmalambdacovering} we have a map of spectral sequences
		\[
		a_{\lambda}^{-1}\SliceSS(\BPfour) \longrightarrow a_{\sigma}^{-1}\SliceSS(\BPfour).
		\]
		Notice that in $a_{\sigma}^{-1}\SliceSS(\BPfour)$, the differentials in \cref{thm-HHRd} completely determine the spectral sequence (See \cite[Remark~9.11]{HHR}). In particular, we have the following differentials in $a_{\sigma}^{-1}\SliceSS(\BPfour)$:
		\[
		d_{2^{k+2}-3}(u_{2\sigma}^{2^{k-1}}a_{\lambda}^{-(2^k-1)}a_{\sigma}^{-(2^{k+1}-1)+i}) = N(\ot_k)a_{\sigma}^{i}.
		\]
		On $E_{2^{k+2}-3}$-page, this is the only differential happens in this degree.
		
		By \cref{prop-C_4coe} and the gold relation, the class $u_{2\sigma}^{2^{k-1}}a_{\sigma}^{-(2^{k+1}-1)+i}$ is in the image of
		\[
		\pi_{\star}^{C_4}a_{\lambda}^{-1}H\UZ \rightarrow \pi_{\star}^{C_4}a_{\sigma}^{-1}H\UZ
		\]
		only when $a_{\sigma}$ has a non-negative power, i.e. $i \geq 2^{k+1}-1$. Therefore by naturality, if the class $N(\ot_k)a_{\sigma}^i$, $0 \leq i < 2^{k+1}-1$ is killed in $a_{\lambda}^{-1}\SliceSS(\BPfour)$, the differential killing it must be of length longer than $2^{k+2}-3$. However, by \cref{prop-C_4coe} and \cref{thm-main}, the potential source of such a differential must be trivial in the $E_2$-page. Therefore the classes  $N(\ot_k)a_{\sigma}^i$ for $k>0$ and $0 \leq i \leq 2^{k+1}-1$ are nontrivial permanent cycles.
	\end{proof}
	
	\begin{rmk}\rm
	After inverting $a_{\lambda}$, the element $N(\ot_k)a_{\sigma}^{2^{k+1}-1}$ is zero by \cref{thm-HHRd}.
	\end{rmk}

\section{The Tate spectral sequence of $N_1^2H\FF_2$}\label{sec-tate}
The goal of this section is to advance our knowledge of the Tate spectral sequence of $N_1^2H\FF_2$. We compute it in a range and also give all differentials originating from the first diagonal of slope $-1$. Our main method is comparison with the localized slice spectral sequence, a method we describe first.
	
There is a canonical map $\SliceSS(X) \rightarrow \HFPSS(X)$ which is an isomorphism on the underlying level \cite{Ullman:Thesis}. When $X = N_1^2 H\FF_2 \simeq (\BPfour)^{\Phi C_2}$, combining with \cref{thm:comparison}, we obtain the following comparison map of spectral sequences
\[
a_{\lambda}^{-1}\SliceSS(\BPfour) \rightarrow P^{*}_{4/2}\D \SliceSS(N_1^2H\FF_2) \rightarrow P^{*}_{4/2}\D\HFPSS(N_1^2 H\FF_2),
\]
where we use $P^*_{4/2}$ as a short-hand for the pullback functor $P^*_{C_4/C_2}$ from \cref{sec:GeometricFixedPoints}. Both maps of spectral sequences are compatible with the norm $N_2^4 = N_{C_2}^{C_4}$ by \cref{exam:slicehfpss} and \cref{prop:compare_norm}. 
On the $C_2$-level the composition sends permanent cycles from $\SliceSS(\BPfour)$ to their $C_2$-geometric fixed points. 

We localize the map 
\[
a_{\lambda}^{-1}\SliceSS(\BPfour) \rightarrow  P^{*}_{4/2}\D \HFPSS(N_1^2 H\FF_2)
\]
further at $a_{\sigma}$. Since $\sigma$ is the pullback of the sign representation on $C_4/C_2$, 
\[
a_{\sigma}^{-1}P^{*}_{4/2}\D \HFPSS(N_1^2 H\FF_2) \simeq  P^{*}_{4/2}a_{\sigma}^{-1}\D \HFPSS(N_1^2 H\FF_2).
\]
Notice that localizing at $a_{\sigma}$ in $C_4/C_2$-spectra is exactly smashing with $\tilde{E}\mathcal{F}[C_4/C_2]$, which turns the homotopy fixed points into the Tate fixed points. Therefore
\[
P^{*}_{4/2}a_{\sigma}^{-1}\D \HFPSS(N_1^2 H\FF_2) \simeq P^{*}_{4/2}\D \TateSS(N_1^2 H\FF_2).
\]

The above argument, along with \cref{lem:sigmalambdacovering}, gives the following comparison square, which is central to our computation in this section.
\begin{equation}\label{eq:comparison}
    \begin{tikzcd}
a_{\lambda}^{-1}\SliceSS(\BPfour) \arrow[rr] \arrow[dd, "a_{\sigma}^{-1}(-)"] &  & P^*_{4/2}\D \HFPSS(N_1^2 H\FF_2) \arrow[dd, "a_{\sigma}^{-1}(-)"] \\
                                                                             &  &                                                                \\
a_{\sigma}^{-1}\SliceSS(\BPfour) \arrow[rr]                                   &  & P^*_{4/2}\D \TateSS(N_1^2 H\FF_2)                                
\end{tikzcd}
\end{equation}

\begin{prop}
    In the comparison square, both horizontal maps converge to isomorphisms in homotopy groups.
\end{prop}

\begin{proof}
    The top horizontal map is the composition
    \[
  a_{\lambda}^{-1}\SliceSS(\BPfour) \rightarrow P^{*}_{4/2}\D \SliceSS(N_1^2H\FF_2) \rightarrow P^{*}_{4/2}\D \HFPSS(N_1^2 H\FF_2).
    \]
    The first map converges to an isomorphism by \cref{thm:comparison}. By the Segal conjecture, $N_1^2 H\FF_2$ is cofree, that is, the map $N_1^2 H\FF_2 \rightarrow F(EC_{2+},N_1^2 H\FF_2)$ is an equivalence. By construction, the second map of the spectral sequences converges to this map on homotopy.
        
    The bottom map is the $a_{\sigma}$-localization of the top map, thus also converges to an isomorphism.
\end{proof}

The bottom map in the comparison square is particularly interesting: We completely understand $a_{\sigma}^{-1}\SliceSS(\BPfour)$, which is determined by the fact that it computes $\pi_* H\FF_2$. All differentials are derived from the slice differential theorem \cite[Theorem~9.9, Remark~9.11]{HHR}. On the other hand, the Tate spectral sequence of $N_1^2 H\FF_2$ is very mysterious: its $E_2$-page is determined by the Tate cohomology $\hat{H}^{*}(C_2;\mathcal{A}_*)$, for which we do not know a closed formula yet. Nevertheless, the Segal conjecture shows that the Tate spectral sequence converges to $\pi_*H\FF_2$, meaning almost everything kills each other by differentials. Using \cref{thm:comparison}, we can apply our understanding of the slice spectral sequence to understand partially how differentials work in the Tate spectral sequence.

\cref{fig:comparison} consists of the integral $E_2$-pages of the four spectral sequences in the comparison square. Red elements in the homotopy fixed points and the Tate spectral sequences are those in the image of the horizontal maps. We prove these claims in \cref{cor:image}.
\afterpage{
\begin{figure}[p]
    \centering
    \includegraphics[scale = 0.37]{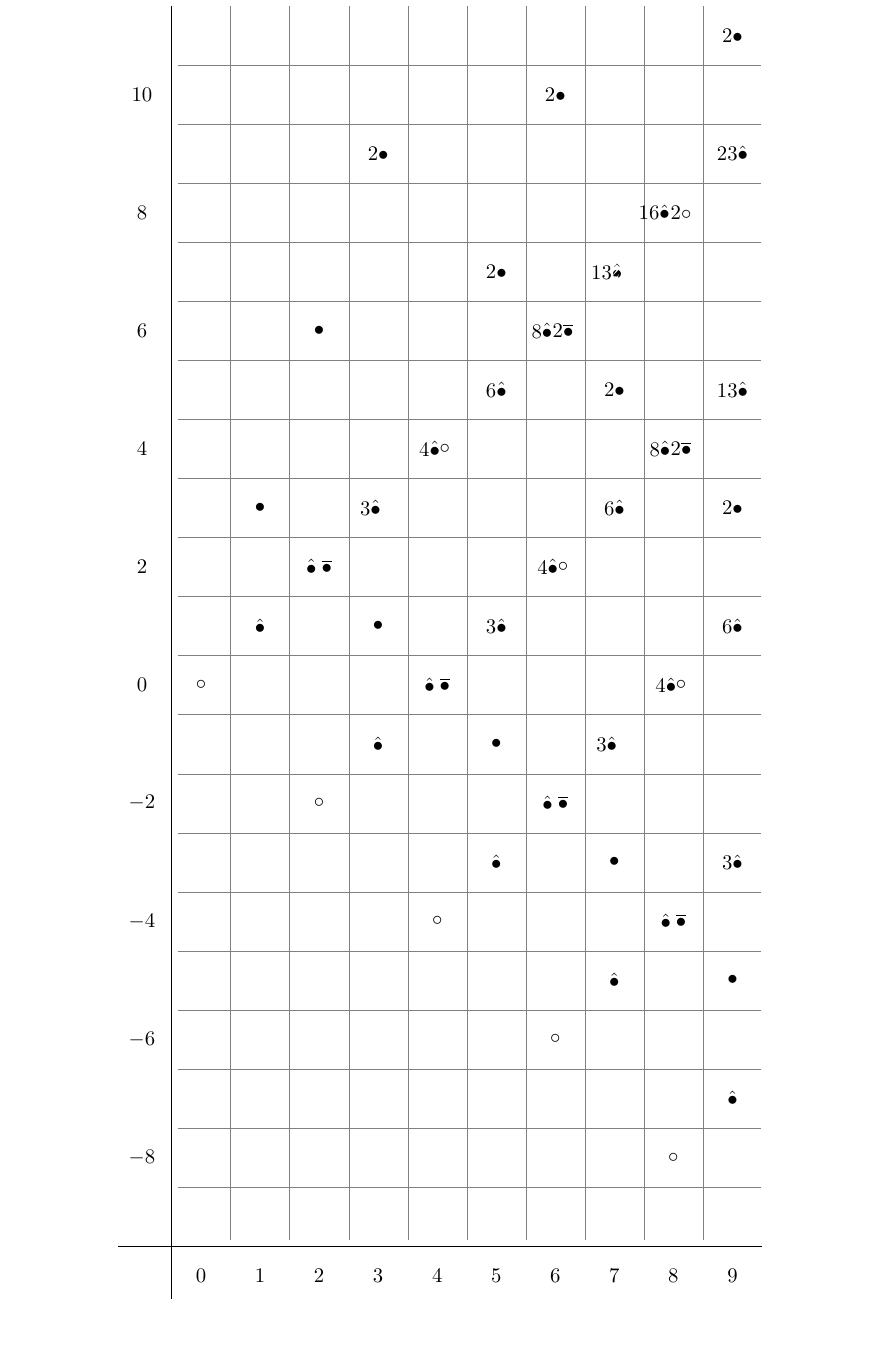}
    \includegraphics[scale = 0.37]{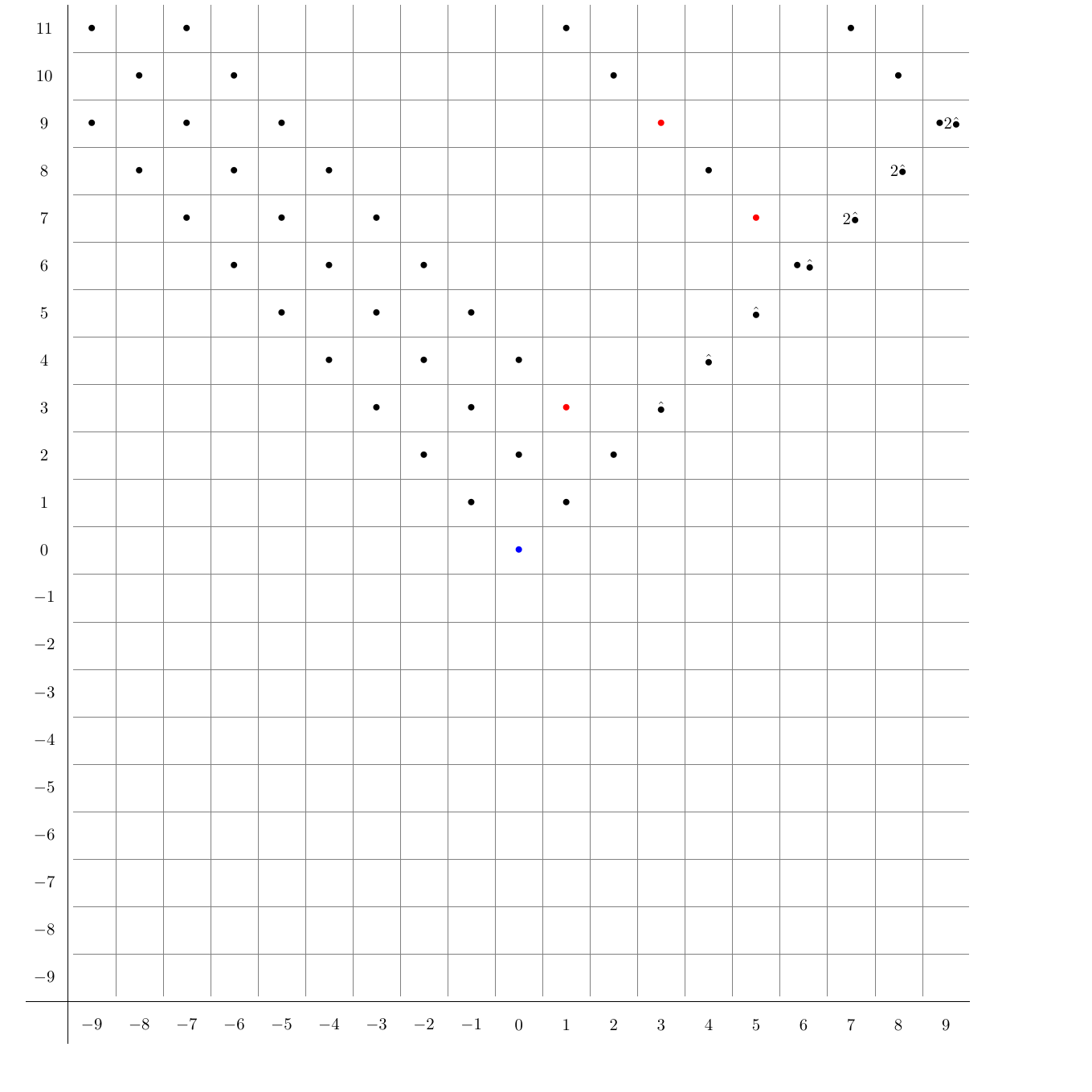}\\
    \includegraphics[scale = 0.37]{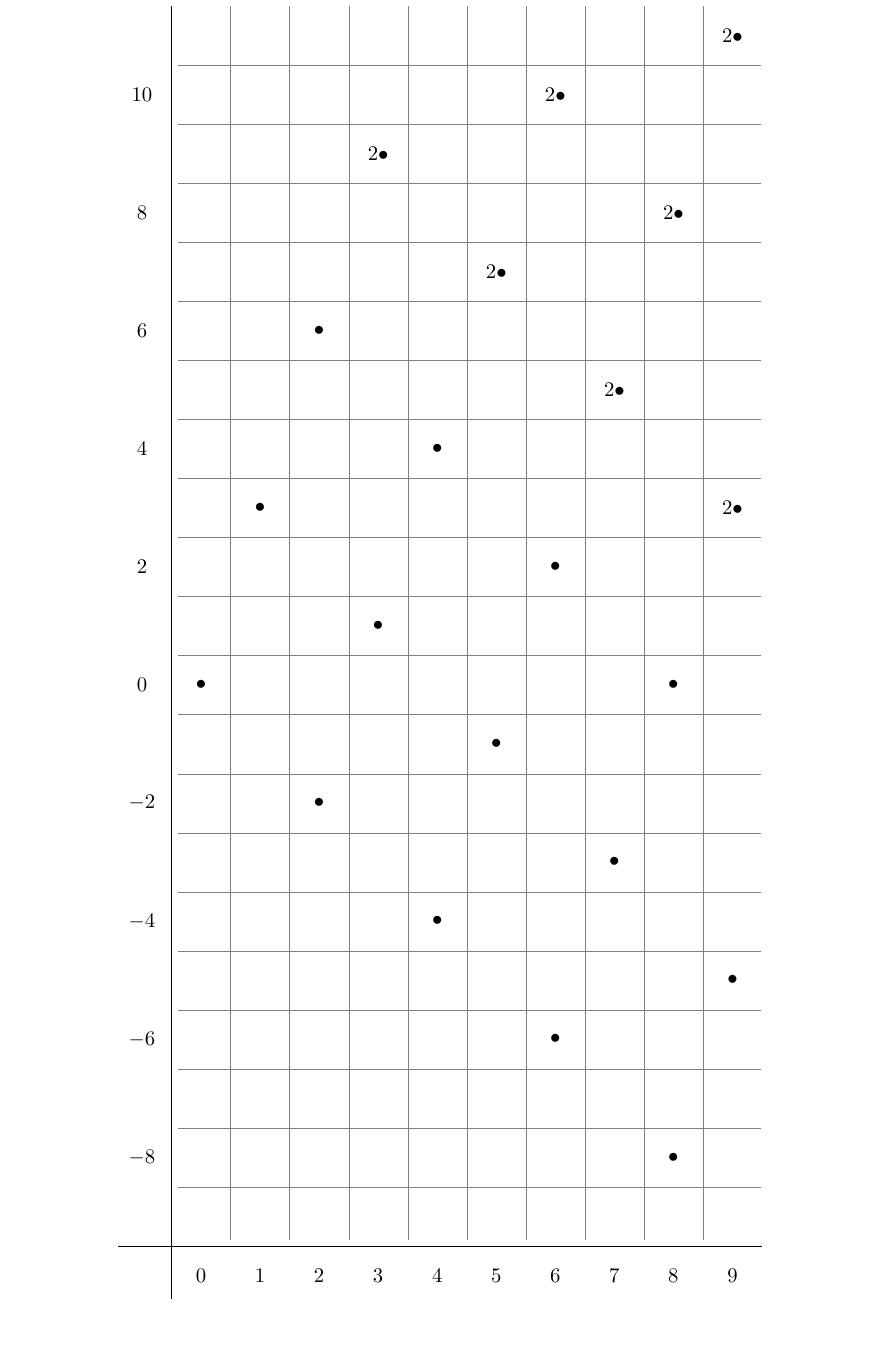}
    \includegraphics[scale = 0.37]{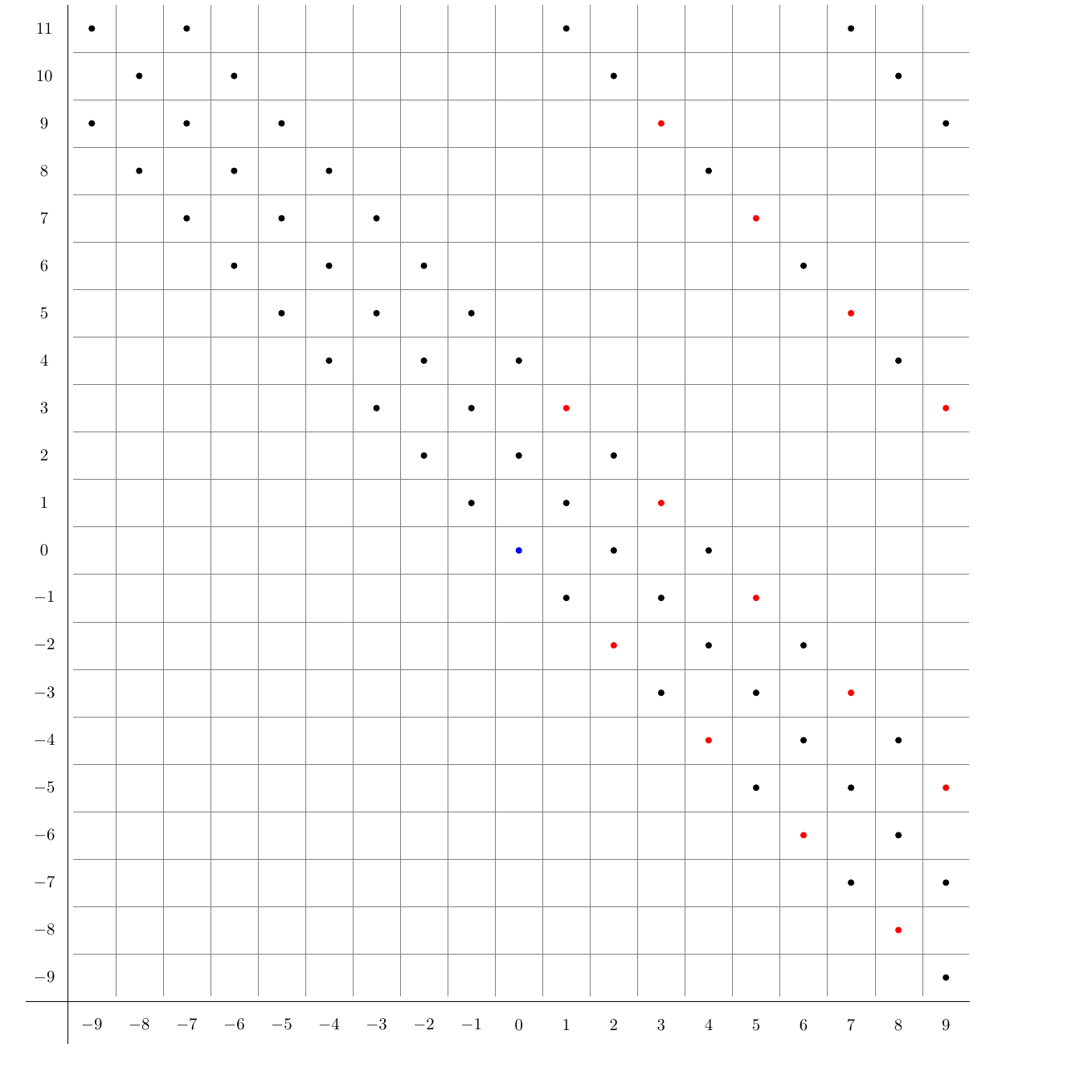}
    \caption{The comparison square (\ref{eq:comparison}).}
    \label{fig:comparison}
\end{figure}}

Using the comparison square, we establish an infinite family of differential in the Tate spectral sequence. We also compute all differentials in the Tate spectral sequence in the same range we computed $a_{\lambda}^{-1}\SliceSS(\BPfour)$ in \cref{section:sliceSScomputation}. Specifically, we show all differentials hitting elements from stem $0$ to $8$ which map non-trivially into the homotopy orbit spectral sequence. In the (doubled) Tate spectral sequence of \cref{fig:comparison}, they are elements below slope $1$ from stem $0$ to $8$.
 
Because of the comparison square, we make our statements and arguments entirely in the spectral sequences $P^*_{4/2}\D \HFPSS$ and $P^*_{4/2}\D \TateSS$. The translation back to the $C_2$-homotopy fixed points and the Tate spectral sequence is straightforward. As a reference, $\TateSS(N_1^2 H\FF_2)$ with known differentials is shown as \cref{fig:Tate}.

To start the computation, we want to understand how the maps of the comparison square behave on the $E_2$-page. By \cref{thm:comparison}, they are determined by the $C_2$-geometric fixed points of elements in $\pi_{\star}^{C_2}\BPfour$.

	\begin{prop}\label{prop:ti}
        Under the equivalence $\Phi^{C_2} \BPfour \simeq H\FF_2 \wedge H\FF_2$, the $C_2$-geometric fixed points of $\ot_i$ and $\gamma \ot_i$ in $\pi_{(2^i - 1)\rho_2}^{C_2} \BPfour$ are $\xi_i$ and $\zeta_i$, the Milnor generators and their conjugates in the dual Steenrod algebra, respectively.
	\end{prop}
	\begin{proof}
	    We will show $\Phi^{C_2}(\ot_i) = \xi_i$; the formula $\Phi^{C_2}(\gamma \ot_i) = \zeta_i$ follows from the fact that the residue $C_4/C_2$-action on $\Phi^{C_2} BP^{((C_4))}$ becomes the conjugate action on the dual Steenrod algebra.
	    
	    Let $e: S^0 \rightarrow \BPR$ be the unit map and $F_1$ and $F_2$ be the formal groups laws on $\pi_{*\rho_2}^{C_2}BP_{\RR} \wedge BP_{\RR}$ induced by the map 
	    \[
	    BP_{\RR} \wedge S^0 \xrightarrow{id \wedge e} BP_{\RR} \wedge BP_{\RR}
	    \]
	    and
	    \[
	    S^0 \wedge BP_{\RR} \xrightarrow{e \wedge id} BP_{\RR} \wedge BP_{\RR}
	    \]
	    respectively, and let $\tilde{x}_1$, $\tilde{x}_2$ be the corresponding power series generators. As in \cref{sec:sliceSSBackground}, the elements $\ot_i$ are defined as
	    \[
	    \tilde{x}_2 = {\sum_{i = 0}^{\infty}} {}^{{}^{\scriptstyle F_2}}\, \ot_i \tilde{x}_1^{2^i}.
	    \]
	    Taking $\Phi^{C_2}$ maps $\tilde{x}_1$ and $\tilde{x}_2$ to the two $MO$-orientation on $H\FF_2 \wedge H\FF_2$. The following lemma completes the proof.
	\end{proof}
	
	\begin{lem}
	    Let $x_1, x_2 \in (H\FF_2 \wedge H\FF_2)^1(\RR P^{\infty})$ be the $MO$-orientations corresponding to the maps
	    \[
	    H\FF_2 \wedge S^0 \xrightarrow{id \wedge \mu} H\FF_2 \wedge H\FF_2
	    \]
	    and
	    \[
	    S^0 \wedge H\FF_2 \xrightarrow{\mu \wedge id} H\FF_2 \wedge H\FF_2
	    \]
	    respectively. Then we have
	    \[
	    x_2 = \sum_{i = 0}^{\infty}\xi_i x_1^{2^i}.
	    \]
	\end{lem}
	\begin{proof}
	    Identify $(H\FF_2 \wedge H\FF_2)^*(\RR P^{\infty})$ with $\mathcal{A}_*[\![ x_1 ]\!]$, and write $x_2 = \sum\limits_{j = 0}^{\infty}a_j x_1^{j+1}$. We will show that $a_{2^i - 1} = \xi_i$ and all other $a_j$'s are $0$. First, since the power series $\sum\limits_{j = 0}^{\infty}a_j x_1^{j+1}$ is an automorphism of the additive formal group law in an $\FF_2$-algebra, we must have $a_0 = 1$, and $a_j = 0$ for $j \neq 2^i - 1$.
	    
	    Let $I$ be an admissible sequence and define
	    \[
	    \theta_I: (H\FF_2 \wedge H\FF_2)^*(\RR P^{\infty}) \rightarrow (H\FF_2 \wedge H\FF_2)^{* + |I|}(\RR P^{\infty})
	    \]
	    to be the composition
	    \[
	    \RR P^{\infty} \rightarrow H\FF_2 \wedge H\FF_2 \xrightarrow{id \wedge Sq^I} \Sigma^{|I|}H\FF_2 \wedge H\FF_2.
	    \]
	    One can verify directly that $\theta_I$ has the following properties:
	    \begin{itemize}
	        \item $\theta_I(x_2) = x_2^{2^n}$ if $I = (2^{n-1},2^{n-2},...,2,1)$, and $\theta_I(x_2) = 0$ otherwise.
	        \item $\theta_I(x_1) = x_1$ if $I = (0)$ and $\theta_I(x_1) = 0$ otherwise.
	        \item On homotopy, the map $H\FF_2 \wedge H\FF_2 \xrightarrow{id \wedge Sq^I} \Sigma^{|I|}H\FF_2 \wedge H\FF_2$ induces the map 
	        \[\theta_I: \mathcal{A}_* \longrightarrow \mathcal{A}_{*-|I|}.\] 
	        For any $\xi \in \mathcal{A}_*$, $\theta_I(\xi)$ is the cap product  \[\mathcal{A}_* \xrightarrow{\Delta} \mathcal{A}_* \otimes \mathcal{A}_* \xrightarrow{id \otimes \langle -, Sq^I\rangle} \mathcal{A}_*\]
	        between $\xi$ and $Sq^I$.  In the case when $|\xi| = |I|$, $\theta_I(\xi) = \langle \xi, Sq^I \rangle$, the pairing between the Steenrod algebra and its dual.
	        \item When $I = (i)$, $\theta_I$ satisfies the Cartan formula:
	        \[
	        \theta_i(ab) = \sum_{j = 0}^i \theta_{i-j}(a)\theta_j(b).
	        \]
	    \end{itemize}
	    Now, let $I = (2^{n-1},2^{n-2},...,2,1)$, and apply $\theta_I$ to $x_2 = \sum\limits_{i = 0}^{\infty}a_{2^i - 1} x_1^{2^i}$. The left hand side becomes
	    \[
	    x_2^{2^n} = \left(\sum\limits_{i = 0}^{\infty}a_{2^i - 1} x_1^{2^i} \right)^{2^n} = \sum\limits_{i = 0}^{\infty}a_{2^i - 1}^{2^n} x_1^{2^{n + i}}, 
	    \]
	    and the right hand side becomes 
	    \[
	    \theta_I\left(\sum\limits_{i = 0}^{\infty}a_{2^i - 1} x_1^{2^i}\right) = \sum\limits_{j = 0}^{\infty}\theta_I(a_{2^i - 1}) x_1^{2^i}.
	    \]
	    Comparing the coefficient of $x_1^{2^n}$ in both expressions, we see that $1 = \theta_I(a_{2^n-1}) = \langle a_{2^n - 1},Sq^I \rangle$.  Now, if $I$ is any other admissible sequence with $|I| = 2^{n} - 1$, then $\theta_I(x_2) = 0$ and thus $0 = \theta_I(a_{2^n-1})=\langle a_{2^n - 1}, Sq^I \rangle$. This is exactly the definition of $\xi_n$, see \cite[Chapter~6, Proposition~1]{MosherTangora}.
	\end{proof}

We pause here to clarify notations in $P^*_{4/2}\D \HFPSS(N_1^2 H\FF_2)$ and $P_{4/2}^*\D \TateSS(N_1^2 H\FF_2)$. The $C_2$-level of this spectral sequence is the spectral sequence of the doubled Postnikov tower of $H\FF_2 \wedge H\FF_2$, treated as a $C_2$-equivariant spectral sequence whose underlying level is trivial (and thus $a_{\sigma_2}$ acts invertibly). Therefore, given an element $x \in \mathcal{A}_*$, there are elements in different $RO(G)$-grading differing by powers of $a_{\sigma_2}$ that deserve the name $x$. We name the corresponding element in the integral grading by $x$, and name all others by $a_{\sigma_2}^{i}x$ for some $i \in \ZZ$. Notice that in this way, $\xi_n$ has stem and filtration $2^n - 1$ since we are working in the doubled spectral sequence. Under this notation, the map
\[
a_{\lambda}^{-1}\SliceSS(\BPfour) \rightarrow P^*_{4/2}\D \HFPSS(N_1^2 H\FF_2)
\]
sends $\ot_i$ to $a_{\sigma_2}^{-(2^i-1)}\xi_i$ (on the $C_2$-level), as follows from \cref{thm:comparison} and \cref{prop:ti}: since the target spectral sequence collapses on $C_2$-level, the image of $\ot_i$ is determined by its $RO(C_2)$-degree and its image under $\Phi^{C_2}$.

In the $C_4$-level, we need to be extra careful. By taking $N=N_2^4$ on $\ot_i \mapsto a_{\sigma}^{-(2^i-1)}\xi_i$, we see that 
\[
N(\ot_i) \mapsto a_{\lambda}^{-(2^i-1)}N(\xi_i),
\]
where $N(\xi_i)$ is in $RO(C_4)$-degree $(2^i-1)(1+\sigma)$ and filtration $2(2^i - 1)$. The complication comes from the fact that there are other generators of Tate cohomology than $N(\xi_i)$. For example, the element $\xi_1$ is invariant under the conjugate action and thus gives a generator of $\hat{H}^0(C_2;(\mathcal{A}_*)_1)$. For such generators in degree $i$ of $\mathcal{A}_*$, we will use the notation $b_i$, and define that they are in the integral grading. For example, the generator of $\hat{H}^0(C_2;(\mathcal{A}_*)_1)$ is named $b_1$, and has bidegree $(1,1)$ in the double of the homotopy fixed points and the Tate spectral sequence. Since the square of $b_1$ restricts to $\xi_1^2 = \xi_1\zeta_1$, we have (for degree reasons) a multiplicative relation
\[
b_1^2 = N(\xi_1)u_{\sigma},
\]
where $u_{\sigma}$ is a generator of the Tate cohomology of trivial module
\[
\hat{H}^{\star}(C_2;\FF_2) \cong \FF_2[a_{\sigma}^{\pm},u_{\sigma}^{\pm}].
\]
The generator $u_{\sigma}$ has degree $1-\sigma$ and $a_{\sigma}$ has degree $-\sigma$. The classical integral graded Tate cohomology 
\[
\hat{H}^{*}(C_2;\FF_2) \cong \FF_2 [x^{\pm}]
\]
with degree $1$ generator $x$ is related to the $RO(C_2)$-graded cohomology via $x = u_{\sigma}a_{\sigma}^{-1}$.	Since the sign representation on $C_4/C_2$ pulls back to the sign representation on $C_4$, we use the same notations $u_{\sigma}$ and $a_{\sigma}$ in the pullback of the homotopy fixed points and the Tate spectral sequence. 

In summary, the $C_2$-level of $P^*_{4/2}\D \HFPSS(N_1^2 H\FF_2)$ has the form
\[
\FF_2[a_{\sigma_2}^{\pm}][\xi_1,\xi_2,\cdots]
\]
where $\xi_i$ has both degree and filtration $2^i - 1$.
The $C_4$-level of $P^*_{4/2}\D \HFPSS(N_1^2 H\FF_2)$ has the form
\[
H^0(C_2;\mathcal{A}_*)[a_{\lambda}^{\pm},a_{\sigma},u_{\sigma}^{\pm}]/(Tr(x)a_{\sigma}, \forall x \in \mathcal{A}_*),
\]
where an element in $H^0(C_2;(\mathcal{A}_*)_i)$ has both degree and filtration $i$ if it doesn't restricts to elements of the form $\xi_k\zeta_k$; if this happens, the integral graded element is named by $N(\xi_k)u_{\sigma}^{2^i - 1}$. The class $u_{\sigma}$ is of stem $1-\sigma$ and of filtration $0$, while $a_{\sigma}$ is of stem $-\sigma$ and filtration $1$. 

The $C_4$-level of $P^*_{4/2}\D \TateSS(N_1^2 H\FF_2)$ has the form 
\[
\hat{H}^0(C_2;\mathcal{A}_*)[a_{\lambda}^{\pm},a_{\sigma}^{\pm},u_{\sigma}^{\pm}],
\]
with names of elements in $\hat{H}^0$ from the image of the surjective map $H^0(C_2;\mathcal{A}_*) \rightarrow \hat{H}^0(C_2;\mathcal{A}_*)$.

\begin{prop}\label{prop:power}
    In the Tate cohomology $\hat{H}^0(C_2;\mathcal{A}_*)$ the following elements are nontrivial:
\[
\xi_1\zeta_1, \xi_2\zeta_2,(\xi_2\zeta_2)^2,\xi_3\zeta_3,(\xi_3\zeta_3)^2, (\xi_3\zeta_3)^3, (\xi_i\zeta_i)^k
\]
for $i \geq 4$ and $k \leq 4$.
\end{prop}
The proof is purely combinatorical and is irrelevant to other parts of the paper. It uses computations and ideas from \cite{Crossley-Whitehouse}.
\begin{proof}
    We argue by monomial degrees in $\mathcal{A}_* = \FF_2[\xi_1,\xi_2,\cdots]$ and the Milnor conjugate formula
    \[
        \zeta_i = \sum_{j=0}^{i-1} \xi_{i-j}^{2^j}\zeta_j.
    \]
    The conjugate formula tells us that the transfer of a monomial (i.e.\ the sum of the monomial and its conjugate) in the $\xi_i$ can only increase its monomial degree. It also tells us that the monomial with minimal monomial degree in $(\xi_i\zeta_i)^k$ is $\xi_i^{2k}$. Therefore, $(\xi_i\zeta_i)^k$ being in the image of transfer can only happen when $\xi_i^{2k}$ appears in the transfer of a monomial, which has smaller monomial degree and the same topological degree.
    
    To streamline the computation, we define that a monomial $P$ has bidegree $(a,b)$ if $P$ has monomial degree $b$ and topological degree $a - b$. In this way, $\xi_i$ has bidegree $(2^i,1)$. To find monomials which have smaller monomial degree and the same topological degree, we can look at the binary expansion of $a$. Let $t$ be a positive integer, and $\mathrm{Bin}(t)$ be the number of $1$s in the binary expansion of $t$. If a monomial has bidegree $(a,b)$, then both $a$ must be even and $Bin(a) \leq b$. Since $\xi_i$ has bidegree $(2^i,1)$, any monomial whose transfer contains $\xi_i^{2k}$ must be in bidegree $(a,b)$ where $a - b = 2k(2^i -1)$.
    
    We will only check the highest power of $\xi_i \zeta_i$ listed in the statement of the proposition, since if $(\xi_i \zeta_i)^k$ is nontrivial in Tate cohomology, then $(\xi_i \zeta_i)^j$ for $j \leq k$ are all nontrivial.
    
    The class of $\xi_1\zeta_1$ is obviously nontrivial in Tate cohomology, so we start our argument with $(\xi_2\zeta_2)^2$. Writing it as a polynomial in the $\xi_i$, the leading term is $\xi_2^4$, which has bidegree $(16,4)$. We only need to check if there is a nontrivial monomial in bidegree $(14,2)$. Since $\mathrm{Bin}(14)  = 3 >  2$, there is no monomial in this degree. Therefore $(\xi_2\zeta_2)^2$ is nontrivial in the Tate cohomology.
    
    Next we consider $(\xi_3\zeta_3)^3$. The leading term is $\xi_3^6$, which has bidegree $(48,6)$. In $(46,4)$ there is only one monomial $\xi_5\xi_3\xi_2\xi_1$. It is direct to check that
    \[
    \xi_5\xi_3\xi_2\xi_1 + \zeta_5\zeta_3\zeta_2\zeta_1 \neq (\xi_3\zeta_3)^3.
    \]
    In $(42,2)$, there is no monomial since $Bin(42) = 3 > 2$.
    
    For $(\xi_i\zeta_i)^4$ where $i > 3$, a similar argument applies. When $i = 4$ there is a monomial in $(126,6)$, namely $\xi_6\xi_5\xi_4\xi_3\xi_2\xi_1$, but it cannot transfer to $(\xi_4\zeta_4)^4$. And there is no monomial with smaller monomial degree with the same topological degree. When $i > 4$ there is simply no suitable monomial below $(2^{i+3},8)$ since $\mathrm{Bin}(2^{i+3} - t)$ for $t = 2,4,6$ are all greater than $8-t$.
\end{proof}

The proof can certainly be generalized. For example, $(\xi_i\zeta_i)^8$ are nontrivial in the Tate cohomology for $i > 11$. However, what we proved is sufficient for our computation. 

Recall that the element $u_{2\sigma}$ in $\pi^{C_4}_{2-2\sigma}H\UZ$ maps to $u_{\sigma}^2$ in Tate cohomology.

\begin{cor}\label{cor:image}
    Under the map
    \[
    a_{\sigma}^{-1}\SliceSS(\BPfour) \rightarrow P^*_{4/2}\D \TateSS(N_1^2 H\FF_2),
    \]
    the classes 
    \[
    N(\ot_i)^j a_{\lambda}^ka_{\sigma}^lu_{2\sigma}^m
    \]
    map to
    \[
    N(\xi_i)^j a_{\lambda}^{k - (2^i - 1)}a_{\sigma}^lu_{\sigma}^{2m},
    \]
    for $j \geq 0$ and $k,l,m \in \ZZ$.
    The image is nontrivial if and only if $(\xi_i\zeta_i)^j$ represents a nontrivial element in $\hat{H}^0(C_2;\mathcal{A}_*)$.
\end{cor}

The slice differentials in \cite{HHR} completely describe $a_{\sigma}^{-1}\SliceSS(\BPfour)$. By understanding its image in $P^*_{4/2}\D \TateSS(N_1^2H\FF_2)$, we can deduce many differentials in the Tate spectral sequence. We prove all differentials in their most natural $RO(C_4)$-degree. They can be translated into the integral degree by invertible $a_{\lambda}$ and $a_{\sigma}$ multiplications. \cref{fig:DTate} presents $P^*_{4/2}\D \TateSS(N_1^2H\FF_2)$ with differentials proved below. For reference, \cref{fig:Tate} presents the original Tate spectral sequence $\TateSS(N_1^2H\FF_2)$ with the same differentials.

\begin{figure}
    \centering
    \includegraphics[scale = 0.5]{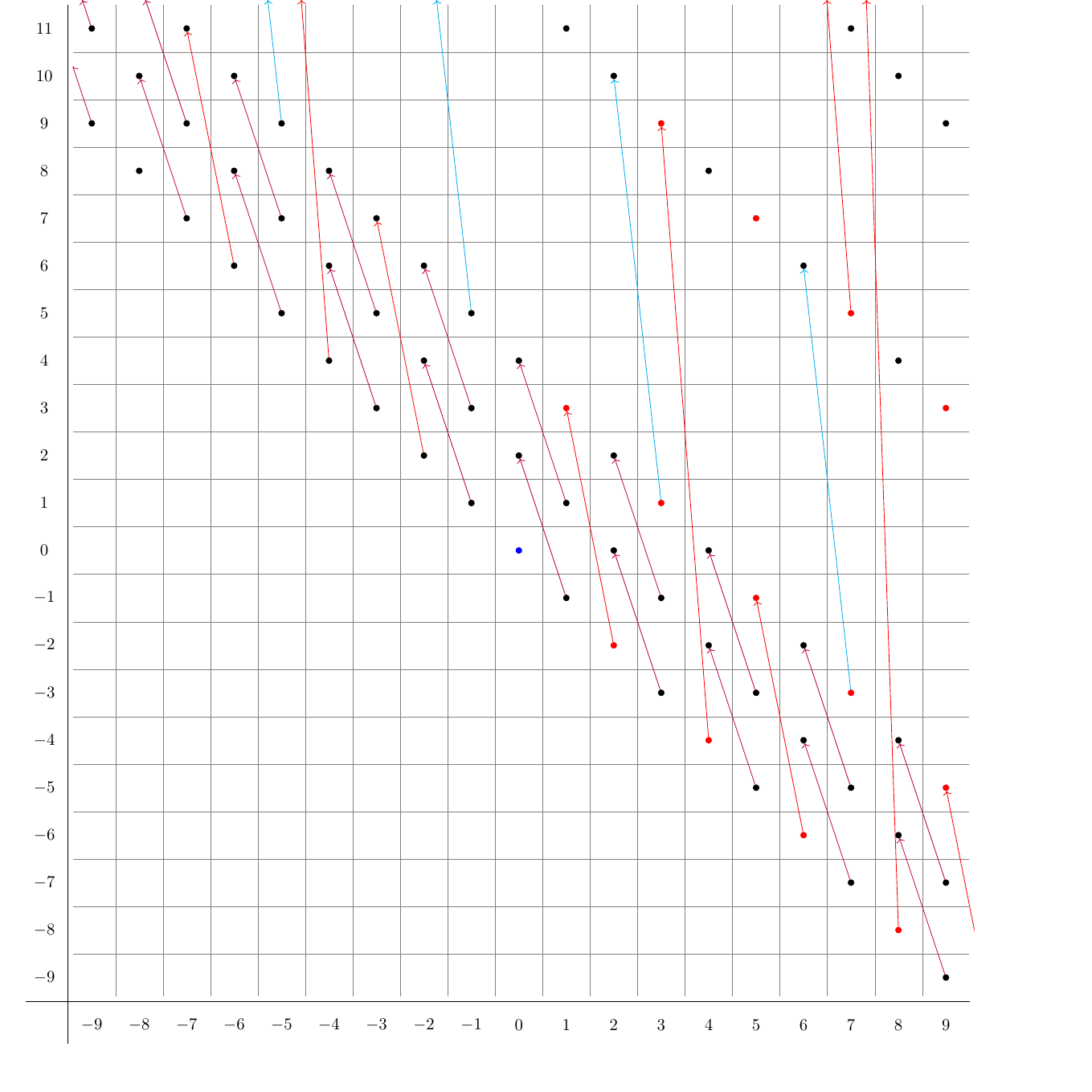}
    \caption{$P^*_{4/2}\D \TateSS(N_1^2H\FF_2)$ with differentials.}
    \label{fig:DTate}
\end{figure}

\begin{figure}
    \centering
    \includegraphics[scale = 0.5]{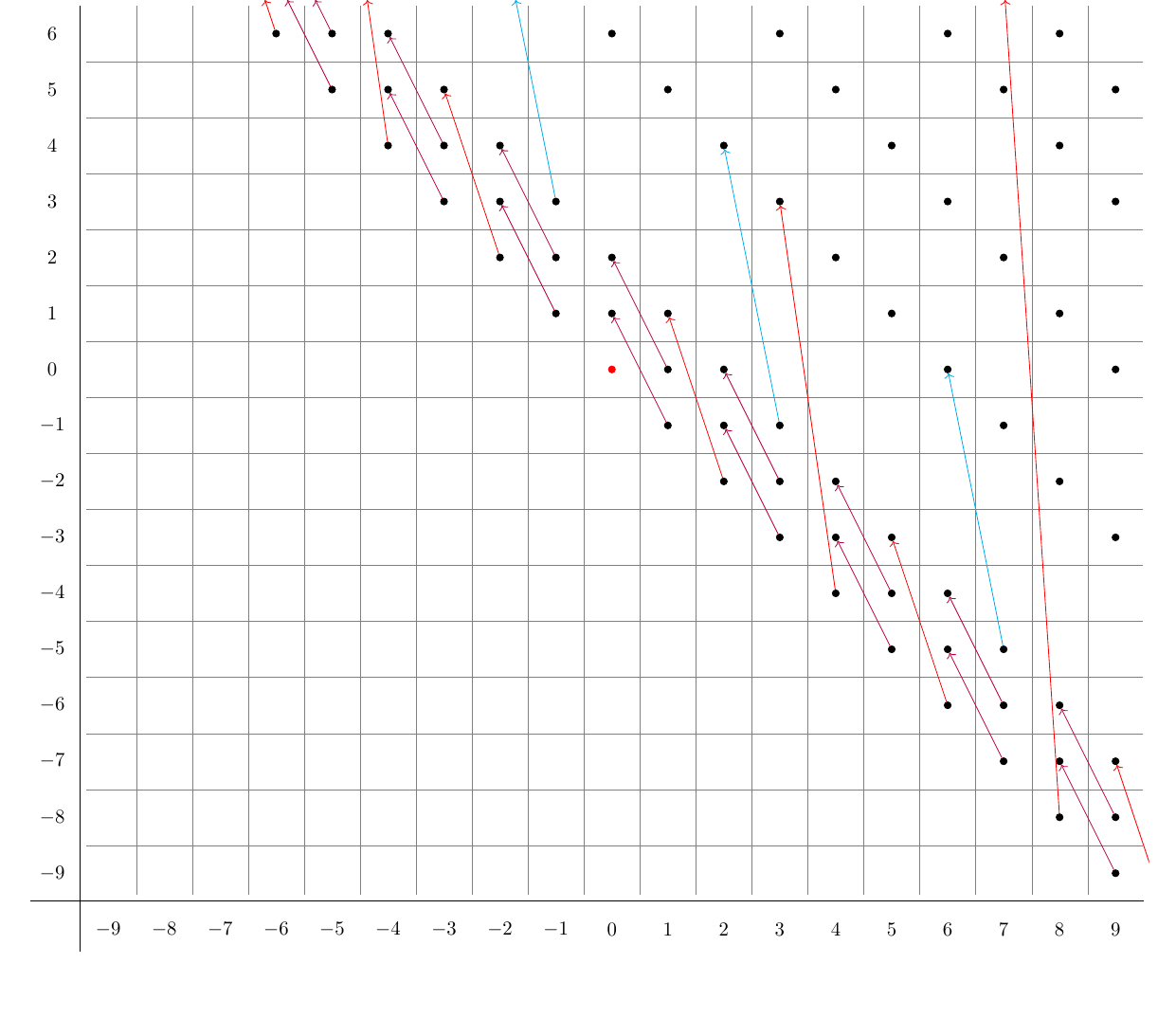}
    \caption{$\TateSS(N_1^2H\FF_2)$ with differentials.}
    \label{fig:Tate}
\end{figure}

\begin{thm}\label{thm:TateDiff}
    In $P^*_{4/2}\D \TateSS(N_1^2 H\FF_2)$, we have differentials
    \[
        d_{2^{k+2}-3}(u_{\sigma}^{2^k}) = N(\xi_k)a_{\sigma}^{2^{k+1} - 1}.
    \]
    for all $k \geq 1$.
\end{thm}

\begin{proof}
    By \cref{prop:power}, both the source and the target are in the image of
    \[
    a_{\sigma}^{-1}\SliceSS(\BPfour) \rightarrow P^*_{4/2}\D \TateSS(N_1^2 H\FF_2).
    \]
    There is no element below $u_{\sigma}^{2^k}$ in the $E_2$-page, so $N(\xi_k)a_{\sigma}^{2^{k+1} - 1}$, which is in stem $2^k - 1 - 2^k \sigma$ and filtration $2^{k+2} - 3$, can only be killed by a differential of length at most $2^{k+2} - 3$. Therefore we only need to show that $N(\xi_k)a_{\sigma}^{2^{k+1} - 1}$ is not killed by a differential of length less than $2^{k+2} - 3$.  To prove this, we show that if such a shorter differential happens, it implies $N(\xi_k)a_{\sigma}^{2^{k+1} - 2} = 0$ in $\pi^{C_2}_{\star}N_1^2 H\FF_2$. Combining with \cref{prop:ti}, it contradicts \cref{prop-cycle}.
    
    Assume that there is a differential $d_l(x) = N(\xi_k)a_{\sigma}^{2^{k+1} - 1}$ in $P^*_{4/2}\D \TateSS(N_1^2 H\FF_2)$ for $l < 2^{k+2} - 3$.
    To show $N(\xi_k)a_{\sigma}^{2^{k+1} - 2} = 0$ in $\pi^{C_2}_{\star}N_1^2 H\FF_2$, we work with the map
    \[
    P^*_{4/2}\D \HFPSS(N_1^2 H\FF_2) \rightarrow P^*_{4/2}\D \TateSS(N_1^2 H\FF_2) 
    \]
    in $RO(C_4)$-grading. By writing an arbitrary element $y$ of stem $m+n\sigma$ in $P^*_{4/2}\D \TateSS(N_1^2 H\FF_2)$ in the form $zu_{\sigma}^fa_{\sigma}^g$ with $z \in H^0(C_2, \pi_{|z|}N_1^2 H\FF_2)$, we see that $y$ is in the image from the spectral sequence $P^*_{4/2}\D \HFPSS(N_1^2 H\FF_2)$ if and only if $g\geq 0$. This happens if and only if the filtration $|z|+g$ is at least $m+n = (|z|+f) +(-f-g) = |z|-g$.  
    
    The class $x$ is of stem $2^k-2^k\sigma$ and filtration $2^{k+2}-3-l$; the filtration is at least $2$ since $l$ must be odd because of the doubling operator $\D$. Thus $xa_{\sigma}^{-1}$ is still in the image from $P^*_{4/2}\D \HFPSS(N_1^2 H\FF_2)$ and likewise is $N(\xi_k)a_{\sigma}^{2^{k+1} - 1}a_{\sigma}^{-1} = N(\xi_k)a_{\sigma}^{2^{k+1} - 2}$. Thus, we have $d_l(xa_{\sigma}^{-1}) = N(\xi_k)a_{\sigma}^{2^{k+1} - 2}$ in $P^*_{4/2}\D \HFPSS(N_1^2 H\FF_2)$, unless the target is killed by a shorter differential. In any case, $N(\xi_k)a_{\sigma}^{2^{k+1} - 2} = 0$ in $\pi^{C_2}_{\star}N_1^2 H\FF_2$.
    \end{proof}

The exact same argument gives the following differentials.

\begin{cor}
    In $P^*_{4/2} \D \TateSS(N_1^2 H\FF_2)$, we have differentials:
    \[
    d_{2^{k+2}-3}(N(\xi_k)^j u_{\sigma}^{2^k}) = N(\xi_k)^{j+1} a_{\sigma}^{2^{k+1}-1}.
    \]
\end{cor}

 These differentials and their propagation are the red differentials in \cref{fig:Tate}. Notice that $ N(\xi_k)^{j+1}$ can be zero in the $E_2$-page of the Tate spectral sequence. For example, $N(\xi_1)^2 = 0$ in the Tate cohomology, since 
\[
\xi_1^2 \zeta_1^2 = \xi_1^4 = \xi_1(\xi_2 + \zeta_2) = Tr(\xi_1\xi_2)
\]
As a result, the target of $d_5(N(\xi_1)u_{\sigma}^2)$ predicted by the corollary is zero. Instead, $N(\xi_1)u_{\sigma}^2$ supports a nontrivial $d_9$, see \cref{prop:exotic_tate_d}

As an interesting consequence, we can bound the length of differentials on elements of the first diagonal in the Tate spectral sequence of $N_1^2 R$ for a large family of ring spectra.

\begin{cor}\label{cor:bound}
    Let $R$ be a non-equivariant $(-1)$-connected homotopy ring spectrum with $\pi_0(X) \cong \ZZ_S$ being a localization of $\ZZ$ such that $\frac12\notin \ZZ_S$. Let $v \in \hat{H}^{2}(C_2;\pi_0N_1^2 X)\cong \hat{H}^{2}(C_2;\ZZ_S)$ be the generator of the Tate cohomology. Then $v^{2^k}$ supports a non-trivial differential of length $l_k$ with 
        \[
        \rho(2^{k+1}) \leq l_k \leq 2^{k+2}-1.
        \]
    Here $\rho(n)$ is the Radon-Hurwitz number (with $\rho(n)-1$ the maximal number of independent vector fields on $S^{n-1}$): for $n= k2^{4b+c}$ with $k$ odd and $0\leq c < 4$, it is defined as $\rho(n) = 8b+2^c$.\footnote{Beware that the lower bound is based on a classical result for which we don't know of a published reference.}
\end{cor}

\begin{proof}
    Consider the sequence of non-equivariant ring maps
    \[
    S^0 \rightarrow R \rightarrow H\FF_2,
    \]
    where the last map is the composition of the $0$-Postnikov section and the mod $2$ map. 
    
    Take the norm $N_1^2$ and take the Tate spectral sequence, we have
    \[
    \TateSS(S^0) \rightarrow \TateSS(N_1^2 R) \rightarrow \TateSS(N_1^2 H\FF_2).
    \]
    Note that $\pi_0N_1^2X \cong \ZZ_S \otimes \ZZ_S\cong \ZZ_S$. Since $v$ maps to $u_{\sigma}^2a_{\sigma}^{-2}$ in $\TateSS(N_1^2 H\FF_2)$, \cref{thm:TateDiff} gives the upper bound. The lower bound is given by the corresponding differential in $\TateSS(S^0)$. Positive powers of $v$ lies in the homotopy orbit spectral sequence part of $\TateSS(S^0)$, and the homotopy orbit spectral sequence can be identified with the Atiyah-Hirzebruch spectral sequence of $\RR P^{\infty}$ with homology theory $\pi_*$. This spectral sequence is the stabilization of the EHP spectral sequence. The element corresponding to $v^{2^k}$ in the EHP spectral sequence supports differentials related to the vector fields of sphere problems: it supports a differential of length $\rho(2^{k+1})$, with target in the image of $J$. (See \cite[Lectures 20 and 21]{MillerVectorFields}.)
\end{proof}

Now we discuss the rest of differentials in the range we are concerned with, which are differentials hitting elements in Figure \ref{fig:Tate} below slope $1$ and in stem $0$ to $8$. In this range, the ring $\hat{H}^0(C_2;\mathcal{A}_*)$ is presented by the following generators:
\begin{align*}
	b_1 &  \textrm{ restricts to } \xi_1 \\
	N(\xi_1)u_{\sigma} & \textrm{ restricts to } \xi_1\zeta_1\\
    N(\xi_2)u_{\sigma}^3 & \textrm{ restricts to }\xi_2\zeta_2\\
\end{align*}
with relations:
\begin{align*}
    b_1^2 + N(\xi_1)u_{\sigma} & = 0\\
    b_1^3 &= 0\\
    b_1 N(\xi_2)u_{\sigma}^3 & = 0\\
\end{align*}

The following differentials present the only remaining differentials in this range.

\begin{prop}\label{prop:exotic_tate_d}
    In $P^*_{4/2}\D \TateSS(N_1^2 H\FF_2)$, we have differentials:
    \[
        d_3(u_{\sigma}) = b_1u_{\sigma}^{-1}a_{\sigma}^2.
    \]
    \[
        d_3(b_1) = N(\xi_1) u_{\sigma}^{-1}a_{\sigma}^2.
    \]
    \[
        d_9(N(\xi_1)u_{\sigma}^2) = N(\xi_2)u_{\sigma}^{-1}a_{\sigma}^5.
    \]
\end{prop}

\begin{proof}
    For the first differential, since $u_{\sigma}^2$ supports a $d_5$ by \cref{thm:TateDiff}, $u_{\sigma}$ must support a shorter differential, and the $d_3$ is the only possibility.
    
    For the second differential, consider the class $N(\xi_1)u_{\sigma}$, which is the class in $(2,2)$ in \cref{fig:Tate}. Its preimage in the homotopy fixed points spectral sequence is the only class in stem $2$ that can support a nontrivial restriction. By \cref{thm:htpy}, $\pi_2$ of the homotopy fixed points is $\ZZ/4$, therefore its generator must support a nontrivial restriction (\cref{prop-cohomological}). Thus $N(\xi_1)u_{\sigma}$ doesn't support a differential in the homotopy fixed points spectral sequence, and it must be killed by a differential in the Tate spectral sequence. The only possible source is $b_1u_{\sigma}^2a_{\sigma}^{-2}$ and the differential is a $d_3$. Since $u_{\sigma}^2$ supports a $d_5$, we obtain $d_3(b_1) = N(\xi_1) u_{\sigma}^{-1}a_{\sigma}^2$ by multiplication by $u_{\sigma}^{-2}a_{\sigma}^2$.
    
    For the last differential, we only need to show that the class at $(6,6)$ in \cref{fig:Tate}, which has the name $N(\xi_2)u_{\sigma}^3$, is a cycle. By the same argument as above, its preimage is the only class in stem $6$ that can support a restriction, and this indeed happens by \cref{thm:htpy}. In the Tate spectral sequence, the only possible differential killing it has the form
    \[
        d_9(N(\xi_1)u_{\sigma}^6a_{\sigma}^{-5}) = N(\xi_2)u_{\sigma}^3.
    \]
    Multiplying both sides by $u_{\sigma}^{-4}a_{\sigma}^5$, we obtain the last differential.
\end{proof}

In \cref{fig:Tate}, the first two differentials and their propagation are colored purple. The last differential is colored blue. 

The computation of the Tate spectral sequence is largely limited by the complexity of the Tate cohomology $\hat{H}^*(C_2;\mathcal{A}_*)$. A better understanding of the Tate cohomology shall allow us to compute most differentials in the Tate spectral sequence via comparison to the localized slice spectral sequence, but can also feed back to the computation of the slice spectral sequence of $\BPfour$.
	
	\bibliography{math}{}
	\bibliographystyle{alpha}
\end{document}